\documentclass[12pt]{amsart}
\usepackage{a4wide,enumerate,color,graphicx}
\usepackage{amsmath}
\usepackage{scrextend} 
\usepackage{amssymb}
\usepackage{esint}
\usepackage{amsfonts}
\usepackage{amsthm}
\usepackage{mathrsfs}
\usepackage{dsfont}
\usepackage{mhchem}     
\newtheorem{lemma}{Lemma}[section]
\newtheorem{prop}[lemma]{Proposition}
\newtheorem{theo}[lemma]{Theorem}

\newtheorem{rem}[lemma]{Remark}
\newtheorem{coro}[lemma]{Corollary}

\DeclareMathOperator{\divv }{div}

\allowdisplaybreaks

\begin{document}

\title[Multicomponent compressible fluids]{Mass transport in multicomponent compressible fluids: Local and global well-posedness in classes of strong solutions for general class-one models}

\author[D. Bothe]{Dieter Bothe}
\address{Mathematische Modellierung und Analysis, Technische Universit\"at Darmstadt, Alarich-Weiss-Str. 10, 64287 Darmstadt, Germany}
\email{bothe@mma.tu-darmstadt.de}

\author[P.-E. Druet]{Pierre-Etienne Druet}
\address{Weierstrass Institute, Mohrenstr. 39, 10117 Berlin, Germany}
\email{pierre-etienne.druet@wias-berlin.de}


\date{\today}


\thanks{This research is supported by the grant DR1117/1-1 of the German Science Foundation}
\maketitle

\begin{abstract} 
We consider a system of partial differential equations describing mass transport in a multicomponent isothermal compressible fluid. The diffusion fluxes obey the Fick-Onsager or Maxwell-Stefan closure approach. Mechanical forces result into one single convective mixture velocity, the barycentric one, which obeys the Navier-Stokes equations. The thermodynamic pressure is defined by the Gibbs-Duhem equation. Chemical potentials and pressure are derived from a thermodynamic potential, the Helmholtz free energy, with a bulk density allowed to be a general convex function of the mass densities of the constituents. 

The resulting PDEs are of mixed parabolic--hyperbolic type. We prove two theoretical results concerning the well-posedness of the model in classes of strong solutions: 1. The solution always exists and is unique for short--times and 2. If the initial data are sufficiently near to an equilibrium solution, the well-posedness is valid on arbitrary large, but finite time intervals. Both results rely on a contraction principle valid for systems of mixed type that behave like the compressible Navier-Stokes equations. The linearised parabolic part of the operator possesses the self map property with respect to some closed ball \emph{in the state space}, while being contractive \emph{in a lower order norm} only. In this paper, we implement these ideas by means of precise \emph{a priori} estimates in spaces of exact regularity.  
\end{abstract}

\section{Mass transport for a multicomponent compressible fluid}

This paper is devoted to the mathematical analysis of general \emph{class-one models} of mass transport in isothermal multicomponent fluids. We are interested in the theoretical issues of unique solvability and continuous dependence (in short: well-posedness) in classes of strong solutions for the underlying PDEs. To start with, we shall expose the model very briefly. An extensive derivation from thermodynamic first principles is to find in \cite{bothedreyer}, or \cite{dreyerguhlkemueller}, \cite{dreyerguhlkemueller19} for the extension to charged constituents. There are naturally alternative modelling approaches: The reader who wishes exploring the model might for instance consult the references in these papers, or the book \cite{giovan}.

\textbf{Model for the bulk.} We consider a molecular mixture of $N\geq 2$ chemical species $\ce{A_1,\ldots,A_N}$ assumed to constitute a fluid phase. 
The convective and diffusive mass transport of these species and their mechanical behaviour are described by the following balance equations:
\begin{alignat}{2}
\label{mass}\partial_t \rho_i + \divv( \rho_i \, v + J^i) & = r_i & &  \text{ for } i = 1,\ldots,N\, ,\\
\label{momentum}\partial_t (\varrho \, v) + \divv( \varrho \, v\otimes v - \mathbb{S}(\nabla v)) + \nabla p & = \sum_{i=1}^N \rho_i \, b^i(x, \, t)  & & \, .
\end{alignat}
The equations \eqref{mass} are the partial mass balances for the partial mass densities $\rho_1,\ldots,\rho_N$ of the species. We shall use the abbreviation $\varrho := \sum_{i=1}^N \rho_i$ for the total mass density. The barycentric velocity of the fluid is called $v$ and the thermodynamic pressure $p$.
In the Navier-Stokes equations \eqref{momentum}, the viscous stress tensor is denoted $\mathbb{S}(\nabla v)$. The vector fields $b^1,\ldots,b^N$ are the external body forces. The \emph{diffusions fluxes} $J^1,\ldots,J^N$, which are defined to be the non-convective part of the mass fluxes, must satisfy by definition the necessary side-condition $\sum_{i=1}^N J^i = 0$. Following the thermodynamic consistent Fick--Onsager closure approach described by \cite{bothedreyer}, \cite{dreyerguhlkemueller19} (older work in \cite{MR59, dGM63}), the diffusion fluxes $J^1,\ldots,J^N$ obey, in the isothermal case,
\begin{align}\label{DIFFUSFLUX}
J^i = - \sum_{j=1}^N M_{i,j}(\rho_1, \ldots, \rho_N) \, (\nabla \mu_j - b^j(x, \, t)) \, \text{ for } i =1,\ldots,N \, .
\end{align}
The \emph{Onsager matrix} $M(\rho_1, \ldots,\rho_N)$ is a symmetric, positive semi-definite $N\times N$ matrix for every $(\rho_1,\ldots,\rho_N) \in \mathbb{R}^N_+$. In all known linear closure approaches this matrix satisfies
\begin{align}\label{CONSTRAINT}
\sum_{i=1}^N M_{i,j}(\rho_1,\ldots,\rho_N) = 0 \quad \text{ for all } (\rho_1,\ldots,\rho_N) \in \mathbb{R}^N_+ \, .
\end{align}
One possibility to compute the special form of $M$ is for instance to invert the Maxwell--Stefan balance equations. For the mathematical treatment of this algebraic system, the reader can consult \cite{giovan}, \cite{bothe11}, \cite{justel13} or \cite{herbergpruess}. Or $M$ is constructed directly in the form $P^T \, M^{0} \, P$, where $M^{0}$ is a given matrix of full rank, and $P$ is a projector guaranteeing that \eqref{CONSTRAINT} is valid. The paper \cite{bothedruetMS} establishes equivalence relations between the Fick--Onsager and the Maxwell--Stefan constitutive approaches, so that we do not need here further specifying the structure of the tensor $M$.

The quantities $\mu_1,\ldots,\mu_N$ are the chemical potentials. The material theory which provides the definition of $\mu$ is based on the assumption that the Helmholtz free energy of the system possesses only a bulk contribution with density $\varrho\psi$. Moreover, this function possesses the special form
\begin{align*}
 \varrho\psi = h(\rho_1, \ldots, \rho_N) \, ,
\end{align*}
where $h: \, \mathcal{D} \subseteq \mathbb{R}^N_+ \rightarrow \mathbb{R}$ is convex and sufficiently smooth in the range of mass densities relevant for the model. For the sake of simplicity, we shall in fact assume that $h$ possesses a smooth convex extension to the entire range of admissible mass densities $\mathbb{R}^N_+ = \{\rho \in \mathbb{R}^N \, : \, \rho_i > 0 \text{ for } i =1,\ldots,N\}$. The \emph{chemical potentials} $\mu_1,\ldots,\mu_N$ of the species are related to the mass densities $\rho_1, \ldots,\rho_N$ via
\begin{align}\label{CHEMPOT}
\mu_i = \partial_{\rho_i}h(\rho_1, \ldots, \rho_N) \, .
\end{align}
In \eqref{momentum}, the thermodynamic pressure has to obey the isothermal Gibbs-Duhem equation
\begin{align}\label{GIBBSDUHEM}
\sum_{i=1}^N \rho_i \, d\mu_i = dp \, 
\end{align}
where $d$ is the total differential. This yields, up to a reference constant, a relationship between $p$ and the variables $\rho_1,\ldots,\rho_N$ which is often called the Euler equation:
\begin{align}\label{GIBBSDUHEMEULER}
p = -h(\rho_1,\ldots,\rho_N) + \sum_{i=1}^N \rho_i \, \mu_i = -h(\rho_1,\ldots,\rho_N) + \sum_{i=1}^N \rho_i\, \partial_{\rho_i}h(\rho_1, \ldots, \rho_N) \, .
\end{align}
For the mathematical theory of this paper, we do not need to assume a special form of the free energy density, but rather formulate general assumptions: The free energy function is asked to be a \emph{Legendre function} on $\mathbb{R}^N_+$ with surjective gradient onto $\mathbb{R}^N$. 
For illustration, let us remark that the choices
\begin{itemize}
\item $h = k_B \, \theta \, \sum_{i=1}^N n_i \, \ln \frac{n_i}{n^{\text{ref}}}$;
\item $h = K\, F\left(\sum_{i=1}^N n_i \, \bar{v}_i^{\text{ref}}\right) + k_B \, \theta \, \sum_{i=1}^N n_i \, \ln \frac{n_i}{n}$;
\item $h = \sum_{i=1}^N K_i \, n_i\, \bar{v}_i^{\text{ref}} \, ((n_i\, \bar{v}_i^{\text{ref}})^{\alpha_i-1} + \ln (n_i\, \bar{v}_i^{\text{ref}})) +  k_B \, \theta \, \sum_{i=1}^N n_i \, \ln \frac{n_i}{n}$;
\end{itemize}
are covered by the results of this paper. In these examples, $\theta > 0$ is the constant absolute temperature, $n_i := \rho_i/m_i$ is the number density ($m_i > 0$ the molecular mass of species $\ce{A_i}$), and $n = \sum_{i=1}^N n_i$ is the total number density. The first example models the free energy for a mixture of ideas gases with a reference value $n^{\text{ref}} > 0$. In the second example, the constants $\bar{v}_i^{\text{ref}}> 0$ are reference volumes introduced in \cite{dreyerguhlkelandstorfer} to explain solvatisation effects in electrolytes, $K >0$ is the compression module of the fluid, and $F$ is a general non-linear convex function related to volume extension. The third example, with constants $K_i > 0$ and $\alpha_i \geq 1$ shows that more complex state equations can as well be included in the setting. For the convenience of the reader, we briefly show in the Appendix, Section \ref{Legendre}, that the examples fit into the abstract framework of our well-posedness theorems. We also remark that stating assumptions directly on the thermodynamic potential $h$ is possible, because it is always possible to find this potential from the knowledge of the chemical potentials or of the state equation of the physical system, as shown in \cite{bothedreyer}, or in the upcoming paper \cite{bothedruetfe}.

Reaction densities $r_i = r_i(\rho_1,\ldots,\rho_N)$ or $r_i = r_i(\mu_1,\ldots,\mu_N)$ for $i = 1,\ldots,N$ will be considered in \eqref{mass} only for the sake of generality. We shall not enter the very interesting details of their modelling. We just note that these functions are likewise subject to the constraint $\sum_{i=1}^N r_i(\rho_1,\ldots,\rho_N) = 0$ for all $(\rho_1,\ldots,\rho_N) \in \mathbb{R}^N_+$, which expresses the conservation of mass by the reactions. As to the stress tensor $\mathbb{S}$, we shall restrict for simplicity to the standard Newtonian form with constant coefficients. The paper, however, provides methods which are sufficient to extend the results to the case of density and composition dependent viscosity coefficients.

\textbf{Boundary conditions.} We investigate the system \eqref{mass}, \eqref{momentum} in a cylindrical domain $Q_T := \Omega \times ]0,T[$ with a bounded domain $\Omega \subset \mathbb{R}^3$ and $T>0$ a finite time. It is possible to treat the case $\Omega \subset \mathbb{R}^d$ for general $d\geq 2$ with exactly the same methods.

We are mainly interested in results for the bulk operators. Thus, we shall not be afraid of some simplification concerning initial and boundary operators. A lot of interesting phenomena like mass transfer at active boundaries, or chemical reactions with surfactants, shall not be considered here but in further publications. Boundary conditions are also often the source of additional problems for the mathematical theory, like: Mixed boundary conditions, non-smooth boundaries, singular initial data. All this can, however, only be dealt with in the context of weak solutions, and is not our object here.

We consider the initial conditions
\begin{alignat}{2}\label{initial0rho}
\rho_i(x,\, 0) & = \rho^0_i(x) & & \text{ for } x \in \Omega, \, i = 1,\ldots, N \, ,\\
\label{initial0v} v_j(x, \, 0) & = v^0_j(x) & & \text{ for } x \in \Omega, \, j = 1,2,3\, .
\end{alignat}
For simplicity, we consider the linear homogeneous boundary conditions
\begin{alignat}{2}
\label{lateral0v} v & = 0 & &\text{ on } S_T := \partial \Omega \times ]0,T[ \, ,\\
\label{lateral0q} \nu \cdot J^i & = 0 & &\text{ on } S_T \text{ for } i = 1,\ldots,N \, .
\end{alignat}

\section{Mathematical analysis: state of the art and our results}\label{mathint}

The local or global existence and uniqueness of strong solutions to the class-one model exposed in the introduction has, from this point of view of generality, not yet been studied. More generally, there are relatively few published investigations with rigorous analysis about mass transport equations for a multicomponent system, with or without chemical reactions, being coupled to equations of Navier-Stokes type. In the theoretical study of this problem two different branches or disciplines of PDE analysis are meeting each other: diffusion--reaction systems and mathematical fluid dynamics.

The first fundamental observation in studying the system is that the differential operator generated by the mass transport equation is not parabolic. This is due to the condition \eqref{CONSTRAINT}, which implies that the second--order spatial differential operator possesses one zero eigenvalue. The total mass density satisfies the continuity equation $\partial_t \varrho + \divv(\varrho \, v) = 0$. One of the coordinates of the vector of unknowns behaves inherently hyperbolic.

One important question is how to deal with this hyperbolic component. Among the papers representing most important advances for the understanding of the field, we can mention \cite{herbergpruess} and \cite{chenjuengel}. The first paper is concerned with local-in-time well-posedness analysis for strong solutions, while the second deals with globally defined weak solutions for the Maxwell-Stefan closure of the diffusion fluxes. Both papers, however, rely on the same fundamental idea to eliminate the hyperbolic component by assuming $\varrho = \text{const.}$ (incompressibility). Under this condition, the Navier-Stokes equations reduce to their incompressible variant and decouple from the mass transport system. This system can be solved independently and re-expressed as a parabolic problem for the mass fractions $\rho_1/\varrho, \ldots,\rho_N/\varrho$ (in \cite{herbergpruess}) or for differences of chemical potentials (in \cite{chenjuengel}). In both cases there remains only $N-1$ independent variables. Let us briefly remark that the Navier-Stokes equations do not occur explicitly in \cite{herbergpruess} but are treated (in addition to other difficulties though) in \cite{bothepruess}. Note that in \cite{bothesoga}, a class of multicomponent mixtures has been introduced for which the use of the incompressible Navier-Stokes equation is more realistic: Incompressibility is assumed for the solvent only, and diffusion is considered against the solvent velocity.

In the context of compressible fluids, the global weak solution analysis of non-isothermal class-one models was initiated in \cite{feipetri08}, where a simplified diffusion law (diagonal, full-rank closure) was considered so that the problem of degenerate parabolicity is avoided. In \cite{mupoza15} and \cite{za15} the fluxes are calculated from a constitutive equation similar to \eqref{DIFFUSFLUX}, though for a special choice of the mobility matrix and of the thermodynamic potential. The global existence of weak solution is tackled by means of diverse stabilisation techniques and a tool called Bresch-Desjardins technique, which exploits a special dependence of viscosity coefficients on density to obtain estimates for the density gradient. 

The first paper dealing with the full class-one model exposed in the introduction for more general thermodynamic potentials $h = h(\rho)$ and closure relations for diffusion fluxes and reaction terms is the investigation \cite{dredrugagu16} and the subsequent \cite{dredrugagu17a,dredrugagu17b,dredrugagu17c} in the context of charged carriers (electrolytes). There ideas of \cite{chenjuengel} were generalised in order to rewrite the PDEs as a coupled system with the following structure:
\begin{enumerate}[(a)]
\item \label{paran} A doubly non-linear parabolic system for $N-1$ variables $q_1, \ldots,q_{N-1}$ called relative chemical potentials (for instance $q_i := \mu_i-\mu_N$ for $i = 1,\ldots,N-1$);
 \item \label{ns}The compressible Navier-Stokes equations with pressure $p = P(\varrho, \, q_1,\ldots, q_{N-1})$ to determine the variables $\varrho$ and $v$.
\end{enumerate}
The concrete form of this system is given by the equations 
\begin{alignat*}{2}
 \partial_t R(\varrho, \, q) + \divv( R(\varrho, \, q) \, v - \widetilde{M}(\varrho, \, q) \, (\nabla q - \tilde{b}) ) & = \tilde{r}(\varrho, \, q) \, ,& & \\
\partial_t \varrho + \divv(\varrho \, v) & = 0 \, ,& & \\
 \partial_t (\varrho \, v) + \divv( \varrho \, v\otimes v - \mathbb{S}(\nabla v)) + \nabla P(\varrho, \, q) & = R(\varrho, \, q) \cdot \tilde{b}(x, \, t) + \varrho \, \bar{b}(x, \, t)  & & 
\end{alignat*}
in which $\widetilde{M} \in \mathbb{R}^{(N-1)\times (N-1)}$ is a \emph{positive} operator, the Jacobian $R_q \in \mathbb{R}^{(N-1)\times (N-1)}$ is likewise positive definite, and $\tilde{b}, \, \bar{b}, \, \tilde{r}$ are suitable transformations of the vector of bulk forces and of the reaction term. This formulation has many advantages, the most obvious one being that it allows to handle the total mass density with Navier-Stokes techniques and eliminates the tedious positivity constraints on the partial mass densities.
Applying these ideas, we were able to prove the global existence of certain weak solutions under the restriction that the non-zero eigenvalues of $M = M(\rho)$ remain strictly positive independently of the possible vanishing of species.

In this paper, we show that the reformulation based on \eqref{paran}, \eqref{ns} is also suited to study the local and global well-posedness for strong solutions, without restriction on the particular form of the free energy (inside of the assumption $h = h(\rho)$ with $h$ \emph{of Legendre type}, a notion to be defined below), and for general $M = M(\rho)$ symmetric and positively semi--definite of rank $N-1$. From the point of view of its structure, the reformulated system of equations consists of the compressible Navier-Stokes equations, coupled to a doubly nonlinear parabolic system of dimension $N-1$ for the unknown $q$. For $N =2$, the equation for $q$ is scalar, and we would face a variant of the so-called Navier-Stokes-Fourier system with $q$ playing the role of the temperature, and a density-dependent diffusion coefficient $\widetilde{M}$. 

The general method used to study these systems in classes of strong or classical solutions is the contraction principle valid for short times or small perturbations of equilibrium solutions (a property sometimes improperly called 'small data'). We have to pay attention to the fact, though, that the parabolic contraction principle does not apply here in its pure form. There have been two types of attempts to study mixed systems like Navier-Stokes and Navier-Stokes-Fourier. The first method consists in passing to Lagrange coordinates, in terms of which there is an explicit inversion formula for the continuity equation. Then, the density is eliminated, and it is possible to study the parabolic part of the system with a nonlocal term. This is the approach exposed for instance in \cite{tani}, \cite{solocompress3} with short-time well-posedness results in the scale of H\"older spaces (see also \cite{solocompress2} for corresponding results without proofs in scale of Hilbert-spaces). The second method sticks to the Eulerian coordinates and it exploits precise estimates to control the growth of the solution. Early results for this approach are to be found in \cite{solocompress} for the Navier-Stokes operator in the Sobolev (non Hilbertian) scale of spaces, and in \cite{matsunishi,valli82,valli83,vallizaja} for the Sobolev-Hilbert scale. Further short comments on this type of literature are given after the statement of the main Theorems.

In our case, for $N > 2$ the parabolic system for $q$ is non-diagonal, but its linearised principal part in smooth points is still parabolic in the sense of Petrovki, normally elliptic in Amanns notation. This is clearly a nontrivial extension of the traditional problems of fluids mechanics. We shall study the problem in the class proposed in the paper \cite{solocompress} for Navier-Stokes: $W^{2,1}_p$ with $p$ larger than the space dimension for the components of the velocity and $W^{1,1}_{p,\infty}$ for the densities. For the new variable $q$ we also choose the parabolic setting of $W^{2,1}_p$. Within these classes we are able to prove the local existence and the semi-flow property for strong solutions. We shall also prove the global existence under the condition that the initial data are sufficiently near to an equilibrium (stationary) solution. Since this result foots on stability estimates in the state space, we however need to assume the higher regularity of the initial data in order to obtain some stability from the continuity equation. Thus, these solutions exist and are unique on arbitrary large time intervals, but they might not enjoy the extension property. 

A further feature worth to mention is that in our treatment, the question of positivity of the mass densities is reduced to obtaining a $L^{\infty}$ estimate for the relative chemical potentials $q_1, \ldots ,q_{N-1}$ and a positivity estimate for the total mass density $\varrho$. This is a consequence of the fact that we recover $\rho_1, \ldots, \rho_N$ in the form of a continuous map $\mathscr{R}(\varrho, \, q)$ with range in $\mathbb{R}^N_+$. The positivity of a solution $\varrho$ to the continuity equation depends only on the smoothness of the velocity field $v$, while the $L^{\infty}$ bound for $q$ is a natural consequence of the choice of the state space. In this way, the question of positivity is entirely reduced to the smoothness issue, and strong solutions remain by definition positive as long as they are bounded in the state space.

At last we would like to mention that, while finishing this investigation, we became aware of the recent work \cite{piashiba18}. Here the authors study the short-time well-posedness for a model similar to the one considered in \cite{mupoza15,za15}, with certain restrictions to some particular choices for the thermodynamic potentials and kinetic matrix. The paper foots on the same change of variables as in \cite{dredrugagu16}, and it uses a reformulation similar to \eqref{paran}, \eqref{ns}. The problem is studied in the $L^pL^q$ parabolic setting by means of the Lagrange coordinate transformation. This approach provides interesting complements to the methods proposed in the present paper, and vice versa.

\subsection{Main results}

We assume that $\Omega \subset \mathbb{R}^3$ is a bounded domain and $T >0$. We denote $Q= Q_T = \Omega \times ]0,T[$.

In order to formulate our results, we first recall a few notations and definitions. At first for $\ell = 1, \, 2, \ldots$ and $1 \leq p \leq +\infty$ we introduce the anisotropic/parabolic Sobolev spaces
\begin{align*}
W^{2\ell, \ell}_p(Q) := & \{ u \in L^p(Q) \, : \, D_t^{\beta}D^{\alpha}_x u \in L^p(Q) \, \forall\,\, 1 \leq 2 \, \beta + |\alpha| \leq 2 \, \ell \} \, ,\\
\|u\|_{W^{2\ell,\ell}_p(Q)} := & \sum_{0\leq 2 \, \beta + |\alpha| \leq 2 \, \ell} \|D_t^{\beta}D^{\alpha}_x u \|_{L^p(Q)}
\end{align*}
and, with a further index $1 \leq r < \infty$, the spaces
\begin{align*}
W^{\ell}_{p,r}(Q) = W^{\ell,\ell}_{p,r}(Q) := & \{u \in L^{p,r}(Q) \, : \,\sum_{0\leq \beta + |\alpha| \leq \ell} D^{\alpha}_x \, D^{\beta}_t u \in L^{p,r}(Q) \} \, ,\\
\|u\|_{W^{\ell,\ell}_{p,r}(Q)} := & \sum_{0\leq \beta + |\alpha| \leq \ell} \|D_t^{\beta}D^{\alpha}_x u \|_{L^{p,r}(Q)} \, .
\end{align*}
Let us precise that in these notations the space integration index always comes first. For $r = + \infty$, $W^{\ell,\ell}_{p,\infty}(Q)$ denotes the closure of $C^{\ell}(\overline{Q})$ with respect to the norm above, and thus
\begin{align*}
W^{\ell,\ell}_{p,\infty}(Q) := & \{u \in L^{p,\infty}(Q) \, : \,\sum_{0\leq \beta + |\alpha| \leq \ell} D^{\alpha}_x \, D^{\beta}_t u \in C([0,T]; \, L^p(\Omega)) \} \, .
\end{align*}
We moreover need the concept of \emph{essential smoothness} for a proper, convex function $f: \, \mathbb{R}^N \rightarrow \mathbb{R}$ (see \cite{rockafellar}, page 251). 
For an essentially smooth, strictly convex function with open domain, the operation of conjugation is identical with applying the classical Legendre transform. We will therefore call $h: \, \mathbb{R}^N_+ \rightarrow \mathbb{R}$ a \emph{Legendre function} if it belongs to $C^1(\mathbb{R}^N_+)$, is strictly convex, and if $|\nabla_{\rho} h(\rho)| \rightarrow +\infty$ for $\rho \rightarrow \partial \mathbb{R}^N_+$. If the function $h$ is moreover \emph{co-finite} (\cite{rockafellar}, page 116), the gradient mapping $\nabla h$ is invertible between the domain of $h$ and the entire space $\mathbb{R}^N$. Typical free energy densities $h$ are co-finite functions of Legendre type as shown in Appendix, Section \ref{Legendre}.

Due to \eqref{CONSTRAINT}, the diffusion system \eqref{mass} is not parabolic. The matrix $M(\rho) \, D^2h(\rho)$ possesses only $N-1$ positive eigenvalues that moreover might degenerate for vanishing species. There are therefore only $N-1$ 'directions of parabolicity' of the mass transport equations. In order to extract them, we shall need the following standard projector in $\mathbb{R}^N$:
\begin{align*}
\mathcal{P}: \mathbb{R}^N \rightarrow \{1^N\}^{\perp} \, , \quad  \mathcal{P} := \text{Id}_{\mathbb{R}^N} - \frac{1}{N} \, 1^N \otimes 1^N \, .
\end{align*}
Let us introduce also
\begin{align*}
 \mathbb{R}^N_+ := & \{\rho = (\rho_1, \ldots, \rho_N) \in \mathbb{R}^N \, : \, \rho_i > 0 \text{ for } i = 1,\ldots,N\} \, ,\\
 \overline{\mathbb{R}}^N_+ := & \{\rho = (\rho_1, \ldots, \rho_N) \in \mathbb{R}^N \, : \, \rho_i \geq 0 \text{ for } i = 1,\ldots,N\} \, .
\end{align*}
Our first main Theorem is devoted to the short-time existence of a strong solution.
\begin{theo}\label{MAIN}
We fix $p > 3$, and we assume that \begin{enumerate}[(a)]
             \item $\Omega \subset \mathbb{R}^3$ is a bounded domain of class $\mathcal{C}^2$;
             
             \item $M: \, \mathbb{R}^N_+ \rightarrow \mathbb{R}^{N\times N}$ is a mapping of class $C^2(\mathbb{R}^{N}_+; \, \mathbb{R}^{N\times N})$ into the positive semi-definite matrices of rank $N-1$ with constant kernel $1^N = \{1, \ldots, 1\}$;

\item $h: \, \mathbb{R}^N_+ \rightarrow \mathbb{R}$ is of class $C^3(\mathbb{R}^{N}_+)$, and is a co-finite function of Legendre type in its domain $\mathbb{R}^N_+$;

\item $r: \, \mathbb{R}^N_+ \rightarrow \mathbb{R}^{N}$ is a mapping of class $C^1(\mathbb{R}^{N}_+)$ into the orthogonal complement of $1^N$;
\item \label{force} The external forcing $b$ satisfies $\mathcal{P} \, b \in W^{1,0}_p(Q_T; \, \mathbb{R}^{N\times 3})$ and $b -  \mathcal{P} \, b \in L^p(Q_T; \, \mathbb{R}^{N\times 3})$. For simplicity, we assume that $\nu(x) \cdot \mathcal{P} \,b(x, \, t) = 0$ for $x \in \partial \Omega$ and $\lambda_1-$almost all $t \in ]0, \, T[$.
\item The initial data $\rho^0_{1}, \ldots \rho^0_{N}: \, \Omega \rightarrow \mathbb{R}_+$ are strictly positive measurable functions satisfying the following conditions:
\begin{itemize}
 \item The initial total mass density $\varrho_0 := \sum_{i=1}^N \rho_{i}^0$ is of class $W^{1,p}(\Omega)$;
 \item There is $m_0 > 0$ such that $ 0 < m_0 \leq \varrho_0(x)$ for all $x \in \Omega$;
\item The vector defined via $\mu^0 := \partial_{\rho}h(\theta, \, \rho^0_{1}, \ldots \rho^0_{N})$ (initial chemical potentials) satisfies $\mathcal{P} \,\mu^0 \in W^{2-\frac{2}{p}}_p(\Omega; \, \mathbb{R}^N)$;
\item The compatibility condition $\nu(x) \cdot \mathcal{P} \nabla\mu^0(x) = 0$ is valid in $W^{1-\frac{3}{p}}_p(\partial \Omega; \, \mathbb{R}^N)$ in the sense of traces;
\end{itemize}
\item The initial velocity $v^0$ belongs to $W^{2-\frac{2}{p}}_{p}(\Omega; \, \mathbb{R}^3)$ with $v^0 = 0$ in $W^{2-\frac{3}{p}}_{p}(\partial \Omega; \, \mathbb{R}^3)$.
\end{enumerate}
Then, there exists $0 < T^* \leq T$ such that the problem \eqref{mass}, \eqref{momentum} with closure relations \eqref{DIFFUSFLUX}, \eqref{CHEMPOT}, \eqref{GIBBSDUHEMEULER} and boundary conditions \eqref{initial0rho}, \eqref{initial0v},  \eqref{lateral0v}, \eqref{lateral0q} possesses a unique solution in the class
\begin{align*}
 \rho \in W^{1}_{p}(Q_{T^*}; \, \mathbb{R}^N_+), \quad v \in W^{2,1}_p(Q_{T^*}; \, \mathbb{R}^3) \, ,
\end{align*}
such that, moreover, $\mu := \partial_{\rho}h(\theta, \, \rho)$ satisfies $ \mathcal{P} \,\mu \in  W^{2,1}_p(Q_{T^*}; \, \mathbb{R}^N)$. The solution can be uniquely extended to a larger time interval whenever there is $\alpha > 0$ such that
\begin{align*}
\|\mathcal{P} \,\mu\|_{C^{\alpha,\frac{\alpha}{2}}(Q_{T^*})} + \|\nabla (\mathcal{P} \,\mu)\|_{L^{\infty,p}(Q_{T^*})} + \|v\|_{L^{z \, p,p}(Q_{T^*})} + \int_{0}^{T^*} [\nabla v(s)]_{C^{\alpha}(\Omega)} \, ds < + \infty \, .
\end{align*}
Here $z = z(p)$ satisfies $z = \frac{3}{p-2}$ for $3 < p < 5$, $z > 1$ arbitrary for $p = 5$, and $z = 1$ for $p > 5$.
\end{theo}
\begin{rem}
\begin{itemize}
\item A solution $(\rho, \ v)$ in the sense of Theorem \ref{MAIN} is strong: The equations \eqref{mass}, \eqref{momentum} are valid pointwise almost everywhere in $Q_{T^*}$. To see this one uses that \eqref{CONSTRAINT} implies the identity
\begin{align*}
 J = - M(\rho) \, \nabla \mu = - M(\rho) \, \mathcal{P} \, \nabla \mu \, .
\end{align*}
Since Theorem \ref{MAIN} establishes parabolic regularity for $ \mathcal{P} \,\mu$, the contributions $\divv J$ are well defined in $L^p(Q_{T^*})$.

\item The result carries over to the case where the potential $h$ has a smaller domain: $h: \, \mathcal{D} \subset \mathbb{R}^N_+ \rightarrow \mathbb{R}$, provided that $\mathcal{D}$ is open, and that $h$ is \emph{of Legendre type} in $\mathcal{D}$\footnote{The concept of a function of Legendre type on an open set is defined \cite{rockafellar}, Th. 26.5. This is more exact then speaking of a Legendre function, even if we shall also employ this terminology if the context is unequivocal.}.
The initial data must satisfy $\rho^0 \in \mathcal{D}$. The maximal existence time is then further restricted by the distance of $\rho^0$ to $\partial \mathcal{D}$. The case that $h$ is not co-finite, which means that the image of $\nabla_{\rho} h$ is a true subset of $\mathbb{R}^N$, corresponds to constraints affecting the chemical potentials, which we do not wish to discuss further here.  
\end{itemize}
\end{rem}
\begin{rem}
Other functional space settings are applicable:
\begin{itemize}
\item The parabolic H\"older-space scale $C^{2+\alpha, \, 1+\frac{\alpha}{2}}$ seems to be very natural. It was applied successfully to the compressible Navier-Stokes equations with energy equation: see \cite{tani} and \cite{solocompress3} for a (unfortunately very short) proof of local well-posedness;  
\item The Hilbert--space scale $W^{2\ell,\ell}_2$ with $\ell$ sufficiently large. In the latter approach one uses the conservation law structure of the system to derive \emph{a priori} bounds for higher derivatives of the solution in $L^2$. For several variants of the method in the case of Navier-Stokes or Navier-Stokes-Fourier, see \cite{matsunishi,valli82,valli83,vallizaja}, \cite{solocompress2}, or also \cite{cho,hoff,bellafei} and references. Usually, somewhat more regularity of the domain and the coefficients is demanded because it is necessary to differentiate the equations several times.
\item In \cite{feinovsun} the Navier-Stokes-Fourier system was also studied in classes of higher square--integrable derivatives. In this case the maximal existence time can be characterised by a weaker criterion. Indeed, the boundedness of the velocity gradient suffices to guarantee that the solution can be extended.  
\end{itemize}
Proving the local well-posedness for the mixture case in these classes should be possible 'under suitable modifications'. The quotation marks hint toward a substantial problem: The principal part of the parabolic system for the variables $q_1,\ldots,q_{N-1}$ is non-diagonal for $N > 2$. This might be an obstacle to simply transferring the results.
\end{rem}
Our second main result concerns global existence under suitable restrictions for the data. Here the concept of an equilibrium solution is first needed. An equilibrium solution for \eqref{mass}, \eqref{momentum} is defined as a vector $(\rho_1^{\text{eq}}, \ldots, \rho_N^{\text{eq}}, \, v_1^{\text{eq}}, \, v_2^{\text{eq}}, \, v_3^{\text{eq}})$ of functions defined in $\Omega$ with
\begin{align*}
 \rho^{\text{eq}} \in W^{1,p}(\Omega; \, \mathbb{R}^N_+), \quad v^{\text{eq}} \in W^{2,p}(\Omega; \, \mathbb{R}^3)
\end{align*}
and the vector field $\mu^{\text{eq}} = \partial_{\rho}h(\theta, \, \rho^{\text{eq}})$ satisfies $ \mathcal{P} \,\mu^{\text{eq}} \in  W^{2,p}(\Omega; \, \mathbb{R}^N)$. For these functions, the relations
\begin{alignat}{2}
\label{massstat} \divv( \rho^{\text{eq}}_i \, v^{\text{eq}} - \sum_{j=1}^N M_{i,j}(\rho^{\text{eq}}) \, (\nabla \mu_j^{\text{eq}} - b^j(x)) & = 0 & &  \text{ for } i = 1,\ldots,N \, ,\\
\label{momentumsstat} \divv( \varrho^{\text{eq}} \, v^{\text{eq}}\otimes v^{\text{eq}} - \mathbb{S}(\nabla v^{\text{eq}})) + \nabla p^{\text{eq}} & = \sum_{i=1}^N \rho_i^{\text{eq}} \, b^i(x)  & & 
\end{alignat}
are valid in $\Omega$. Here we let $p^{\text{eq}} = -h(\theta, \, \rho^{\text{eq}}) + \sum_{i=1}^N \rho_i^{\text{eq}} \, \mu_i^{\text{eq}}$. The boundary conditions are $v^{\text{eq}} = 0$ on $\partial \Omega$ and $\nu(x) \cdot M_{i,j}(\rho^{\text{eq}}) \, (\nabla \mu_j^{\text{eq}} - b^j(x)) = 0$ on $\partial \Omega$. We show that the problem \eqref{mass}, \eqref{momentum} possesses a unique strong solution on arbitrary large, but finite time intervals, given that:
\begin{enumerate}[(i)]
\item The equilibrium solution and the initial data are sufficiently smooth;
\item The distance of the initial data to an equilibrium solution is sufficiently small.
\end{enumerate}
\begin{theo}\label{MAIN3}
We adopt the assumptions of Theorem \ref{MAIN}, but also assume that $b = b(x)$ does not depend on time with $b \in W^{1,p}(\Omega; \, \mathbb{R}^{N\times 3})$ and $r \equiv 0$. In addition, we assume that an equilibrium solution $(\rho^{\text{eq}}, \, v^{\text{eq}}) \in W^{1,p}(\Omega; \, \mathbb{R}^N_+) \times  W^{2,p}(\Omega; \, \mathbb{R}^3)$ is given and that the data possess the additional regularity
\begin{align*}
 \varrho^{\text{eq}}, \, \varrho^0 \in  W^{2,p}(\Omega), \quad  v^{\text{eq}} \in W^{3,p}(\Omega; \, \mathbb{R}^3), \, v^0 \in W^{2,p}(\Omega; \, \mathbb{R}^3)   \, .
\end{align*}
Then, for every $0 < T < + \infty$, there exists $R_1 > 0$, depending on $T$ and on the respective norms of the data, such that if
\begin{align*}
 \|\mathcal{P}\, (\mu^0 - \mu^{\text{eq}})\|_{W^{2-\frac{2}{p}}_p(\Omega; \, \mathbb{R}^N)} + \|\varrho^0 - \varrho^{\text{eq}}\|_{W^{1,p}(\Omega)} +  \|v^0 - v^{\text{eq}}\|_{W^{2-\frac{2}{p}}_p(\Omega; \, \mathbb{R}^3)} \leq R_1
\end{align*}
the problem \eqref{mass}, \eqref{momentum} with closure relations \eqref{DIFFUSFLUX}, \eqref{CHEMPOT}, \eqref{GIBBSDUHEMEULER} and boundary conditions \eqref{initial0rho}, \eqref{initial0v},  \eqref{lateral0v}, \eqref{lateral0q} possesses a unique solution in $Q_T$ in the same class as in Theorem \ref{MAIN}.
\end{theo}
\begin{rem}
 \begin{itemize}
  \item One particular stability issue for the compressible Navier-Stokes equations and for the Navier-Stokes-Fourier system is well studied in the Hilbert space setting in \cite{matsunishi,valli82,valli83,vallizaja}. It is proved there that some 'equilibrium solution' $v^{\text{eq}} = 0$ and $\varrho^{\text{eq}} = \text{const.}$ (and, in the case of Navier-Stokes-Fourier, $\theta^{\text{eq}} =  \text{const.}$) is globally stable. For initial data sufficiently close to this solution, there indeed exists a global strong solution ($T = + \infty$). Extensions of this result to the stability of other stationary solutions are, to the best of our knowledge, not available. 
  \item Robustness estimates on bounded time intervals are to be found in Theorem 1 of \cite{bellafei}, as well as more recent references in the stability discussion. 
  \item The additional regularity required in Theorem \ref{MAIN3} for the data is sufficient, but might be not optimal. Since we are interested in a qualitative result we do not attempt to formulate minimal assumptions in this place.
 \end{itemize}
\end{rem}

\subsection{Organisation of the paper}

Section \ref{changevariables} explains the change of variables in the transport problem, which is the core of our method. The closely related Section \ref{reformulation} reformulates the partial differential equations and the main results in these new variables.

In Section \ref{technique1} we introduce the differential operators to be investigated in the analysis and the Banach spaces in which they are defined. We state the $C^1-$property for these operators and discuss diverse technicalities such as trace properties and the extension of the boundary data.

Section \ref{twomaps} introduces two methods of linearising in order to reformulate the operator equation as a fixed-point problem. The first method freezes both the coefficients and the lower-order terms, and is applied to prove the short-time existence. The second method is somewhat more demanding and relies on linearising the entire lower-order part of the operators around a suitable extension of the initial data. This second method allows to prove stability estimates and is used for the global existence result. 

The technical part is occupied by the remaining sections. Section \ref{ESTI} states the main continuity estimates for the inverse of the principal part of the linearised operators. 

In Section \ref{contiT}, we apply these estimates to show the controlled growth of the solution and the state space estimates for the short-time existence and uniqueness. We prove the convergence of the fixed-point iteration in Section \ref{FixedPointIter}.

%

For the global existence, we prove the main estimate in the section \ref{contiT1}, and the existence of a fixed-point in Section \ref{FixedPointT1}.

\section{A change of variables to tackle the analysis}\label{changevariables}

The system \eqref{mass}, \eqref{momentum} exhibits several features that might restrain the global well-posedness: 1. The diffusion system is coupled at the highest order; 2. This system possesses mixed parabolic--hyperbolic character, with possibly degenerated parabolicity; 3. The mass densities are subject to positivity constraints. We show in the present paper that there is a reformulation of the problem allowing to eliminate the positivity constraints on $\rho$ and to handle the singularity due to $M \, 1^N = 0$. 
A first main idea is to use the chemical potentials as principal variables. With the help of the conjugate convex function $h^*$ to $h$, we invert the relation $\mu_i = \partial_{\rho_i}h(\rho)$ for $i=1,\ldots,N$, which reads
\begin{align}\label{convexanal}
\rho_i = \partial_{\mu_i}h^*(\mu_{1},\ldots,\mu_N) \text{ for } i=1,\ldots,N \, .
\end{align}
If $h$ is of Legendre type on $\mathbb{R}^N_+$ and co-finite, then $h^*$ is strictly convex and smooth on $\mathbb{R}^N$. Thus, the natural domain of $\mu$ is not subject to constraints. The idea to pass to dual variables to avoid the positivity constraints in multicomponent transport problems is not new. It was probably introduced first in the context of the weak solution analysis of semiconductor equations (see a. o. \cite{gagroe}). The method has been generalised in the context of a \emph{boundedness by entropy method}: See, among others, \cite{juengel15,juengel17} to allow the weak solution analysis of full rank parabolic systems. 

In the context of Fick-Onsager or equivalent closure equations for the diffusion fluxes, the PDE system exhibits a rank $N-1$ parabolicity. In the literature, this parabolicity could be exploited by imposing an incompressibility condition which allows to eliminate one variable: See \cite{chenjuengel}, \cite{herbergpruess}, \cite{bothepruess} for this approach. In these cases the free energy is positively homogeneous, and the thermodynamic pressure resulting from \eqref{GIBBSDUHEMEULER} is constant.

In the paper \cite{dredrugagu16}, we first proposed to combine the inversion formula \eqref{convexanal} with a linear transformation in order to eliminate the positivity constraints and to exploit the rank $N-1$ parabolicity, without imposing restriction on the pressure of the physical system. Note that already in the papers \cite{bothedreyer,dreyerguhlkemueller} devoted mainly to modelling, the diffusion problem is partly formulated in variables $\varrho$ (total mass density) and $\mu_1-\mu_N, \ldots,\mu_{N-1}-\mu_N$ (differences of chemical potentials). These models single out one particular species, introducing some asymmetry. For the theoretical investigation we shall therefore rather follow \cite{dredrugagu16} where the choice of the projector is left open. 

We choose a basis $\xi^1,\ldots,\xi^{N-1}, \, \xi^N$ of $\mathbb{R}^N$ such that $\xi^N = 1^N$, and introduce the uniquely determined $\eta^1,\ldots,\eta^N \in \mathbb{R}^N$ such that $\xi^i \cdot \eta^j = \delta^{i}_{j}$ for $i,j = 1,\ldots,N$ (dual basis). We define
\begin{align*}
q_{\ell} := \eta^{\ell} \cdot \mu := \sum_{i=1}^N \eta^{\ell}_i \, \mu_i \text{ for } \ell = 1,\ldots,N-1 \, .
\end{align*}
We call $q_1,\ldots,q_{N-1}$ the \emph{relative chemical potentials}. We can now express
\begin{align*}
\varrho = \sum_{i=1}^N \rho_i & = 1^{N} \cdot \nabla_{\mu}h^*(\mu_1,\ldots,\mu_N) = \sum_{i=1}^N \partial_{\mu_i} h^*(\mu_1,\ldots,\mu_N) \\
& = 1^N \cdot \nabla_{\mu}h^*(\sum_{\ell = 1}^{N-1} q_{\ell} \, \xi^{\ell} + (\mu \cdot \eta^N) \, 1^N) \, .
\end{align*}
This is an algebraic equation of the form $F(\mu \cdot \eta^N, \, q_1, \ldots, q_{N-1}, \, \varrho) = 0$. We notice that
\begin{align*}
\partial_{\mu \cdot \eta^N} F(\mu \cdot \eta^N, \, q_1, \ldots, q_{N-1}, \, \varrho) = D^{2}h^*(\mu) 1^N \cdot 1^N > 0 \, ,
\end{align*}
due to the strict convexity of the conjugate function. Thus, the latter algebraic equation defines the last component $\mu \cdot \eta^N$ implicitly as a differentiable function of $\varrho$ and $q_1, \ldots, q_{N-1}$. We call this function $\mathscr{M}$ and obtain the equivalent formula
\begin{align}\label{MUAVERAGE}
\mu & = \sum_{\ell=1}^{N-1} q_{\ell} \, \xi^{\ell} + \mathscr{M}(\varrho, \, q_1,\ldots,q_{N-1}) \, 1^N \, ,\\
\label{RHONEW}\rho & = \nabla_{\mu}h^*( \sum_{\ell=1}^{N-1} q_{\ell} \, \xi^{\ell} + \mathscr{M}(\varrho, \, q_1,\ldots,q_{N-1}) \, 1^N) \, ,
\end{align}
where only the total mass density $\varrho$ and the relative chemical potentials $q_1,\ldots,q_{N-1}$ occur as free variables. Since the pressure obeys the Euler equation \eqref{GIBBSDUHEMEULER}
\begin{align}\label{PRESSUREDEF}
p = h^*(\mu) = &  h^*( \sum_{\ell=1}^{N-1} q_{\ell} \, \xi^{\ell} + \mathscr{M}(\varrho, \, q_1,\ldots,q_{N-1}) \, 1^N) =: P(\varrho, \, q) \, .
 \end{align}
Certain properties of the functions $\mathscr{M}$ and $P$ for general $h = h(\rho)$ have already been studied in the Section 5 of \cite{dredrugagu16}.
Here we need only the following property:
\begin{lemma}\label{pressurelemma}
 Suppose that $h \in C^{3}(\mathbb{R}^N_+)$ is a \emph{Legendre} function in $\mathbb{R}^N_+$, and the image of the gradient map $\nabla_{\rho} h$ is the entire $\mathbb{R}^N$. Then, the formula \eqref{PRESSUREDEF} defines a function $P$ which belongs to $C^2(\mathbb{R}_{+} \times \mathbb{R}^{N-1})$.
\end{lemma}
\begin{proof}
 Due to the main Theorem 26.5 of \cite{rockafellar} on the Legendre transform, we know that the convex conjugate $h^*$ is differentiable and locally strictly convex on the image $\nabla_{\rho} h(\mathbb{R}^N_+)$ of the gradient mapping. In addition, $\nabla_{\rho} h(\mathbb{R}^N_+) = \text{int}(\text{dom}(h^*))$. By assumption, we thus know that $\text{dom}(h^*) = \mathbb{R}^N$. Since $\nabla h$ and $\nabla h^*$ are inverse to each other, the inverse mapping theorem allows to show that $h^* \in C^3(\mathbb{R}^N)$ if and only if $h \in C^3(\mathbb{R}^N_+)$.
 
 Consider now the function $\mathscr{M}$ introduced in \eqref{MUAVERAGE}. Since it is obtained implicitly from the algebraic relation $1^N \cdot \nabla_{\mu}h^*(\sum_{\ell = 1}^{N-1} q_{\ell} \, \xi^{\ell} + (\mu \cdot \eta^N) \, 1^N) - \varrho = 0$, we obtain for the derivatives the expressions
 \begin{align*}
  \partial_{\varrho} \mathscr{M}(\varrho, \, q) = \frac{1}{D^2h^*1^N \cdot 1^N}, \quad \partial_{q_k}\mathscr{M}(\varrho, \, q) = -\frac{D^2h^*1^N \cdot \xi^k}{D^2h^*1^N \cdot 1^N} \, ,
 \end{align*}
in which the Hessian $D^2h^*$ is evaluated at $\mu = \sum_{\ell=1}^{N-1} q_{\ell} \, \xi^{\ell} + \mathscr{M}(\varrho, \, q_1,\ldots,q_{N-1}) \, 1^N$. We thus see that $\mathscr{M}\in C^2(\mathbb{R}_+ \times \mathbb{R}^{N-1})$. Clearly, the formula \eqref{PRESSUREDEF} implies that $P \in C^2(\mathbb{R}_+ \times \mathbb{R}^{N-1})$.
\end{proof}
In order to deal with the right-hand side (external forcing), we also introduce projections for the field $b$. For $\ell = 1,\ldots,N-1$, we define $\tilde{b}^{\ell}(x, \, t) := \sum_{i=1}^N b^i(x, \, t) \, \eta^{\ell}_i$ and $\bar{b}(x, \, t) := \sum_{i=1}^N b^i(x, \, t) \, \eta^{N}_i$ in order to express $$b^i(x, \, t) := \sum_{\ell=1}^N \tilde{b}^{\ell}(x, \, t) \, \xi^{\ell}_i + \bar{b}(x, \, t) \text{ for } i =1,\ldots,N \, .$$
For the reaction term $r: \, R^N_+ \rightarrow \mathbb{R}^N$, $\rho \mapsto r(\rho)$, we define
\begin{align*}
 \tilde{r}_k(\varrho, \, q) := \sum_{i=1}^N \xi^k_i \, r_i( \sum_{k=1}^{N-1} R_k(\varrho, \, q) \, \eta^k + \varrho \, \eta^N) \, \text{ for } k = 1,\ldots,N-1 \, .
\end{align*}

\section{Reformulation of the partial differential equations and of the main theorem}\label{reformulation}

We recall \eqref{CONSTRAINT} and we see that the diffusion fluxes have the form
\begin{align*}
& J^i = - \sum_{j=1}^N M_{i,j}(\rho_1,\ldots,\rho_N) \, (\nabla \mu_j - b^j(x, \, t))\\
&=  -\sum_{j=1}^N\left[ \sum_{\ell = 1}^{N-1}  M_{i,j}(\rho_1,\ldots,\rho_N) \, \xi^{\ell}_j \, (\nabla q_{\ell} - \tilde{b}^{\ell}) - M_{i,j}(\rho_1,\ldots,\rho_N) \, (\nabla \mathscr{M}(\varrho, \, q) - \bar{b}(x, \, t))\right] \\
& \quad = - \sum_{\ell = 1}^{N-1}  \left[\sum_{j=1}^N M_{i,j}(\rho_1,\ldots,\rho_N) \, \xi^{\ell}_j\right] \, (\nabla q_{\ell}- \tilde{b}^{\ell}) \, .
 \end{align*}
If we introduce the rectangular projection matrix $\mathcal{Q}_{j,\ell} = \xi^{\ell}_j$ for $\ell = 1,\ldots,N-1$ and $j = 1,\ldots,N$, then $J = - M \, \mathcal{Q} (\nabla q-\tilde{b})$. Thus, we consider equivalently
\begin{alignat*}{2}
\partial_t \rho + \divv( \rho \, v - M \, \mathcal{Q}\, (\nabla q-\tilde{b}(x, \,t))) & = r \, , & & \\
\partial_t (\varrho \, v) + \divv( \varrho \, v\otimes v - \mathbb{S}(\nabla v)) + \nabla P(\varrho, \, q) & = \sum_{i=1}^N \rho_i \, b^i(x, \, t)  & &  \, .
\end{alignat*}
Next we define, for $k = 1,\ldots, N-1$, the maps
\begin{align}\label{RHONEWPROJ}
R_{k}(\varrho, \, q) & := \sum_{j=1}^N \xi^{k}_j \, \rho_j = \sum_{j=1}^N \xi^{k}_j \, \partial_{\mu_j}h^*( \sum_{\ell=1}^{N-1} q_{\ell} \, \xi^{\ell} + \mathscr{M}(\varrho, \, q_1,\ldots,q_{N-1}) \, 1^N  )\, .
\end{align}
Obviously we can express $\rho_i := \sum_{k=1}^{N-1} R_k(\varrho, \, q) \, \eta^k_i + \varrho \, \eta^N_i$.
We note a particular property of the vector field $R$.
\begin{lemma}\label{rhonewlemma}
 Suppose that $h \in C^{3}(\mathbb{R}^N_+)$ is a co-finite Legendre function on $\mathbb{R}^N_+$. Then, the formula \eqref{RHONEWPROJ} defines $R$ as a vector field of class $C([0, \, +\infty[ \times \mathbb{R}^{N-1}; \, \mathbb{R}^{N-1})$ and $C^2(\mathbb{R}_{+} \times \mathbb{R}^{N-1}; \, \mathbb{R}^{N-1})$. The Jacobian $\{R_{k,q_j}\}_{k,j=1,\ldots,N-1}$ is symmetric and positively definite at every $(\varrho, \, q) \in \mathbb{R}_{+} \times \mathbb{R}^{N-1}$ and
\begin{align*}
R_{q}(\varrho, \, q) = \mathcal{Q}^T \, D^{2}h^* \, \mathcal{Q} - \frac{\mathcal{Q}^T \, D^2h^* 1^N \otimes \mathcal{Q}^T \, D^2h^* 1^N}{D^2h^* 1^N \cdot 1^N} \, .
\end{align*}
In this formula, the Hessian $D^2h^*$ is evaluated at $\mu = \sum_{\ell=1}^{N-1} q_{\ell} \, \xi^{\ell} + \mathscr{M}(\varrho, \, q_1,\ldots,q_{N-1})$.
\end{lemma}
The proof is direct, using Corollary 5.3 of \cite{dredrugagu16}. Multiplying the mass transport equations with $\xi^{k}_i$, we obtain that
\begin{align*}
\partial_t R_k(\varrho, \, q) +\divv(R_k(\varrho, \, q) \, v - \underbrace{(\mathcal{Q}^T \, M(\rho) \, \mathcal{Q})_{k,\ell}}_{=:\widetilde{M}_{k,\ell}(\rho)}\,  (\nabla q_{\ell} - \tilde{b}^{\ell}) = (\mathcal{Q}^T \, r)_k \text{ for } k=1,\ldots,N-1 \, .
\end{align*}
It turns out that if the rank of $M(\rho)$ is $N-1$ on all states $\rho \in \mathbb{R}^{N}_+$, the matrix $\widetilde{M}(\rho)$ is symmetric and strictly positively definite on all states $\rho \in \mathbb{R}^{N}_+$. Making use of \eqref{MUAVERAGE}, \eqref{RHONEW}, we can also consider $\widetilde{M}$ as a mapping of the variables $\varrho$ and $q$. Using Lemma \ref{rhonewlemma}, we can establish the following properties of this map.
\begin{lemma}\label{Mnewlemma}
 Suppose that $h \in C^{3}(\mathbb{R}^N_+)$ is a co-finite Legendre function on $\mathbb{R}^N_+$. Suppose further that $M: \, \mathbb{R}^N_+ \rightarrow \mathbb{R}^{N\times N}$ is a mapping into the positively semi-definite matrices of rank $N-1$ with kernel $\{1^N\}$, having entries $M_{i,j}$ of class $C^{2}(\mathbb{R}^N_+) \cap C(\overline{\mathbb{R}}^N_{+})$. Then the formula $ \widetilde{M}(\varrho, \, q) := \mathcal{Q}^T \, M(\rho) \, \mathcal{Q}$ defines a map $\widetilde{M}: \, \mathbb{R}_+ \times \mathbb{R}^{N-1} \rightarrow \mathbb{R}^{(N-1)\times (N-1)}$ into the symmetric positively definite matrices. The entries $\widetilde{M}_{k,j}$ are functions of class $C^{2}(]0, \, +\infty[ \times \mathbb{R}^{N-1})$ and $C([0, \, +\infty[\times \mathbb{R}^{N-1})$. 
\end{lemma}
Overall, we get for the variables $(\varrho, \, q_1, \ldots, q_{N-1}, \, v)$ instead of \eqref{mass}, \eqref{momentum} the equivalent equations
\begin{alignat}{2}
\label{mass2} \partial_t R(\varrho, \, q) + \divv( R(\varrho, \, q) \, v - \widetilde{M}(\varrho, \, q) \, (\nabla q - \tilde{b}(x, \, t)) ) & = \tilde{r}(\varrho, \, q)  \, ,& & \\
\label{mass2tot}\partial_t \varrho + \divv(\varrho \, v) & = 0 \, ,& & \\
\label{momentum2} \partial_t (\varrho \, v) + \divv( \varrho \, v\otimes v - \mathbb{S}(\nabla v)) + \nabla P(\varrho, \, q) & = R(\varrho, \, q) \cdot \tilde{b}(x, \, t) + \varrho \, \bar{b}(x, \, t)  & & \, .
\end{alignat}
Up to the positivity constraint on the total mass density $\varrho$, the latter problem is free of constraints! 

Our first aim is now to show that at least locally--in--time the system \eqref{mass2}, \eqref{mass2tot}, \eqref{momentum2} for the variables $(\varrho, \, q_1, \ldots, q_{N-1}, \, v)$ is well--posed. We consider initial conditions
\begin{alignat}{2}\label{initialq}
q(x, \, 0) & = q_0(x) & & \text{ for } x \in \Omega\, ,\\
\label{initialrho}\varrho(x, \, 0) & = \varrho_0(x) & & \text{ for } x \in \Omega \, ,\\
\label{initialv} v(x, \, 0) & = v_0(x) & & \text{ for } x \in \Omega \, .
\end{alignat}
Due to the preliminary considerations in Section \ref{changevariables}, prescribing these variables is completely equivalent to prescribing initial values for the mass densities $\rho_i$ and the velocity $v$. It suffices to define $\mu^0 = \partial_{\rho}h(\rho^0)$ and then $q^0_k = \mu^0\cdot \eta^k$ for $k=1,\ldots,N-1$. For simplicity, we consider the linear homogeneous boundary conditions
\begin{alignat}{2}
\label{lateralv} v & = 0 & &\text{ on } S_T \, ,\\
\label{lateralq} \nu \cdot \nabla q_{k} & = 0 & &\text{ on } S_T \text{ for } k = 1,\ldots,N-1 \, .
\end{alignat}
The conditions \eqref{lateralq} and \eqref{lateral0q} are equivalent, because we assume throughout the paper that the given forcing $b$ satisfies $\nu(x) \cdot \mathcal{P} \,b(x, \, t) = 0$ for $x \in \partial \Omega$ (see the assumption \eqref{force} in the statement of Theorem \ref{MAIN}). We can also do without this assumption, but at the price of further technical complications -- to be avoided here -- due to the need of conceptualising also surface source terms. Owing to the Lemmas \ref{pressurelemma}, \ref{rhonewlemma} and \ref{Mnewlemma}, the coefficient functions $R$, $\widetilde{M}$ and $P$ are of class $C^2$ in the domain of definitions $\mathbb{R}_+ \times \mathbb{R}^{N-1}$. The set $\Omega$ is assumed smooth likewise (further precisions in the statement of the theorem). We reformulate Theorem \ref{MAIN} for the new variables.
\begin{theo}\label{MAIN2}
Assume that the coefficient functions $R$, $\widetilde{M}$ and $P$ are of class $C^2$, while $\tilde{r}$ is of class $C^1$, in the domain of definition $\mathbb{R}_+ \times \mathbb{R}^{N-1}$. Let $\Omega$ be a bounded domain with boundary $\partial \Omega$ of class $\mathcal{C}^2$. Suppose that, for some $p > 3$, the initial data are of class
\begin{align*}
q^0 \in W^{2-\frac{2}{p}}_p(\Omega; \, \mathbb{R}^{N-1}), \, \varrho_0 \in W^{1,p}(\Omega; \, \mathbb{R}_+), \, v^0 \in W^{2-\frac{2}{p}}_p(\Omega; \, \mathbb{R}^{3}) \, ,
\end{align*}
satisfying $\varrho^0(x) \geq m_0 > 0$ in $\Omega$ and the compatibility conditions $\nu(x) \cdot \nabla q^0(x) = 0$ and $v^0(x) = 0$ on $\partial \Omega$.
Assume that $\tilde{b} \in W^{1,0}_p(Q_T; \, \mathbb{R}^{(N-1)\times 3})$ and $\bar{b} \in L^p(Q_T; \,\mathbb{R}^3)$.

Then there is $0< T^* \leq T$, depending only of these data in the norms just specified, such that the problem \eqref{mass2}, \eqref{mass2tot}, \eqref{momentum2} with boundary conditions \eqref{initialq}, \eqref{initialrho}, \eqref{initialv}, \eqref{lateralv} and \eqref{lateralq} is uniquely solvable in the class
\begin{align*}
(q, \, \varrho,\,v) \in W^{2,1}_p(Q_{T^*}; \, \mathbb{R}^{N-1}) \times W^{1,1}_{p,\infty}(Q_{T^*}; \, \mathbb{R}_+) \times W^{2,1}_p(Q_{T^*}; \, \mathbb{R}^{3})  \, .
\end{align*}
The solution can be uniquely extended in this class to a larger time interval whenever there is $\alpha > 0$ such that $$\|q\|_{C^{\alpha,\frac{\alpha}{2}}(Q_{T^*})} + \|\nabla q\|_{L^{\infty,p}(Q_{T^*})} + \|v\|_{L^{z \, p,p}(Q_{T^*})} + \int_{0}^{T^*} [\nabla v(s)]_{C^{\alpha}(\Omega)} \, ds < + \infty \, ,$$ where $z = z(p)$ is the number defined in Theorem \ref{MAIN}.
\end{theo}

\section{Technicalities}\label{technique1}

\subsection{Operator equation}

For functions $q_1,\ldots, q_{N-1}$, $v_1, \, v_2, \, v_3$ and for non-negative $\varrho$ defined on $\overline{\Omega} \times [0, \, T]$, we define
$\mathscr{A}(q, \, \varrho, \, v) = (\mathscr{A}^1(q, \, \varrho, \, v) , \, \mathscr{A}^2(\varrho, \, v) , \, \mathscr{A}^3(q, \, \varrho, \, v) )$, where
\begin{align*}
\mathscr{A}^1(q, \, \varrho, \, v) & := \partial_t R(\varrho, \, q) + \divv( R(\varrho, \, q) \, v - \widetilde{M}(\varrho, \, q) \, (\nabla q - \tilde{b}(x, \, t)) ) - \tilde{r}(\varrho, \, q) \\
\mathscr{A}^2(\varrho, \, v) & := \partial_t \varrho + \divv(\varrho \, v)\\
\mathscr{A}^3(q, \, \varrho, \, v) & := \varrho \, (\partial_t v + (v\cdot\nabla) v) - \divv \mathbb{S}(\nabla v) + \nabla P(\varrho, \, q) - R(\varrho, \, q) \cdot \tilde{b}(x, \,t) - \varrho \, \bar{b}(x, \,t) \, .
\end{align*}
We shall moreover introduce another related operator. This trick allows to deal with the time derivative of $\varrho$ occurring in $\mathscr{A}^1$, which is a coupling in the highest order. Consider a solution $u = (q, \, \varrho, \, v)$ to $\mathscr{A}(u) = 0$. Computing time derivatives in the equation $\mathscr{A}^1(u) = 0$, we obtain that
\begin{align*}
& R_{\varrho} \, (\partial_t \varrho + v \cdot \nabla \varrho) + \sum_{j=1}^{N-1} R_{q_j} \, (\partial_t q_j + v \cdot \nabla q_j) + R \, \divv  v- \divv (\widetilde{M} \, \nabla q) \\
 & \quad = - \divv(\widetilde{M} \, \tilde{b}(x, \, t)) + \tilde{r}  \, .
\end{align*}
Here, all non-linear functions $R, \, R_{\varrho}, \, R_{q}$ and $\widetilde{M}$, $\tilde{r}$ etc. are evaluated at $(\varrho, \, q)$. We next exploit $\mathscr{A}^2(\varrho, \, v) = 0$ to see that $\partial_t \varrho + v \cdot \nabla \varrho = - \varrho \, \divv v$. Thus, under the side-condition $\mathscr{A}^2(\varrho, \, v) = 0$, the equation $\mathscr{A}^1(u) = 0$ is equivalent to 
\begin{align}\begin{split}\label{A1equiv}
&  R_{q}(\varrho, \, q) \, \partial_t q - \divv (\widetilde{M}(\varrho, \, q) \, \nabla q) \\
& \quad  = (R_{\varrho}(\varrho, \, q) \, \varrho - R(\varrho, \, q)) \, \divv v - R_q(\varrho, \, q) \, v \cdot \nabla q  - \divv(\widetilde{M} \, \tilde{b}(x, \, t))+ \tilde{r}(\varrho, \, q)   \, .
\end{split}
\end{align}
We therefore can introduce $\widetilde{\mathscr{A}}(q, \, \varrho, \, v) := (\widetilde{\mathscr{A}}^1(q, \, \varrho, \, v) , \, \mathscr{A}^2(\varrho, \, v) , \, \mathscr{A}^3(q, \, \varrho, \, v) )$, the first component being the differential operator defined by \eqref{A1equiv}. Clearly, $\mathscr{A}(u) = 0$ if and only if $\widetilde{\mathscr{A}}(u) = 0$.

\subsection{Functional setting}

We now introduce a functional setting for which the short--time well--posedness can be proved by relatively elementary means. We essentially follow the parabolic setting of the book \cite{ladu}, which relies on the former study \cite{sologeneral}. We use the standard Sobolev spaces $W^{m,p}(\Omega)$ for $m \in \mathbb{N}$ and $1\leq p\leq +\infty$, the Sobolev-Slobodecki spaces $W^s_p(\Omega)$ for $s >0$ non-integer and, with a further index $1 \leq r \leq +\infty$, the parabolic Lebesgue spaces $L^{p,r}(Q)$ (space index first; $L^p(Q) = L^{p,p}(Q)$).

First, we consider the setting for the parabolic variables $v$ and $q$. For $\ell = 1,2, \ldots$, the Banach-spaces $W^{2\ell, \ell}_p(Q)$ are defined in Section \ref{mathint}. 
For $\ell = 1$, the space $W^{2,1}_p(Q)$ denotes the usual space $ W^1_p(0,T; \, L^p(\Omega)) \cap L^p(0,T; \, W^{2,p}(\Omega))$ of maximal parabolic regularity of index $p$. Moreover, we let $W^{\ell,0}_p(Q_T) := \{u \in L^p(Q) \, : \, D_x^{\alpha} u \in L^p(Q) \, \forall\,\, |\alpha| \leq \ell   \}$. We denote $C(\overline{Q}) = C^{0,0}(\overline{Q})$ the space of continuous functions over $\overline{Q}$ and, for $\alpha, \, \beta \in [0, \, 1]$, we define the spaces of H\"older continuous functions via
\begin{align*}
 C^{\alpha, \, \beta}(\overline{Q}) := & \{ u \in C(\overline{Q}) \, : \, [ u ]_{C^{\alpha,\beta}(\overline{Q})} < + \infty\} \, ,\\
 [u]_{ C^{\alpha, \, \beta}(\overline{Q})} = & \sup_{t \in [0, \, T], \, x,y \in \Omega} \frac{|u(t, \, x) - u(t, \, y)|}{|x-y|^{\alpha}} + \sup_{x \in \Omega, \, t,s \in [0,\, T]} \frac{|u(t, \, x) - u(s, \, x)|}{|t-s|^{\beta}} \, .
\end{align*} 
\begin{rem}[\textbf{Useful properties of $W^{2,1}_p(Q)$}:]\label{parabolicspace}
\begin{itemize}
\item The spatial differentiation is continuous from $W^{2,1}_p$ into $W^{1,0}_p$, and into $C([0, \, T]; \, W^{1-\frac{2}{p}}_p(\Omega))$;
\item The spatial differentiation is continuous from $W^{2,1}_p$ into $L^{\infty,2p-3}(Q)$, into $L^{z_1 \, p,\infty}(Q)$ and into $L^s(Q)$ for $s = 2p-3 + z_1 \, p$. Here $z_1 = z_1(p) := \frac{3}{5-p}$ for $3 < p < 5$, $z_1 \in ]1, \, \infty[$ arbitrary for $p = 5$, and $z_1 := + \infty$ for $p>5$;
\item For $k \in \mathbb{N}$ and $\alpha \in [0, \, 1]$ such that $k+\alpha \leq 2-\frac{5}{p}$, the space $W^{2,1}_p$ embeds continuously into the H\"older space $C^{k+\alpha,\, 0}(\overline{Q})$, and its elements are bounded;
\item The time differentiation is continuous from $W^{2,1}_p$ into $L^p$;
\end{itemize}
\end{rem}
\begin{proof}
The embedding $W^{2,1}_p(Q_T) \subset C([0, \, T]; \, W^{2-\frac{2}{p}}_p(\Omega))$ is known from the references \cite{sologeneral}, \cite{denkhieberpruess} and several others. Thus, $\frac{d}{dx}$ is a linear continuous operator from $W^{2,1}_p(Q)$ into $C([0, \, T]; \, W^{1-\frac{2}{p}}_p(\Omega))$. With the Sobolev embedding theorem (e. g., 8.3.3 in \cite{kuf} or XI.2.1 in \cite{visi}), we know that $W^{1-\frac{2}{p}}_p(\Omega) \subset L^{\frac{3p}{(5-p)^+}}(\Omega)$. Thus $\frac{d}{dx}$ is continuous into $C([0,T]; \, L^{\frac{3p}{(5-p)^+}}(\Omega))$. For $\alpha := \frac{p}{2p-3}$, the interpolation inequality (see \cite{nirenberginterpo}, Theorem 1)
\begin{align*}
\|\nabla f\|_{L^{\infty}(\Omega)} \leq & C_1 \, \|D^2f\|_{L^p(\Omega)}^{\alpha} \, \|f\|_{L^{\infty}(\Omega)}^{1-\alpha} + C_2 \, \|f\|_{L^{\infty}(\Omega)} \, 
\end{align*}
implies that
\begin{align*}
\|\nabla u\|_{L^{\infty,2p-3}(Q_T)}^{2p-3} \leq 2^{2p-3} \, (C_1^{2p-3} \,  \|D^2u\|_{L^p(Q_T)}^{p} \, \|u\|^{p-3}_{L^{\infty}(Q_T)} + C_2^{2p-3} \, \|u\|_{L^{\infty,2p-3}(Q_T)}^{2p-3}) \, .
\end{align*}
Thus $\nabla u \in L^{\infty,2p-3}(Q)$. The continuity of $\frac{d}{dx}$ into $W^{1,0}_p$ is obvious.
For $k \in \mathbb{N}$ and $\alpha \in [0, \, 1]$ such that $k+\alpha \leq 2-\frac{5}{p}$, the space $W^{2 - \frac{2}{p}}_p(\Omega)$ embeds continuously into the H\"older space $C^{k+\alpha}(\overline{\Omega})$ (see \cite{visi}, XI.2.1). Thus $W^{2,1}_p(Q)$ embeds continuously into the H\"older space $C^{k+\alpha,\, 0}(\overline{Q})$.
\end{proof}
Next, we consider the appropriate functional space setting for the continuity equation. Since this equation has another type, some asymmetry cannot be avoided. We introduce the space
\begin{align*}
W^{1,1}_{p,\infty}(Q) & := \{u \in L^{p,\infty}(Q) \, : \, u_t, \, u_{x_i} \in C([0,T]; \, L^{p}(\Omega)) \text{ for } i = 1,2,3\} \, ,\\
\|u\|_{W^{1,1}_{p,\infty}(Q)} & := \|u\|_{L^{p,\infty}(Q)} + \|u_x\|_{L^{p,\infty}(Q)} + \|u_t\|_{L^{p,\infty}(Q)} \, .
\end{align*}
\begin{rem}[\textbf{Properties of $W^{1,1}_{p,\infty}(Q)$}]\label{contispace}
\begin{itemize}
\item The space $W^{1,1}_{p,\infty}$ embeds continuously into the isotropic H\"older space $C^{1-\frac{3}{p}}(\overline{Q})$, and its elements are bounded;
\item The spatial differentiation is continuous from $W^{1,1}_{p,\infty}$ into $L^{p,\infty}(Q)$;
\item The time differentiation is continuous from $W^{1,1}_{p,\infty}$ into $L^{p,\infty}(Q)$;
\end{itemize}
\end{rem}
\begin{proof}
While the two last properties are obvious, we can deduce the first one from the anisotropic embedding result in the appendix of \cite{krejcipanizzi}.
\end{proof}
Beside the diverse Sobolev embedding results, we shall use for $p > 3$ the interpolation inequality (see \cite{nirenberginterpo}, Theorem 1)
\begin{align}
\label{gagliardo}\|\nabla f\|_{L^{\infty}(\Omega)} \leq & C_1 \, \|D^2f\|_{L^p(\Omega)}^{\alpha} \, \|f\|_{L^p(\Omega)}^{1-\alpha} + C_2 \, \|f\|_{L^p(\Omega)}
\end{align}
valid with $\alpha := \frac{1}{2}+\frac{3}{2p}$ for any function $f$ in $W^{2,p}(\Omega)$. 

We consider the operator $(q, \, \varrho,\, v) \mapsto \mathscr{A}(q, \, \varrho,\, v)$ as acting in the product space 
\begin{align}\label{STATESPACE}
\mathcal{X}_T := W^{2,1}_p(Q_T; \, \mathbb{R}^{N-1}) \times W^{1,1}_{p,\infty}(Q_T) \times W^{2,1}_p(Q_T; \, \mathbb{R}^3) \, .
\end{align}
Since the coefficients of $\mathscr{A}$ are defined only for positive $\varrho$, the domain of the operator is contained in the subset of strictly positive second argument
\begin{align}\label{STATESPACEPOS}
 \mathcal{X}_{T,+} := W^{2,1}_p(Q_T; \, \mathbb{R}^{N-1}) \times W^{1,1}_{p,\infty}(Q_T; \, \mathbb{R}_+) \times W^{2,1}_p(Q_T; \, \mathbb{R}^3) \, .
\end{align}
Since $\mathscr{A}$ is a certain composition of differentiation, multiplication and Nemicki operators, the properties above allow to show the following statement
\begin{lemma}\label{IMAGESPACE}
It the coefficients $R$, $\widetilde{M}$ and $P$ are continuously differentiable in their domain of definition $\mathbb{R}_+ \times \mathbb{R}^{N-1}$, the operator $\mathscr{A}$ is continuous and bounded from $\mathcal{X}_{T,+}$ into
\begin{align*}
\mathcal{Z}_T = L^p(Q_T; \, \mathbb{R}^{N-1}) \times L^{p,\infty}(Q_T) \times L^p(Q_T; \, \mathbb{R}^{3}) \, .
\end{align*}
It the coefficients $R$, $\widetilde{M}$ and $P$ are twice continuously differentiable in their domain of definition $\mathbb{R}_+ \times \mathbb{R}^{N-1}$, the operator $\mathscr{A}$ is continuously differentiable at every point of $\mathcal{X}_{T,+}$.
\end{lemma}
The same holds for the operator $\widetilde{\mathscr{A}}$. The proof of Lemma \ref{IMAGESPACE} can be carried over using standard differential calculus and the properties stated in the Remarks \ref{parabolicspace} and \ref{contispace}. In order to save room, we abstain from presenting it. The same estimates are needed in the proof of the main theorems anyway, and shall be exposed there. We shall moreover make use of a reduced state space, containing only the parabolic components $(q, \, v)$, namely
\begin{align}\label{parabolicSTATESPACE}
\mathcal{Y}_T := W^{2,1}_p(Q_T; \, \mathbb{R}^{N-1}) \times W^{2,1}_p(Q_T; \, \mathbb{R}^3) \, .
\end{align}
\textbf{Some short remarks on notation:} 1. We shall \emph{never} employ local H\"older continuous functions. For the sake of notation we identify $C^{\alpha, \, \beta}(Q)$ with $C^{\alpha, \, \beta}(\overline{Q})$; 2. Whenever confusion is impossible, we shall also employ for a function $f$ of the variables $x \in \Omega$ and $t \geq 0$ the notations $f_x = \nabla f$ for the spatial gradient, and $f_t$ for the time derivative; 3. For the coefficients $R$, $\widetilde{M}$, etc. which are functions of $\varrho$ and $q$, the derivatives are denoted $R_{\varrho}$, $\widetilde{M}_q$ etc.

\subsubsection{Boundary conditions and traces}

As before, we let $S_T = \partial \Omega \times ]0, \, T[$. As is well known, there is a well-defined trace operator $\text{tr}_{S_T} \in \mathscr{L}(W^{1,0}_p(Q), \, L^p(S_T))$ (even continuous with values in $L^p(0,T; \, W^{1-\frac{1}{p}}_p(\partial \Omega))$). Since $W^{2\ell,\ell}_p(Q) \subset W^{1,0}_p(Q)$ for $\ell \geq 1$, we can meaningfully define a Banach space
\begin{align*}
\text{Tr}_{S_T}\, W^{2\ell,\ell}_p(Q) :=&  \{ f \in L^p(S_T) \, : \, \exists \bar{f} \in W^{2\ell,\ell}_p(Q), \, \text{tr}_{S_T}(\bar{f}) = f\} \, ,\\
\|f\|_{\text{Tr}_{S_T} W^{2\ell,\ell}_p(Q)} :=&  \inf_{\bar{f} \in  W^{2\ell,\ell}_p(Q), \, \text{tr}_{S_T}(\bar{f}) = f} \, \|\bar{f}\|_{W^{2\ell,\ell}_p(Q)} \,.
\end{align*}
These spaces have been exactly characterised in terms of anisotropic fractional Sobolev spaces on the manifold $S_T$. The topic is highly technical. In particular, it is known that $\text{Tr}_{S_T}\, W^{2,1}_p(Q) = W^{2-\frac{1}{p}, \, 1-\frac{1}{2p}}(S_T)$: See \cite{denkhieberpruess}, while older references \cite{ladu}, \cite{sologeneral} seem to show only the inclusion $\text{Tr}_{S_T}\, W^{2,1}_p(Q) \subseteq W^{2-\frac{1}{p}, \, 1-\frac{1}{2p}}(S_T)$.

Next, we consider the conditions on the surface $\Omega \times \{0\}$, i. e. the initial conditions. There is a well defined trace operator $\text{tr}_{\Omega \times \{0\}} \in \mathscr{L}(C([0,T]; \, L^p(\Omega)), \, L^p(\Omega))$. Note that $W^{2\ell,\ell}_p(Q) \subset C([0,T]; \, L^p(\Omega))$ for $\ell \geq 1$. Thus, we can define similarly
\begin{align*}
\text{Tr}_{\Omega \times \{0\}} W^{2\ell,\ell}_p(Q) := & \{ f \in L^p(\Omega) \, : \, \exists \bar{f} \in W^{2\ell,\ell}_p(Q), \, \text{tr}_{\Omega \times \{0\}}(\bar{f}) = f\} \, ,\\
\|f\|_{\text{Tr}_{\Omega \times \{0\}} W^{2\ell,\ell}_p(Q)} := & \inf_{\bar{f} \in  W^{2\ell,\ell}_p(Q), \, \text{tr}_{\Omega \times \{0\}}(\bar{f}) = f} \, \|\bar{f}\|_{W^{2\ell,\ell}_p(Q)} \,.
\end{align*}
It is known that $\text{Tr}_{\Omega \times \{0\}} W^{2,1}_p(Q) = W^{2-\frac{2}{p}}_p(\Omega)$, see \cite{sologeneral}, or \cite{denkhieberpruess} and references for a complete characterisation using the Besov spaces. Here we can restrict to the Slobodecki space since $2-\frac{2}{p}$ is necessarily non-integer for $p > 3$. The spaces of zero initial conditions are defined via
\begin{align*}
\phantom{}_0W^{2,1}_p(Q_T) & := \{u \in W^{2,1}_p(Q_T) \, : \, u(0) = 0\}\, , \\ 
\phantom{}_0W^{1,1}_{p,\infty}(Q_T) & := \{u \in W^{1,1}_{p,\infty}(Q_T) \, : \, u(0) = 0\} \, ,\\
\phantom{}_0\mathcal{X}_T & := \phantom{}_0W^{2,1}_p(Q_T;\, \mathbb{R}^{N-1}) \times \phantom{}_0W^{1,1}_{p,\infty}(Q_T) \times \phantom{}_0W^{2,1}_p(Q_T; \, \mathbb{R}^3) \, ,\\
\phantom{}_0\mathcal{Y}_T & := \phantom{}_0W^{2,1}_p(Q_T;\, \mathbb{R}^{N-1}) \times \phantom{}_0W^{2,1}_p(Q_T; \, \mathbb{R}^3) \, .
\end{align*}

\subsubsection{Compatible extension of the boundary data}

The boundary operator for the problem \eqref{mass2}, \eqref{mass2tot}, \eqref{momentum2} on $S_T$ is chosen as simple as possible: linear and homogeneous (see \eqref{lateralq}, \eqref{lateralv}). Thus, we consider $\mathscr{B}(q, \, \varrho,\, v)$ given by
\begin{align*}
\mathscr{B}_1(q, \, \varrho, \, v) = \mathscr{B}_1(q) := & \nu \cdot \nabla q \, ,\\
\mathscr{B}_2 :\equiv & 0 \, ,\\
\mathscr{B}_3(q, \, \varrho, \, v) =  \mathscr{B}_3(v) := & v \, .
\end{align*}
The operator $\mathscr{B}$ is acting on the space $\mathcal{X}_T$.

 As usual for higher regularity, the choice of the initial conditions is restricted by the choice of the boundary operator. The conditions $q^0_i \in \text{Tr}_{\Omega \times \{0\}} W^{2,1}_p(Q_T)$, $v^0_i \in \text{Tr}_{\Omega \times \{0\}} W^{2,1}_p(Q_T)$ guarantee at first the existence of liftings $\hat{q}^0 \in W^{2,1}_p(Q_T; \, \mathbb{R}^{N-1})$ and $\hat{v}^0 \in W^{2,1}_p(Q_T; \, \mathbb{R}^{3})$. It is now necessary to homogenise all boundary data in such a way that these liftings are also in the kernel of the boundary operator. In order to find $\hat{q}^0 \in W^{2,1}_p$ satisfying $\hat{q}^0(0) = q^0$ and $\mathscr{B}_1(\hat{q}^0) = \nu \cdot \nabla \hat{q}^0 = 0$ on $S_T$ and $\hat{v}^0\in W^{2,1}_p$ satisfying $\hat{v}^0(0) = v^0$ and $\mathscr{B}_3(\hat{v}^0) = \hat{v}^0 = 0$ on $S_T$, we refer to the $L^p-$ theory of the Neumann/Dirichlet problem for the heat equation (see among others the monograph \cite{ladu}). There is, in both cases, one necessary compatibility condition,
\begin{align*}
\nu \cdot \nabla q^0 = 0 \text{ on } \partial \Omega, \quad v^0 = 0 \text{ on } \partial \Omega \, ,
\end{align*}
which make sense as identities in $\text{Tr}_{\partial\Omega} W^{1-\frac{2}{p}}_p(\Omega) =W^{1-\frac{3}{p}}_p(\partial \Omega) $ and in $\text{Tr}_{\partial\Omega} W^{2-\frac{2}{p}}_p(\Omega) = W^{2-\frac{3}{p}}_p(\partial \Omega)$. In order to find an extension for $\varrho_0 \in W^{1,p}(\Omega)$, we solve the problem
\begin{align}\label{Extendrho0}
 \partial_t \hat{\varrho}_0 + \divv(\hat{\varrho}_0 \, \hat{v}^0) =0, \quad \hat{\varrho}_0(0) = \varrho_0 \, .
\end{align}
For this problem, the Theorem 2 of \cite{solocompress} establishes unique solvability in $W^{1,1}_{p,\infty}(Q_T)$ and, among other, the strict positivity $\hat{\varrho}_0 \geq c_0(\Omega, \, \|\hat{v}^0\|_{W^{2,1}_p(Q_T; \, \mathbb{R}^3)}) \, \inf_{x \in \Omega} \varrho_0(x)$.

\section{Linearisation and reformulation as a fixed-point equation}\label{twomaps}

We shall present two different manners to linearise the equation $\mathscr{A}(u) = 0$ for $u \in \mathcal{X}_T$ with initial condition $u(0) = u_0$ in $\text{Tr}_{\Omega\times\{0\}} \, \mathcal{X}_T$: 
\begin{itemize}
 \item The first method is used to prove the statements on short-time existence in Theorem \ref{MAIN}, \ref{MAIN2};
\item The second technique shall be used to prove the global existence for restricted data in Theorem \ref{MAIN3};
\end{itemize}
The attentive reader will notice that the main estimate for the second linearisation techniques would also allow to prove the short-time existence. However, it has the drawback to be applicable only if the initial data possess more smoothness than generic elements of the state space $\mathcal{X}_T$. Thus, this technique does not allow to prove a semi-flow property. For this reason we think that, at the price of being lengthy, presenting the first method remains necessary.

In both cases, we start considering the problem to find $u =(q, \, \varrho ,\, v) \in \mathcal{X}_{T,+}$ such that $\widetilde{\mathscr{A}}(u) = 0$ and $u(0) = u_0$, which possesses the following structure:
\begin{align*}
\partial_t \varrho + \divv (\varrho \, v) = & 0 \, ,\\
R_{q}(\varrho, \, q) \, \partial_t q - \divv (\widetilde{M}(\varrho, \, q) \, \nabla q)  = & g(x, \, t, \, q, \, \varrho,\, v, \, \nabla q, \, \nabla \varrho, \, \nabla v) \, ,\\
 \varrho \, \partial_t v - \divv \mathbb{S}(\nabla v) = &  f(x, \, t,\, q, \, \varrho,\, v, \, \nabla q, \, \nabla \varrho, \, \nabla v) \, .
\end{align*}
For the original problem, the functions $g$ and $f$ have the following expressions
\begin{align}\label{A1right}
& g(x, \, t,\, q, \, \varrho,\, v, \, \nabla q, \, \nabla \varrho, \, \nabla v) := (R_{\varrho}(\varrho,\, q) \, \varrho - R(\varrho,\, q)) \, \divv v - R_q(\varrho,\, q) \, v \cdot \nabla q\nonumber \\
 & \qquad - \widetilde{M}_{\varrho}(\varrho, \, q) \, \nabla \varrho \cdot \tilde{b}(x, \,t) - \widetilde{M}_{q}(\varrho, \, q) \, \nabla q \cdot \tilde{b}(x, \,t)  -\widetilde{M}(\varrho, \, q) \, \divv \tilde{b}(x, \,t) - \tilde{r}(\varrho, \, q) \, , \\[0.2cm]
& \label{A3right} f(x, \, t,\, q, \, \varrho,\, v, \,  \nabla q, \, \nabla \varrho, \, \nabla v) := - P_{\varrho}(\varrho,\, q) \, \nabla \varrho -  P_{q}(\varrho,\, q) \, \nabla q 
 - \varrho \, (v\cdot \nabla)v \nonumber\\
 & \phantom{f(x, \, t,\, q, \, \varrho,\, v, \,  \nabla q, \, \nabla \varrho, \, \nabla v)  }  \qquad+ R(\varrho, \,q) \cdot \tilde{b}(x, \,t) + \varrho \, \bar{b}(x, \,t) \, .
\end{align}
In the proofs, we however consider the abstract general form of the right-hand sides. We shall also regard $g$ and $f$ as functions of $x, \,t$ and the vectors $u$ and $D_x u$ and write $g(x, \,t, \, u, \, D_xu)$ etc. 

\subsection{The first fixed-point equation}

For $u^* = (q^*, \, v^*)$ given in $ \mathcal{Y}_T$ (cf. \eqref{parabolicSTATESPACE}) and for unknowns $u = (q, \, \varrho, \, v)$, we consider the following system of equations
\begin{align}\label{linearT1}
\partial_t \varrho + \divv (\varrho \, v^*) = & 0 \, ,\\
\label{linearT2} R_{q}(\varrho, \,q^*) \, \partial_t q - \divv (\widetilde{M}(\varrho, \, q^*) \, \nabla q)  = & g(x, \, t, \, q^*, \, \varrho,\, v^*, \, \nabla q^*, \, \nabla \varrho, \, \nabla v^*) \, ,\\
\label{linearT3} \varrho \, \partial_t v - \divv \mathbb{S}(\nabla v) = & f(x, \, t,\, q^*, \, \varrho,\, v^*, \, \nabla q^*, \, \nabla \varrho, \, \nabla v^*) \, ,
\end{align}
together with the initial conditions \eqref{initialq}, \eqref{initialrho}, \eqref{initialv} and the homogeneous boundary conditions \eqref{lateralv}, \eqref{lateralq}. Note that the continuity equation can be solved independently for $\varrho$. Once $\varrho$ is given, the problem \eqref{linearT2}, \eqref{linearT3} is linear in $(q, \, v)$. 

We will show that the solution map $(q^*, \, v^*) \mapsto (q, \, v)$, denoted $\mathcal{T}$ is well defined from $\mathcal{Y}_T$ into itself. The solutions are unique in the class $\mathcal{Y}_T$. Clearly, a fixed point of $\mathcal{T}$ is a solution to $\widetilde{\mathscr{A}}(q, \, \varrho, \, v) = 0$.

\subsection{The second fixed-point equation}

We consider a given vector $\hat{u}^0 = (\hat{q}^0, \, \hat{\varrho}^0, \, \hat{v}^0) \in \mathcal{X}_T$ such that $\hat{q}^0$ and $\hat{v}^0$ satisfy the initial compatibility conditions. Moreover, we assume that $\hat{\varrho}^0$ obeys \eqref{Extendrho0}.

Consider a solution $u = (q, \,\varrho, \, v) \in \mathcal{X}_T$ to $\widetilde{\mathscr{A}}(u) = 0$. We introduce the differences $r := q - \hat{q}^0$, $w := v - \hat{v}^0$ and $\sigma := \varrho- \hat{\varrho}^0$, and the vector $\bar{u} := (r, \, \sigma, \, w)$. Clearly, $\bar{u}$ belongs to the space $\phantom{}_0\mathcal{X}_T$ of homogeneous initial conditions. The equations $\widetilde{\mathscr{A}}(u) = 0$ shall be equivalently re-expressed as a problem for the vector $\bar{u}$ via $\widetilde{\mathscr{A}}(\hat{u}^0 +\bar{u}) = 0$. The vector $\bar{u} = (r, \, \sigma, \, w)$ satisfies
\begin{align}\label{equationdiff1}
  R_q \, \partial_t r - \divv (\widetilde{M} \, \nabla r) = g^1 := & g -  R_q \, \partial_t \hat{q}^0 + \widetilde{M} \, \triangle \hat{q}^0 - \widetilde{M}_{\varrho} \, \nabla \varrho \cdot \nabla \hat{q}^0 \\
&  -  \widetilde{M}_{q} \, \nabla q \cdot \nabla \hat{q}^0\, ,\nonumber\\
\label{equationdiff2}
 \partial_t \sigma + \divv(\sigma \, v) =& - \divv(\hat{\varrho}_0 \, w) \, ,\\
\label{equationdiff3}
  \varrho \, \partial_t w - \divv \mathbb{S}(\nabla w) = & f^1 =: f - \varrho \partial_t \hat{v}^0 + \divv \mathbb{S}(\nabla \hat{v}^0) \, . 
\end{align}
Herein, the coefficients $R, \, R_q$, etc.\ are evaluated at $(\varrho, \, q)$, while $g$ and $f$ correspond to \eqref{A1right} and \eqref{A3right}.

We next want to construct a fixed-point map to solve \eqref{equationdiff1}, \eqref{equationdiff2}, \eqref{equationdiff3} by linearising the operators $g^1$ and $f^1$ defined in \eqref{equationdiff1} and \eqref{equationdiff3}. At a point $u^* = (q^*, \, \varrho^*, \, v^*) \in \mathcal{X}_{T,+}$ (cf. \eqref{STATESPACEPOS}), we can expand as follows:
\begin{align*}
 g =  g(x, \, t,\, u^*, \, D_xu^*) + \int_{0}^1 & \{ (g_{q})^{\theta} \, (q-q^*) + (g_{\varrho})^{\theta} \, (\varrho-\varrho^*) + (g_v)^{\theta} \, (v-v^*) \\
 & + (g_{q_x})^{\theta} \cdot (q_x-q^*_x) + (g_{\varrho_x})^{\theta} \, (\varrho_x- \varrho^*_x) + (g_{v_x})^{\theta} \cdot (v_x-v^*_x)\} \, d\theta\nonumber \, .
 \end{align*}
 Here the brackets $( \cdot )^{\theta}$, if applied to a function of $x, \, t$, $u$ and $D^1_x u$, stand for the evaluation at $(x, \, t,\, (1-\theta) \, u^* + \theta \, u, \, (1-\theta) \, D_xu^* + \theta \, D_xu)$. In short, in order to avoid the integral and the parameter $\theta$, we write
 \begin{align}\label{Arightlinear}
  g = &g(x, \, t,\, u^*, \, D_xu^*) + g_{q}(u, \, u^*) \, (q-q^*) + g_{\varrho}(u, \, u^*) \, (\varrho-\varrho^*)+ g_v(u, \, u^*)  \, (v-v^*)  \nonumber\\
  & + g_{q_x}(u, \, u^*) \cdot (q_x-q^*_x) + g_{\varrho_x}(u, \, u^*)  \, (\varrho_x- \varrho^*_x) + g_{v_x}(u, \, u^*) \cdot (v_x-v^*_x) \nonumber\\
 =: & g(x, \, t,\, u^*, \, D_xu^*) + g^{\prime}(u, \, u^*) \, (u - u^*) \, . 
\end{align}
We follow this scheme and write in short
\begin{align}\label{Arightlinear2}
  g^1 = & g^1(x, \, t,\, \hat{q}^0, \, \hat{\varrho}^0, \, \hat{v}^0, \, \hat{q}_x^0, \, \hat{\varrho}_x, \, \hat{v}_x^0) + g^1_{q}(u, \, \hat{u}^0) \, r + g^1_{\varrho}(u, \, \hat{u}^0) \, \sigma + g^1_v(u, \, \hat{u}^0)\, w \nonumber\\
 &  + g^1_{q_x}(u, \, \hat{u}^0) \, r_x + g^1_{\varrho_x}(u, \, \hat{u}^0) \, \sigma_x + g^1_{v_x}(u, \, \hat{u}^0) \, w_x \, \nonumber\\
  =: &  \hat{g}^0 + (g^1)^{\prime}(u, \, \hat{u}^0) \, \bar{u} \, .
\end{align}
With obvious modifications, we have the same formula for $f^1$. 
Now we construct the fixed-point map to solve \eqref{equationdiff1}, \eqref{equationdiff2}, \eqref{equationdiff3}. For a given vector $(r^*, \, w^*) \in \phantom{}_0\mathcal{Y}_T$, we define $q^* := \hat{q}^0 + r^*$ and $v^* := \hat{v}^0 + w^*$. We employ the abbreviation
\begin{align}\label{ustar}
u^* := &  (q^*, \, \mathscr{C}(v^*), \, v^*) \in \mathcal{X}_{T,+} \, ,
\end{align}
where $\mathscr{C}$ is the solution operator to the continuity equation with initial datum $\varrho_0$.
For $\bar{u} := (r, \, \sigma, \, w)$, we next consider the linear problem
\begin{alignat}{2} 
\label{linearT1second} R_q(\mathscr{C}(v^*), \, q^* ) \, \partial_t r - \divv (\widetilde{M}(\mathscr{C}(v^*), \, q^* ) \, \nabla r) =&  \hat{g}^0 + (g^{1})^{\prime}(u^*, \, \hat{u}^0) \, \bar{u} \, ,& & \\
\label{linearT2second} \partial_t \sigma + \divv(\sigma \, v^*) = &  - \divv(\hat{\varrho}_0 \, w) \, ,  & & \\
\label{linearT3second} \mathscr{C}(v^*) \, \partial_t w - \divv \mathbb{S}(\nabla w) = & \hat{f}^0 + (f^1)^{\prime}(u^*, \, \hat{u}^0) \, \bar{u} \, , & & 
\end{alignat}
with boundary conditions $\nu \cdot \nabla r = 0$ on $S_T$ and $w = 0$ on $S_T$ and with zero initial conditions. We will show that the solution map $(r^*, \, w^*) \mapsto (r, \, w)$, denoted as $\mathcal{T}^1$, is well defined from $\phantom{}_0\mathcal{Y}_T$ into itself. 
\begin{rem}
If $(r, \, w)$ is a fixed point of $\mathcal{T}^1$, then $u := \hat{u}^0 + (r, \, \sigma, \, w)$ is a solution to $\widetilde{\mathscr{A}}(u) = 0$.
\end{rem}
\begin{proof}
To see this, we note first that a fixed point satisfies $\mathcal{T}^1(r, \, w) = (r, \, w)$, hence the following equations are valid:
\begin{alignat*}{2}
R_q(q, \, \mathscr{C}(v)) \, \partial_t r - \divv (\widetilde{M}(q, \, \mathscr{C}(v)) \, \nabla r) =&  \hat{g}^0 + (g^1)^{\prime}((q, \, \mathscr{C}(v), \, v), \, \hat{u}^0) \, \bar{u} \, ,& & \\
 \partial_t \sigma + \divv(\sigma \, v) = &  - \divv(\hat{\varrho}_0 \, w) \, ,  & & \\
 \mathscr{C}(v) \, \partial_t w - \divv \mathbb{S}(\nabla w) = & \hat{f}^0 + (f^1)^{\prime}((q, \, \mathscr{C}(v), \, v), \, \hat{u}^0) \, \bar{u} \, . & & 
\end{alignat*}
Adding to the second equation the identity \eqref{Extendrho0}, valid by construction, we see that $\tilde{\varrho} := \hat{\varrho}^0 + \sigma$ is a solution to the continuity equation with velocity $v$ and initial data $\varrho^0$. Thus $\tilde{\varrho} = \mathscr{C}(v)$ (uniqueness for the continuity equation, cf. Proposition \ref{solonnikov2} below). Now, we see by the definitions of $(g^1)^{\prime}$ and $(f^1)^{\prime}$ (cf.\ \eqref{Arightlinear2}) that
\begin{align*}
  \hat{g}^0 + (g^1)^{\prime}((q, \, \mathscr{C}(v), \, v), \, \hat{u}^0) \, \bar{u} = & \hat{g}^0 + (g^1)^{\prime}((\hat{q}^0+r, \, \hat{\varrho}^0+\sigma, \, \hat{v}^0 + w), \, \hat{u}^0) \, (r, \, \sigma, \, w)\\
 =&  g^1(\hat{q}^0+r, \, \hat{\varrho}^0+\sigma, \, \hat{v}^0 + w)
 \end{align*}
 and, analogously, $\hat{f}^0 + (f^1)^{\prime}((q, \, \mathscr{C}(v), \, v), \, \hat{u}^0) \, \bar{u}  = f^1$. Thus we recover a solution to the equations \eqref{equationdiff1}, \eqref{equationdiff2} and \eqref{equationdiff3}.
\end{proof}

\subsection{The self-mapping property}

Assuming for a moment that the map $\mathcal{T}$, $(q^*, \, v^*) \mapsto (q, \, v)$ via the solution to \eqref{linearT1}, \eqref{linearT2}, \eqref{linearT3} is well defined in the state space $\mathcal{Y}_T$, then the main difficulty to prove the existence of a fixed-point is to show that $\mathcal{T}$ maps some closed bounded set of $\mathcal{Y}_T$ into itself. If $\mathcal{T}$ is well--defined and continuous, we shall rely on the continuous estimates
\begin{align}\label{CONTROLLEDGROWTH}
 \|(q, \, v)\|_{W^{2,1}_p(Q_t; \, \mathbb{R}^{N-1}) \times W^{2,1}_p(Q_t; \, \mathbb{R}^{3})} \leq \Psi(t, \, R_0, \, \|(q^*, \, v^*)\|_{W^{2,1}_p(Q_t; \, \mathbb{R}^{N-1}) \times W^{2,1}_p(Q_t; \, \mathbb{R}^{3})}) \, ,
\end{align}
valid for all $t \leq T$ with a function $\Psi$ being continuous in all arguments. Here $R_0$ is a parameter standing for the magnitude of the initial data $q^0$, $\varrho_0$ and $v^0$ and of the external forces $b$ in their respective norms. An important observation of the paper \cite{solocompress} is the following.
\begin{lemma}\label{selfmapT}
Suppose that $R_0 > 0$ is fixed. Suppose that for all $t \leq T$, the inequality \eqref{CONTROLLEDGROWTH} is valid with a continuous function $\Psi = \Psi(t, \, R_0, \, \eta) \geq 0$ defined for all $t \geq 0$ and $\eta \geq 0$ and increasing in these arguments. Assume moreover that $\Psi(0, \, R_0, \, \eta) = \Psi^0(R_0) > 0$ is independent of $\eta$. Then there are $t_0 = t_0(R_0) > 0$ and $\eta_0 = \eta_0(R_0) > 0$ such that $\mathcal{T} \, (q^*, \, v^*) := (q, \, v)$ maps the closed ball with radius $\eta_0$ in $\mathcal{Y}_{t_0}$ into itself. 
\end{lemma}
\begin{proof}
We have to show that $\eta_0 := \inf  \{\eta > 0 \, : \, \Psi(t_0, \, R_0, \, \eta) \leq \eta\} > 0$, since then \eqref{CONTROLLEDGROWTH} implies
\begin{align*}
\|\mathcal{T} \, (q^*, \, v^*)\|_{\mathcal{Y}_{t_0}} \leq \Psi(t_0, \, R_0, \, \|(q^*, \, v^*)\|_{\mathcal{Y}_{t_0}}) \leq \|(q^*, \, v^*)\|_{\mathcal{Y}_{t_0}} \, ,
\end{align*}
whenever $\|(q^*, \, v^*)\|_{\mathcal{Y}_{t_0}} \leq \eta_0$. Hence $\mathcal{T} $ maps the closed ball with radius $\eta_0$ in $\mathcal{Y}_{t_0}$ into itself. Now $\eta_0 = 0$ yields $\Psi(t_0, \, R_0, \, \eta) =0$ by the continuity of $\Psi$ (and since $\Psi$ is nonnegative by assumption), hence implies the contradiction $\Psi(0, \, R_0, \, \eta) = \Psi^0(R_0) = 0$.
%
%
%
%
%
\end{proof}
The strategy for proving Theorem \ref{MAIN3} shall be quite similar. We use here the map $\mathcal{T}^1$, $(r^*, \,  w^*) \mapsto (r, \,  w)$ defined via solution to \eqref{linearT1second}, \eqref{linearT2second} and \eqref{linearT3second}. In this case, the fixed-point we look for is in the space $\phantom{}_0\mathcal{Y}_T$ and we expect a continuity estimate of the type
\begin{align}\label{CONTROLLEDGROWTH2}
 \|(r, \,  w)\|_{\phantom{}_0\mathcal{Y}_T} \leq \Psi(T, \, R_0, \, R_1, \, \|(r^*, \,  w^*)\|_{\phantom{}_0\mathcal{Y}_T}) \, .
\end{align}
Here $R_0$ stands for magnitude of the initial data $q^0$, $\varrho_0$ and $v^0$ and of the external forces $b$, while the parameter $R_1$ expresses the distance of these initial data to a stationary/equilibrium solution.  
\begin{lemma}\label{selfmapTsecond}
 Suppose that $T > 0$ and $R_0 >0$ are arbitrary but fixed. Suppose that \eqref{CONTROLLEDGROWTH2} is valid with a continuous function $\Psi = \Psi(T, \, R_0, \, R_1, \, \eta)$ defined for all $R_1 \geq 0$ and $\eta \geq 0$, and increasing in these arguments. Assume moreover that $\Psi(T, \, R_0, \, 0, \, \eta) = 0$ and that $\Psi(T, \, R_0, \, R_1, \, 0) > 0$. Then, there is $\delta > 0$ such that if $R_1 \leq \delta$, we can find $\eta_0 > 0$ such that $\mathcal{T}^1$ maps the set $\{\bar{u} \in \phantom{}_0\mathcal{Y}_T \, : \,   \|\bar{u}\|_{\mathcal{Y}_T} \leq \eta_0\}$ into itself. 
\end{lemma}
The proof can be left to the the reader as it is completely similar to the one of Lemma \ref{selfmapT}.
%
%
%
In order to prove the Theorems we shall therefore prove the continuity estimate \eqref{CONTROLLEDGROWTH}, \eqref{CONTROLLEDGROWTH2}. This is the main object of the next sections.

\section{Estimates of linearised problems}\label{ESTI}

In this section, we present the estimates on which our main results in Theorem \ref{MAIN}, \ref{MAIN2} are footing. In order to motivate the procedure, we recall that we want to prove the continuity estimate \eqref{CONTROLLEDGROWTH} for the map $\mathcal{T}$ in Section \ref{twomaps}. With this fact in mind it shall be easier for the reader to follow the technical exposition. The proof is split into several subsections. To achieve also more simplicity in the notation, we introduce indifferently for a function or vector field $f \in W^{2,1}_p(Q_T; \, \mathbb{R}^k)$ ($p > 3$ fixed, $k\in \mathbb{N}$) and $t \leq T$ the notation
\begin{align}\label{Vfunctor}
 \mathscr{V}(t; \, f) :=  \|f\|_{W^{2,1}_p(Q_t; \, \mathbb{R}^k)} + \sup_{s \leq t} \|f(\cdot, \, s)\|_{W^{2-\frac{2}{p}}_p(\Omega; \, \mathbb{R}^k)} \, .
\end{align}
Moreover, we will need H\"older half-norms. For $\alpha, \, \beta \in [0,\, 1]$ and $f$ scalar--valued, we denote
\begin{align*}
[f]_{C^{\alpha}(\Omega)} := \sup_{x \neq y \in \Omega} \frac{|f(x) - f(y)|}{|x-y|^{\alpha}}, \quad [f]_{C^{\alpha}(0,T)} := \sup_{t \neq s \in [0,T]} \frac{|f(t) - f(s)|}{|t-s|^{\alpha}}\\
[f]_{C^{\alpha,\beta}(Q_T)} := \sup_{t \in [0, \, T]} [f(\cdot, \, t)]_{C^{\alpha}(\Omega)} +  \sup_{x \in \Omega} [f(x, \cdot)]_{C^{\beta}(0,T)}  \, .
\end{align*}
The corresponding H\"older norms $\|f\|_{C^{\alpha}(\Omega)}$, $\|f\|_{C^{\alpha}(0,T)}$ and $f \in C^{\alpha,\beta}(Q_T)$ are defined adding the corresponding $L^{\infty}-$norm to the half-norm.

\subsection{Estimates for a linearised problem in the variables $q_1,\ldots, q_{N-1}$}

We commence with a statement concerning the linearisation of $\widetilde{\mathscr{A}}^1$ (cf.\ \eqref{A1equiv}).
\begin{prop}\label{A1linmain}
Assume that $R_q, \, \widetilde{M}: \, \mathbb{R}_+ \times \mathbb{R}^{N-1} \rightarrow \mathbb{R}^{(N-1)\times (N-1)}$ are maps of class $C^1$ into the set of positively definite matrices. Suppose further that $q^{*} \in W^{2,1}_p(Q_T; \, \mathbb{R}^{N-1})$ and $\varrho^* \in W^{1,1}_{p,\infty}(Q_T)$ ($p > 3$) are given, where $\varrho^*$ is strictly positive. We denote $R_q^* := R_q(\varrho^*, \, q^{*})$ and $ \widetilde{M}^* := \widetilde{M}(\varrho^*, \, q^{*})$. For $t \leq T$, we further define
\begin{align*}
m^*(t) := \inf_{(x,s) \in Q_t} \varrho^*(x, \, s) > 0 \, , \quad M^*(t) := \sup_{(x,s) \in Q_t} \varrho^*(x, \, s) \, .
\end{align*}
We assume that $g \in L^p(Q_T; \, \mathbb{R}^{N-1})$ and $q^0 \in W^{2-\frac{2}{p}}(\Omega)$ are given and that $\nu \cdot \nabla q^0(x) = 0$ in the sense of traces on $\partial \Omega$. Then there is a unique $q \in W^{2,1}_p(Q_T; \, \mathbb{R}^{N-1})$, solution to the problem
\begin{align}\label{qlinear}
R_q^* \,  q_t - \divv (\widetilde{M}^{*} \, \nabla q) = g \, \text{ in } Q_T, \quad \nu \cdot \nabla q = 0 \text{ on } S_T, \, \quad
q(x, \, 0) = q^0(x) \text{ in } \Omega \, .
\end{align}
Moreover, there is a constant $C$ independent on $T$, $q$, $\varrho^*$ and $q^*$ as well as a continuous function $\Psi_1 = \Psi_1(t, \, a_1,\ldots,a_6)$ defined for all $t \geq 0$ and all numbers $a_1, \ldots, a_6 \geq 0$ such that for all $t \leq T$ and $0 < \beta \leq 1$, it holds that
\begin{align*}
& \mathscr{V}(t; \, q) \leq C \, \Psi_{1,t}\, \left[(1+[\varrho^*]_{C^{\beta,\frac{\beta}{2}}(Q_t)})^{\tfrac{2}{\beta}} \, \|q^0\|_{W^{2-\frac{2}{p}}_p(\Omega)} + \|g\|_{L^p(Q_t)}\right] \, , \\[0.1cm]
& \Psi_{1,t} := \Psi_1(t, \, (m^*(t))^{-1}, \, M^*(t),  \, \|q^*(0)\|_{C^{\beta}(\Omega)}, \, 
\mathscr{V}(t; \, q^*), \, [\varrho^*]_{C^{\beta,\frac{\beta}{2}}(Q_t)}, \, \|\nabla \varrho^*\|_{L^{p,\infty}(Q_t)}) \, .
\end{align*}
In addition, $\Psi_1$ is increasing in all arguments and $\Psi_1(0, \, a_1,\ldots,a_6) = \Psi_1^0(a_1, \, a_2, \, a_3)$ does not depend on the last three arguments.
\end{prop}
\begin{proof}
 We prove here only the unique solvability. Due to the technicality, the proof of the estimate will be given separately hereafter.
 After computation of the divergence and inversion of $R_q^*$ in \eqref{qlinear}, the vector field $q$ is equivalently asked to satisfy the relations
\begin{align}\label{petrovskisyst}
 q_t - [R_q^*]^{-1} \, \widetilde{M}^* \, \triangle q = [R_q^*]^{-1} \, g + [R_q^*]^{-1} \, \nabla \widetilde{M}^* \cdot \nabla q \, \, .
\end{align}
The matrix $A^* := [R_q^*]^{-1} \, \widetilde{M}^*$ is the product of two symmetric positive semi-definite matrices. The Lemma \ref{ALGEBRA} implies that the eigenvalues are real and strictly positive. Moreover,
\begin{align}\label{EVPROD}
 \frac{\lambda_{\min}(\widetilde{M}^*)}{\lambda_{\max}(R_q^*)} \leq \lambda_{\min}(A^*) \leq \lambda_{\max}(A^*) \leq  \frac{\lambda_{\max}(\widetilde{M}^*)}{\lambda_{\min}(R_q^*)} \, .
\end{align}
Thus, the equations \eqref{petrovskisyst} are a linear parabolic system in the sense of Petrovski (\cite{ladu}, Chapter VII, Paragraph 8, Definition 2). We apply the result of \cite{sologeneral}, Chapter V recapitulated in \cite{ladu}, Chapter VII, Theorem 10.4, enriched and refined in several contributions of the school of maximal parabolic regularity as for instance in \cite{denkhieberpruess}, and we obtain the unique solvability. Note that in the case of the equations \eqref{petrovskisyst}, this machinery does not need to be applied in its full complexity. The reason is that the differential operator of second order in space is the Laplacian in each row of the system. Using this fact, the continuity estimate for the system \eqref{petrovskisyst} can be established by elementary means, as revealed by the proof of Lemma \ref{A1linUMFORM} below (Appendix, Section \ref{technique}). From the estimate we can easily pass to solvability by linear continuation.  
\end{proof}
Due to its technicality, the proof of the estimate is split into several steps. The first step, accomplished in the following Lemma, is the principal estimate. Subsequent statements are needed to attain the bound as formulated in Proposition \ref{A1linmain}.
\begin{lemma}\label{A1linUMFORM}
We adopt the assumptions of Proposition \ref{A1linmain}. Then for $\beta \in ]0, \, 1]$ arbitrary, there is a constant $C$ independent on $T$, $q$, $\varrho^*$ and $q^*$ such that, for all $t \leq T$,
\begin{align*}
\mathscr{V}(t; \, q) \leq &  C \, \phi^*_{0,t} \, (1+ [\varrho^*]_{C^{\beta,\frac{\beta}{2}}(Q_t)} + [q^*]_{C^{\beta,\frac{\beta}{2}}(Q_t)})^{\frac{2}{\beta}} \,  (\|q^0\|_{W^{2-\frac{2}{p}}_p(\Omega)} + \|q\|_{W^{1,0}_p(Q_t)}) \\
& + C \,  \phi^*_{1,t} \,  (\|g\|_{L^p(Q_t)} + \|\nabla \varrho^* \cdot \nabla q\|_{L^p(Q_t)} + \|\nabla q^* \cdot \nabla q\|_{L^p(Q_t)} ) \, . 
\end{align*}
For $i = 0, \, 1$, there is a continuous function $\phi_{i}^* = \phi^*_i(a_1, \, a_2, \, a_3)$, defined for $a_1,\, a_2,\, a_3 \geq 0$ and increasing in each argument, such that $ \phi^*_{i,t} = \phi^*_i((m^*(t))^{-1}, \,  M^*(t), \, \|q^*\|_{L^{\infty}(Q_t; \, \mathbb{R}^{N-1})})$.
\end{lemma}
\begin{rem}\label{eigenvaluesandlipschitzconstants}
 The proof shall moreover show that the growth of $\phi^*_{0,t}, \, \phi^*_{1,t}$ can be estimated by a function of the minimal/maximal eigenvalues of the matrices $R_q^*$ and $ \widetilde{M}^*$, and of their local Lipschitz constants over the range of $(\varrho^*, \, q^*)$. To extract this point more easily, we define for a Lipschitz continuous matrix valued mapping $A: \, \mathbb{R}^+\times \mathbb{R}^{N-1} \rightarrow \mathbb{R}^{N-1} \times \mathbb{R}^{N-1}$, taking values in the positive definite matrices (for instance $A = R_q$ or $A = \widetilde{M}$), and for $t \geq 0$ functions
 \begin{align*}
  \lambda_0(t, \, A^*) := & \inf_{(x, \, s) \in Q_t} \lambda_{\min}[A(\varrho^*(x, \,s), \, q^*(x, \, s))]\, ,\\
  \lambda_1(t, \, A^*) := & \sup_{(x, \, s) \in Q_t} \lambda_{\max}[A(\varrho^*(x, \,s), \, q^*(x, \, s))]\, ,\\
  L(t, \, A^*) := & \sup_{(x, \, s) \in Q_t} |\partial_{\varrho}A(\varrho^*(x, \,s), \, q^*(x, \, s))| + \sup_{(x, \, s) \in Q_t} |\partial_{q}A(\varrho^*(x, \,s), \, q^*(x, \, s))|\, .
 \end{align*}
It is possible to reinterpret these expressions as increasing functions of $ (m^*(t))^{-1}$, of $M^*(t)$, and of $\|q^*\|_{L^{\infty}(Q_t; \, \mathbb{R}^{N-1})}$.
In the statement of Lemma \ref{A1linUMFORM}, we then can choose
 \begin{align}\label{LinftyFACTORS}
\phi^*_{0,t} := & \frac{\lambda_{0}(t, \, R_q^*)+\lambda_{1}(t, \, \widetilde{M}^*)}{\lambda_{0}^{\tfrac{3}{2}}(t, \, R_q^*)} \, \frac{\max\{1, \, \frac{\lambda_{1}(t, \, \widetilde{M}^*)}{\lambda_{0}(t, \, R_q^*)}\}}{\min\{1, \, \frac{\lambda_{0}(t, \, \widetilde{M}^*)}{\lambda_{1}(t, \, R_q^*)}\}} \times \nonumber\\
 & \times \left(1+ \frac{ \max\{1, \, \frac{\lambda_{1}(t, \, \widetilde{M}^*)}{\lambda_{0}(t, \, R_q^*)}\}}{ \min\{1, \, \frac{\lambda_{0}(t, \, \widetilde{M}^*)}{\lambda_{1}(t, \, R_q^*)}\}} \,  \frac{(\lambda_{1}(t, \, \widetilde{M}^*) + \lambda_{0}(t, \, R_q^*)) \, (L(t, \, \widetilde{M}^*) + L(t, \, R_q^*) ) }{\lambda_{0}^{\frac{5}{2}}(t, \, R_q^*)}\right) \, ,\nonumber\\
\phi^*_{1,t} := & \frac{ (1+L(t, \, \widetilde{M}^*)) \, \max\{1, \, \frac{\lambda_{1}(t, \, \widetilde{M}^*)}{\lambda_{0}(t, \, R_q^*)}\}}{\lambda_{0}^{\tfrac{3}{2}}(t, \, R_q^*) \, \min\{1, \, \frac{\lambda_{0}(t, \, \widetilde{M}^*)}{\lambda_{1}(t, \, R_q^*)}\}} \,  \, . 
\end{align}
\end{rem}
The proof is interesting, but lengthy. Since the use of H\"older norms to control the dependence on the coefficients is a classical tool, we prove these statements in the Appendix, Section \ref{technique}.

In order to prove Proposition \ref{A1linmain}, we need some reformulation of the estimate of Lemma \ref{A1linUMFORM}.
\begin{coro}\label{A1linUMFORMsecond}
We adopt the situation of Proposition \ref{A1linmain}, and for $\beta \in ]0, \, 1]$, we denote
\begin{align*}
 \phi_{2,t}^*  & := \phi^*_{0,t} \, (1+ [\varrho^*]_{C^{\beta,\frac{\beta}{2}}(Q_t)} + [q^*]_{C^{\beta,\frac{\beta}{2}}(Q_t)})^{\frac{2}{\beta}} \, ,\\
 B^*_t & := \phi^*_{2,t} + 1 +(\phi^*_{1,t})^{\tfrac{2p}{p-3}} \, (\sup_{s\leq t} \|q^*(s)\|_{W^{2-\frac{2}{p}}_p(\Omega)} + \|\nabla \varrho^*\|_{L^{p,\infty}(Q_t)})^{\tfrac{2p}{p-3}} \, .
\end{align*}
Then there are constants $c_1, \, c_2$, independent on $T$, $q$, $\varrho^*$ and $q^*$, such that
\begin{align*}
\mathscr{V}(t; \, q) \leq c_1 \, \big(1+ t^{\tfrac{1}{p}}\, B^*_t \, \exp(c_2 \, t\, [B^*_t]^p)\big) \,  (\phi^*_{2,t} \, \|q^0\|_{W^{2-\frac{2}{p}}_p(\Omega)} + \phi^*_{1,t} \, \|g\|_{L^p(Q_t)})  \, .
\end{align*}
\end{coro}
\begin{proof}
We start from the main inequality of Lemma \ref{A1linUMFORM}. Raising it to the $p-$th power, we obtain
\begin{align}\label{ControlProducts0}
 \mathscr{V}^p(t, \, q) \leq & C_p \, (\phi^*_{2,t})^p \, (\|q^0\|^p_{W^{2-\frac{2}{p}}_p(\Omega)} + \|q\|^p_{W^{1,0}_p(Q_t)}) \nonumber\\
 & + C_p \,(\phi^*_{1,t})^p  (\|g\|_{L^p(Q_t)}^p + \|\nabla \varrho^* \cdot \nabla q\|^p_{L^p(Q_t)} + \|\nabla q^* \cdot \nabla q\|^p_{L^p(Q_t)} )  \, . 
\end{align}
Since $\|q\|_{W^{1,0}_p(Q_t)}^p = \int_{0}^t \|q(s)\|_{W^{1,p}(\Omega)}^p \, ds$ and $\|q(s)\|_{W^{1,p}(\Omega)} \leq \sup_{\tau\leq s} \|q(\tau)\|_{W^{2-\frac{2}{p}}_p(\Omega)}$, we see that $\|q\|_{W^{1,0}_p(Q_t)}^p \leq \int_{0}^t \sup_{\tau\leq s} \|q(\tau)\|_{W^{2-\frac{2}{p}}_p(\Omega)}^p \, ds$.

Owing to the G.~N. inequality \eqref{gagliardo}, $$\|\nabla q(s)\|_{L^{\infty}(\Omega)} \leq C_1 \, \|D^2q(s)\|_{L^p(\Omega)}^{\alpha} \, \|q(s)\|_{L^p(\Omega)}^{1-\alpha} + C_2 \, \|q(s)\|_{L^p(\Omega)}$$ for $\alpha := \tfrac{1}{2}+\tfrac{3}{2p}$. Employing also Young's inequality, $a \, b \leq \epsilon \, a^{\tfrac{1}{\alpha}} + c_{\alpha} \, \epsilon^{-\tfrac{\alpha}{1-\alpha}} \, b^{\tfrac{1}{1-\alpha}}$, valid for all $\epsilon > 0$ and $a, \, b > 0$ it follows that
\begin{align}\label{ControlProducts1}
&  \|\nabla \varrho^* \cdot \nabla q\|_{L^p(Q_t)}^p \leq \int_{0}^t |\nabla \varrho^*(s)|_{p}^p \, |\nabla q(s)|_{\infty}^p \, ds\nonumber\\
&  \leq C_1 \, \int_{0}^t |\nabla \varrho^*(s)|_{p}^p \, |D^2q(s)|_{p}^{p\alpha} \, |q(s)|_{p}^{p(1-\alpha)} \, ds + C_2 \,\int_{0}^t |\nabla \varrho^*(s)|_{p}^p \, |q(s)|_{p}^p \, ds\nonumber\\
& \leq  \epsilon \, \int_{0}^t  |D^2q(s)|_{p}^{p} \, ds + c_{\alpha} \, \epsilon^{-\tfrac{\alpha}{1-\alpha}} \, \int_{0}^t |\nabla \varrho^*(s)|_{p}^{\frac{p}{1-\alpha}} \, |q(s)|_{p}^{p} \, ds + C_2 \,\int_{0}^t |\nabla \varrho^*(s)|_{p}^p \, |q(s)|_{p}^p \, ds\nonumber\\
&  \leq \epsilon \, \int_{0}^t  |D^2q(s)|_{p}^{p} \, ds + \int_{0}^t |q(s)|_{L^p}^{p} \, (c_{\alpha} \, \epsilon^{-\frac{\alpha}{1-\alpha}}\, |\nabla \varrho^*(s)|_{p}^{\frac{p}{1-\alpha}} + C_2\,  |\nabla \varrho^*(s)|_{p}^p) \, ds \, . 
\end{align}
Here we have denoted by $| \cdot |_{r}$ the norm in $L^r(\Omega)$ in order to save room. Similarly,
\begin{align}\label{ControlProducts2}
 \|\nabla q^* \cdot \nabla q\|_{L^p(Q_t)}^p \leq & \epsilon \, \int_{0}^t  |D^2q(s)|_{p}^{p} \, ds  + \int_{0}^t |q(s)|_{p}^{p} \, (c_{\alpha} \, \epsilon^{-\frac{\alpha}{1-\alpha}} \,  |\nabla q^*(s)|_{p}^{\frac{p}{1-\alpha}} + C_2\,  |\nabla q^*(s)|_{p}^p) \, ds\, .
\end{align}
In \eqref{ControlProducts1} and \eqref{ControlProducts2} we estimate roughly $ \|q(s)\|_{L^p(\Omega)} \leq \sup_{\tau \leq s} \|q(\tau)\|_{W^{2-\frac{2}{p}}_p(\Omega)}$. Using the abbreviation $F^*(s) := \|\nabla q^*(s)\|_{L^p(\Omega)}^p+  \|\nabla \varrho^*(s)\|_{L^p(\Omega)}^p $, it follows that
 \begin{align}\label{ControlProducts3}
&  \|\nabla \varrho^* \cdot \nabla q\|^p_{L^p(Q_t)} + \|\nabla q^* \cdot \nabla q\|^p_{L^p(Q_t)} \leq  2 \, \epsilon \, \int_{0}^t  |D^2q(s)|_{p}^{p} \, ds \\
& \qquad  + \int_{0}^t \sup_{\tau\leq s} \|q(s)\|_{W^{2-\frac{2}{p}}_p(\Omega))}^p\, [ c_{\alpha} \, \epsilon^{-\frac{\alpha}{1-\alpha}} \,  (F^*(s))^{\frac{1}{1-\alpha}} + C_2\,  F^*(s)] \, ds \, . \nonumber
 \end{align}
Recalling \eqref{ControlProducts0}, we choose $\epsilon = \frac{1}{4\,C_p\, (\phi^*_{1,t})^p}$. This, in connection with \eqref{ControlProducts0}, \eqref{ControlProducts3}, now yields
\begin{align}\label{ControlProducts4}
 \frac{1}{2} \,  \mathscr{V}^p(t; \, q) & \leq c_p \, ((\phi^*_{2,t})^p \, \|q^0\|^p_{W^{2-\frac{2}{p}}_p(\Omega)} + (\phi^*_{1,t})^p \, \|g\|_{L^p(Q_t)}^p + E(t) \, \int_{0}^t f(s) \, ds) \, ,\\ 
f(s) & := \sup_{\tau \leq s} \|q(s)\|^p_{W^{2-\frac{2}{p}}_p(\Omega)}\, ,\nonumber\\
E(t)&  := (\phi^*_{2,t})^p  + \tilde{c}_{p,\alpha} \,  (\phi^*_{1,t})^{\frac{p}{1-\alpha}} \,  \sup_{s \leq t} (F^*(s))^{\frac{1}{1-\alpha}}
+ C_2 \, (\phi^*_{1,t})^p \,  \sup_{s \leq t} F^*(s) \, .\nonumber
\end{align}
The latter inequality implies that $f(t) \leq A(t) + E(t) \, \int_0^t f(s)  \, ds$. By Gronwall's Lemma, we obtain that $f(t) \leq A(t) \, \exp(t \, E(t))$, which means that
\begin{align}\label{supi}
  \sup_{s \leq t} \|q(s)\|^p_{W^{2-\frac{2}{p}}_p(\Omega)} \leq c_p \, ((\phi^*_{2,t})^p \, \|q^0\|^p_{W^{2-\frac{2}{p}}_p(\Omega)} + (\phi^*_{1,t})^p \, \|g\|_{L^p(Q_t)}^p) \,\exp\left(t \, E(t)\right) \, .
\end{align}
Combining \eqref{supi} and \eqref{ControlProducts4} we obtain that
\begin{align}\label{supii}
  \frac{1}{2} \,  \mathscr{V}^p(t; \, q) \leq c_p \, [(\phi^*_{2,t})^p \, \|q^0\|^p_{W^{2-\frac{2}{p}}_p(\Omega)} + (\phi^*_{1,t})^p \, \|g\|_{L^p(Q_t)}^p] \, [1+ t\, E(t) \, \exp(t\, E(t))] \, .
\end{align}
Estimating $F^*(s) \leq (\sup_{s \leq t} \|q^*(s)\|_{W^{2-\frac{2}{p}}_p(\Omega)} + \|\nabla \varrho^*\|_{L^{p,\infty}(Q_t)})^p$, we obtain after another application of Young's inequality with power $1/\alpha$ that
\begin{align*}
E(t) & \leq  c_{p,\alpha} \, [(\phi^*_{2,t})^p + 1 +(\phi^*_{1,t})^{\tfrac{p}{1-\alpha}} \, (\sup_{s\leq t} \|q^*(s)\|_{W^{2-\frac{2}{p}}_p(\Omega)} + \|\nabla \varrho^*\|_{L^{p,\infty}(Q_t)})^{\frac{p}{1-\alpha}}] \, .
 \end{align*}
We raise both sides of \eqref{supii} to the power $1/p$. Recall also that $\alpha =\frac{1}{2}+\frac{3}{2p}$ and the claim follows.
\end{proof}
To conclude the section, we show how to obtain the estimate of Proposition \ref{A1linmain} as stated. We start from the Lemma \ref{A1linUMFORMsecond}, in which $\phi_{2,t}^* := \phi_{0,t}^* \, (1+ [\varrho^*]_{C^{\beta,\frac{\beta}{2}}(Q_t)} + [q^*]_{C^{\beta,\frac{\beta}{2}}(Q_t)})^{\frac{2}{\beta}}$. Moreover, Lemma \ref{A1linUMFORM} shows that $\phi^*_{0,t} = \phi^*_0((m^*(t))^{-1}, \, M^*(t), \, \|q^*\|_{L^{\infty}(Q_t)})$ is an increasing function of its arguments. 
First, we roughly estimate
\begin{align}\label{rough}
 \phi_{2,t}^*  \leq & \phi^*_{0,t} \, [1+ [q^*]_{C^{\beta,\frac{\beta}{2}}(Q_t)}]^{\frac{2}{\beta}}\,  [1+ [\varrho^*]_{C^{\beta,\frac{\beta}{2}}(Q_t)}]^{\frac{2}{\beta}} \, .
 \end{align}
Making use of the Lemma \ref{HOELDERlemma}, we can further estimate the quantities $[q^*]_{C^{\beta,\frac{\beta}{2}}(Q_t)}$ and $\|q^*\|_{L^{\infty}(Q_t)}$ in the latter expression. Define $\gamma$ as in Lemma \ref{HOELDERlemma}. For $0< \beta < \min\{1, \, 2-\frac{5}{p}\}$ it follows that
\begin{align}\label{hoeldercontrol}
 \|q^*\|_{C^{\beta,\frac{\beta}{2}}(Q_t)} \leq  C(t) \, (\|q^*(0)\|_{C^{\beta}(\Omega)} + t^{\gamma} \, \mathscr{V}(t; \, q^*)) \, ,
\end{align}
where we assume for simplicity $C(t) \geq 1$ in Lemma \ref{HOELDERlemma}. Using that $\phi^*$ is increasing in its arguments, we obtain with the help of \eqref{hoeldercontrol}
\begin{align*}
\phi^*_{0,t} \, [1+ [q^*]_{C^{\beta,\frac{\beta}{2}}(Q_t)}]^{\frac{2}{\beta}} \leq & \phi^*_0((m^*(t))^{-1}, \, M^*(t), \, C(t) \, ( \|q^*(0)\|_{C^{\beta}(\Omega)} + t^{\gamma} \, \mathscr{V}(t; \, q^*))) \times\\
& \times [1+ C(t) \, ( \|q^*(0)\|_{C^{\beta}(\Omega)} + t^{\gamma} \, \mathscr{V}(t; \, q^*) ) ]^{\frac{2}{\beta}}\, .
\end{align*}
The latter expression is next reinterpreted as a function $\phi^*_3$ of the arguments $t$, $(m^*(t))^{-1}$, $M^*(t)$, $\|q^*(0)\|_{C^{\beta}(\Omega)}$ and $\mathscr{V}(t; \, q^*)$, in that order. This means that for $t \geq 0$ and $a_1, \ldots, a_4 \geq 0$, we define
\begin{align*}
 \phi^*_3(t,\, a_1, \, \ldots, a_4)  = \phi^*_0(a_1, \, a_2,\,  \,  C(t) \, (a_3 + t^{\gamma} \, a_4)) \, [1+ C(t)\, ( a_3 + t^{\gamma} \, a_4 ) ]^{\frac{2}{\beta}} \, .
\end{align*}
This definition shows in particular that $\phi^*_3(0,\, a_1, \, \ldots, a_4) = \phi^*_0(a_1, \, a_2, \, C_0 \, a_3)\, (1+ C_0 \, a_3)^{\frac{2}{\beta}} $ is independent on $a_4$. Moreover, in view of \eqref{rough},
\begin{align}\label{pii}
  \phi_{2,t}^*  \leq  \phi^*_3(t,\, (m^*(t))^{-1}, \, M^*(t), \,  \|q^*(0)\|_{C^{\beta}(\Omega)}, \,  \mathscr{V}(t; \, q^*)) \, [1+ [\varrho^*]_{C^{\beta,\frac{\beta}{2}}(Q_t)}]^{\frac{2}{\beta}} \, .
\end{align}
We next invoke Lemma \ref{A1linUMFORMsecond}, where we have shown that
\begin{align}\label{piii}
\mathscr{V}(t; \, q) \leq c_1 \, [\phi^*_{2,t} \, \|q^0\|_{W^{2-\frac{2}{p}}_p(\Omega)} + \phi^*_{1,t} \, \|g\|_{L^p(Q_t)}] \, [1+ t^{\tfrac{1}{p}}\, B^*_t \, \exp(c_2 \, t\, [B^*_t]^p)] \, ,
\end{align}
and that $B^*_t = \phi^*_{2,t} + 1 +(\phi^*_{1,t})^{\tfrac{2p}{p-3}} \, (\sup_{s\leq t} \|q^*(s)\|_{W^{2-\frac{2}{p}}_p(\Omega)} + \|\nabla \varrho^*\|_{L^{p,\infty}(Q_t)})^{\tfrac{2p}{p-3}}$. Due to \eqref{pii}, we can bound $\phi_2^*$ and see that $t^{\frac{1}{p}} \, B^*_t$ is estimated by a function $\phi_4^*$ of the quantities $t$, $(m^*(t))^{-1}$, $M^*(t)$, $\|q^*(0)\|_{C^{\beta}(\Omega)}$, $\mathscr{V}(t; \, q^*)$, and $ [\varrho^*]_{C^{\beta,\frac{\beta}{2}}(Q_t)}$, $\|\nabla \varrho^*\|_{L^{p,\infty}(Q_t)}$, in that order. The function $\phi_4^*$ is defined via
\begin{align*}
\phi_4^*(t, \, a_1,\ldots,a_6) = t^{\frac{1}{p}} \,   [1+\phi_3^*(t, \, a_1,\ldots,a_4) \, (1+a_5)^{\frac{\beta}{2}} + \phi_{1}^*(a_1,a_2,a_3) \, (a_4 + a_6)^{\frac{2p}{p-3}}] \, .
\end{align*}
In particular, $\phi_4^*(0, \, a) = 0$. Moreover we obtain with the help of \eqref{piii} that
\begin{align*}
\mathscr{V}(t; \, q) \leq & c_1 \, (1+ \phi^*_{4,t} \, \exp(c_2 \, [\phi^*_{4,t}]^p)) \,  [\phi^*_{3,t} \, (1+[\varrho^*]_{C^{\beta,\frac{\beta}{2}}(Q_t)})^{\tfrac{2}{\beta}}  \, \|q^0\|_{W^{2-\frac{2}{p}}_p(\Omega)} + \phi^*_{1,t} \, \|g\|_{L^p(Q_t)}] \,  \, .
\end{align*}
We define $\Psi_{1,t} :=  (1+ \phi^*_{4,t} \, \exp(c_2 \, [\phi^*_{4,t}]^p)) \, \max\{\phi^*_{3,t}, \, \phi^*_{1,t}\}$. More precisely, for all non-negative numbers $a_1,\ldots,a_6$, we define
\begin{align*}
 \Psi_1(t, \, a_1,\ldots,a_6) := & (1+ \phi^*_{4}(t, \, a_1,\ldots,a_6) \, \exp(c_2 \, [\phi^*_{4}(t, \, a_1,\ldots,a_6)]^p)) \times\\
&  \times \max\{\phi^*_{3}(t,\, a_1, \ldots, a_4) , \, \phi^*_{1}(a_1, \, a_2, \, a_3) \} \, .
 \end{align*}
By means of the properties of $\phi_3^*$ and $\phi_4^*$, we verify that
 \begin{align}\label{psi0formula}
  \Psi_1(0, \, a_1,\ldots,a_6) & =  \max\{\phi^*_0(a_1, \, a_2, \, C_0 \, a_3)\, (1+ C_0 \, a_3)^{\frac{2}{\beta}}, \, \phi^*_{1}(a_1, \, a_2, \, a_3) \}\nonumber\\
  &=: \Psi^0_1(a_1, \, a_2, \, a_3) \, ,
 \end{align}
 which is independent of the arguments $(a_4, \, a_5, \, a_6)$. This is the claim of Proposition \ref{A1linmain}. We could further specify the function $\Psi^0_1$ by means of \eqref{LinftyFACTORS}.

\subsection{Estimates for linearised problems for the variables $v$ and $\varrho$}

After these technicalities we can rely on estimates already available. We first quote \cite{solocompress}, Theorem 1.
\begin{prop}\label{solonnikov1}
Suppose that $\varrho^* \in C^{\alpha,0}(Q_T)$, $0< \alpha \leq 1$, is strictly positive, and denote $M^*(t) := \max_{Q_t} \varrho^*$, $m^*(t) := \min_{Q_t} \varrho^*$.
Let $f \in L^p(Q_T; \, \mathbb{R}^3)$ and $v^0 \in W^{2-\frac{2}{p}}_p(\Omega; \,  \mathbb{R}^3)$ with $v^0 = 0$ on $\partial \Omega$.

Then there is a unique solution $v \in W^{2,1}(Q_T; \, \mathbb{R}^3)$ to $\varrho^* \, \partial_t v - \divv \mathbb{S}(\nabla v) = f$ in $Q_T$ with the boundary conditions $v = 0$ on $S_T$ and $v(x, \, 0) = v^0(x)$ in $\Omega$. Moreover, for all $t \leq T$,
\begin{align*}
\mathscr{V}(t; \, v) \leq & C \, \phi_{5,t}^* \, (1+\sup_{s\leq t} [\varrho(s)]_{C^{\alpha}(\Omega)})^{\frac{2}{\alpha}} \, (\|f\|_{L^p(Q_t)} + \|v^0\|_{W^{2-\frac{2}{p}}_p(\Omega)} + \|v\|_{L^p(Q_t)}) \, , \\
\phi_{5,t}^* := & \left(\frac{1}{\min\{1, \, m^*(t)\}}\right)^{\frac{2}{\alpha}} \, \left(\frac{M^*(t)}{m^*(t)} \right)^{\frac{p+1}{p}} \,  \, .\nonumber
\end{align*}
\end{prop}
With a Gronwall argument like in Lemma \ref{A1linUMFORMsecond}, we can get rid of $\|v\|_{L^p(Q_t)}$. 
\begin{coro}\label{normvmain}
Adopt the assumptions of Proposition \ref{solonnikov1}. Then there is $C$ independent on $t$, $\varrho^*$, $v^0$, $f$ and $v$, and a continuous function $\Psi_2 = \Psi_2(t, \, a_1, \, a_2, \, a_3)$, defined for all $t \geq 0$ and all $a_1, \, a_2, \, a_3 \geq 0$, such that
 \begin{align*}
\mathscr{V}(t; \, v) \leq & C \, \Psi_{2,t} \, (1+ \sup_{s\leq t} [\varrho^*(s)]_{C^{\alpha}(\Omega)})^{\frac{2}{\alpha}} \, (\|f\|_{L^p(Q_t)} + \|v^0\|_{W^{2-\frac{2}{p}}_p(\Omega)})\, , \nonumber\\
\Psi_{2,t} := & \Psi_2(t, \, (m^*(t)), \, M^*(t), \, \sup_{s\leq t} [\varrho^*(s)]_{C^{\alpha}(\Omega)}) \, .
\end{align*}
The function $\Psi_2$ is increasing in all its arguments and the value of $\Psi_2(0, \, a_1, \, a_2, \, a_3)  = \phi_{5}^*(a_1, \,a_2)$ is independent on $a_3$.
\end{coro}
\begin{proof}
Introduce first the abbreviation $\tilde{\phi}_{5,t}^* := \phi_{5,t}^* \, (1+\sup_{s\leq t} [\varrho(s)]_{C^{\alpha}(\Omega)})^{\frac{2}{\alpha}}$. We raise the estimate in Prop. \ref{solonnikov1} to the $p-$power and obtain that
\begin{align*}
 \mathscr{V}^p(t; \, v) \leq & c_p \,  (\tilde{\phi}_{5,t}^*)^p \, (\|f\|_{L^p(Q_t)}^p + \|v^0\|_{W^{2-\frac{2}{p}}_p(\Omega)}^p + \int_{0}^t \|v(s)\|^p_{L^p(\Omega)} \, ds) \, .
\end{align*}
We then argue as in Corollary \ref{A1linUMFORMsecond}, using Gronwall. We let $\Psi_{2,t} := \phi_{5,t}^* \, (1 + t^{\frac{1}{p}} \, \tilde{\phi}_{5,t}^* \, e^{C \, (\tilde{\phi}_{5,t}^*)^p \, t})$, and the claim follows.
\end{proof}
We recall one further Theorem of \cite{solocompress} concerning the linearised continuity equation.
\begin{prop}\label{solonnikov2}
 Suppose that $v^* \in W^{2,1}_p(Q_T; \, \mathbb{R}^3)$. Suppose that $\varrho_0 \in W^{1,p}(\Omega)$ satisfies $0 < m_0 \leq \varrho_0(x) \leq M_0 < + \infty$ in $\Omega$.
 Then the problem $\partial_t \varrho + \divv(\varrho \, v^*) = 0$ in $Q_T$ with $\varrho(x, \, 0) = \varrho_0(x)$ in $\Omega$ possesses a unique solution of class $W^{1,1}_{p,\infty}(Q_T)$ for which
 \begin{align*}
  m_0 \, [\phi_{6,t}^*]^{-1} \leq \varrho(x, \, t) \leq M_0 \, \phi_{6,t}^*  \text{ for all } (x, \, t) \in Q_T \, ,
 \end{align*}
with $\phi_{6,t}^* := e^{\sqrt{3} \, \|v^*_x\|_{L^{\infty,1}(Q_t)}}$. Moreover, for all $t < T$ and $0 < \alpha < 1$,
\begin{align*}
& \|\nabla \varrho(t)\|_{L^p(\Omega)} \leq \sqrt{3} \, [\phi_{6,t}^*]^{(2+\tfrac{1}{p})\sqrt{3}} \, (\|\nabla \varrho_0\|_{L^p(\Omega)} +\sqrt{3} \,  \|\varrho_0\|_{L^{\infty}(\Omega)} \, \|v_{x,x}\|_{L^{p,1}(Q_t)}) \, ,\\
& [\varrho(t)]_{C^{\alpha}(\Omega)} \leq 3^{\tfrac{\alpha}{2}}\, [\phi_{6,t}^*]^{1+\alpha} \, ( [\varrho_0]_{C^{\alpha}(\Omega)} + \sqrt{3} \,  \|\varrho_0\|_{L^{\infty}(\Omega)} \, \int_0^t [v^*_x(s)]_{C^{\alpha}(\Omega)} \, ds)  \, .
 \end{align*}
 For $\alpha < 1$, there is $c = c_{\alpha}$ such that for all $(x, \, t) \in Q_T$,
 \begin{align*}
  [\varrho(x)]_{C^{\alpha}(0,t)} \leq & c \, \|\varrho_0\|_{C^{\alpha}(\Omega)} \, \phi_{6,t}^*\, \Big(\|v_x^*\|_{L^{\infty,\frac{1}{1-\alpha}}(Q_t)}
  + (\|v^*\|_{L^{\infty}(Q_t)} \, \phi_{6,t}^*)^{\alpha} \, (1 + \int_0^t [v_x^*(\tau)]_{C^{\alpha}(\Omega)}\, d\tau)\Big)  \, .
 \end{align*}
\end{prop}
\begin{proof}
 The three first estimates are stated and proved explicitly in \cite{solocompress}. In order to estimate the time H\"older norm, we invoke the formula
 \begin{align*}
  \varrho(x, \, t) = \varrho_0(y(0; \, t,x)) \, \exp\left(-\int_{0}^t \divv v(y(\tau; \, t,x), \, \tau) \, d\tau \right) \, .
 \end{align*}
 We shall use the abbreviation $F(t) :=- \int_{0}^t \divv v(y(\tau; \, t,x), \, \tau) \, d\tau$. The map $\tau \mapsto y(\tau; \, t,x)$ is the unique characteristics through $(x, \, t)$ with speed $v^*$. Recall also the formula given between numbers (15) and (16) in \cite{solocompress} via
  \begin{align}\label{charactimederiv}
   |\partial_t y(\tau; \, t,x)| \leq \sqrt{3} \, \|v^*\|_{L^{\infty}(Q_t)} \, \phi_{6,t}^* \, .
   \end{align}
Making use of the latter, we show for $s < t$ that
 \begin{align*}
  |F(t) - F(s)| \leq & 3^{\frac{1}{2}} \, \int_s^t \|v_x(\tau)\|_{L^{\infty}(\Omega)} \, d\tau + \int_{0}^s |\divv v(y(\tau; \, t,x), \, \tau) - \divv v(y(\tau; \, s,x), \, \tau)| \, d\tau\\
  \leq & 3^{\frac{1}{2}} \, \int_s^t \|v_x(\tau)\|_{L^{\infty}(\Omega)} \, d\tau +  3 \,  \int_0^s [v_x(\tau)]_{C^{\alpha}(\Omega)} \, (\sup_{\sigma \in [s,t]} |y_t(\tau; \, \sigma, \, x)|)^{\alpha} \, d\tau \, (t-s)^{\alpha} \, \\
  \leq & 3^{\frac{1}{2}} \, (t-s)^{\alpha} \, (\|v_x\|_{L^{\infty,\frac{1}{1-\alpha}}(Q_t)} \, + 3^{\frac{1}{2}+\frac{\alpha}{2}} \, (\|v^*\|_{L^{\infty}(Q_t)} \, \phi_{6,t}^*)^{\alpha} \,   \int_0^t [v_x(\tau)]_{C^{\alpha}(\Omega)}\, d\tau) \, .
 \end{align*}
For $s <t$, we have
\begin{align*}
 \varrho(x, \, t) - \varrho(x, \, s) = (\varrho_0(y(0; \, t,x))  -\varrho_0(y(0; \, s,x))  \, e^{F(t)} +  \varrho_0(y(0; \, s,x)) \, (e^{F(t)} - e^{F(s)}) \,
\end{align*}
By means of \eqref{charactimederiv}, it follows that 
\begin{align*}
& |  \varrho(x, \, t) - \varrho(x, \, s)| \leq \\
& [\varrho_0]_{C^{\alpha}(\Omega)} \, (\sqrt{3} \, \|v^*\|_{L^{\infty}(Q_t)} \, \phi_{6,t}^*)^{\alpha} \, (t-s)^{\alpha} \, e^{F(t)}  + \|\varrho_0\|_{L^{\infty}(\Omega)} \, e^{\max\{|F(t)|, |F(s)|\}} \, |F(t)-F(s)|\\
& \leq  c \, (t-s)^{\alpha} \, \phi_{6,t}^* \, \|\varrho_0\|_{C^{\alpha}(\Omega)} \times \\
& \qquad \times (\|v_x\|_{L^{\infty,\frac{1}{1-\alpha}}(Q_t)} + (\|v^*\|_{L^{\infty}(Q_t)} \, \phi_{6,t}^*)^{\alpha} \,  (1+ \int_0^t [v_x(\tau)]_{C^{\alpha}(\Omega)}\, d\tau )) \, .
\end{align*}
\end{proof}
We also need to restate these estimates as to be later able to quote them more conveniently.
\begin{coro}\label{normrhomain}
We adopt the assumptions of Proposition \ref{solonnikov2}. Define $m(t) := \inf_{Q_t} \varrho$ and $M(t) := \sup_{Q_t} \varrho$. We choose $\beta = 1-\tfrac{3}{p}$. Then there are functions $\Psi_3, \, \Psi_4,\Psi_5$ of the variables $t$, $\frac{1}{m_0}$, $M_0$, $\|\nabla \varrho_0\|_{L^{p}(\Omega)}$ and $\mathscr{V}(t; \, v^*)$, in that order, such that
\begin{align*}
 \frac{1}{m(t)} \leq & \Psi_3(t,  \, m_0^{-1}, \, \mathscr{V}(t; \, v^*)), \quad  M(t) \leq \Psi_3(t,\,  M_0, \, \mathscr{V}(t; \, v^*)) \, ,\\
 \|\nabla \varrho\|_{L^{p,\infty}(Q_t)} \leq & \Psi_4(t,  \, m_0^{-1}, \,M_0,\,  \|\nabla \varrho_0\|_{L^{p}(\Omega)}, \, \mathscr{V}(t; \, v^*)) \, ,\\
 [\varrho]_{C^{\beta,\frac{\beta}{2}}(Q_t)} \leq & \Psi_5(t,  \, m_0^{-1}, \,M_0,\,  \|\nabla \varrho_0\|_{L^{p}(\Omega)}, \, \mathscr{V}(t; \, v^*)) \, .
\end{align*}
%
%
%
For $i=3,4,5$, the function $\Psi_i$ is continuous and increasing in all variables. Moreover the expression $\Psi_i(0, \, a_1, \ldots, a_4) = \Psi^0_i(a_1, \, a_2, \,a_3)$ is independent on the last variable $a_4 = \mathscr{V}(t; \, v^*)$. (The function $\Psi_3$ depends only on $t$, $a_1$ or $a_2$ and $a_4$).
\end{coro}
\begin{proof}
The Sobolev embedding theorems imply that $\|v^*_x\|_{L^{\infty}(\Omega)} \leq C \, \|v^*\|_{W^{2,p}(\Omega)}$. It therefore follows from H\"older's inequality that
$\|v^*_x\|_{L^{\infty,1}(Q_t)} \leq C \, t^{1-\frac{1}{p}} \, \|v^*\|_{W^{2,1}_p(Q_t)}$. 

Thus $\phi_{6,t}^* \leq \exp(C \, t^{1-\frac{1}{p}} \, \|v^*\|_{W^{2,1}_p(Q_t)})$. Invoking the Proposition \ref{solonnikov2},
\begin{align*}
 \frac{1}{m(t)} \leq \frac{1}{m_0} \,  \exp(\sqrt{3} \, \|v^*_x\|_{L^{\infty,1}(Q_t)}) \leq &  \frac{1}{m_0} \,  \exp(C \, t^{1-\frac{1}{p}} \, \|v^*\|_{W^{2,1}_p(Q_t)}) \\
 := & \Psi_3(t, \,  m_0^{-1}, \, \mathscr{V}(t; \, v^*))  \, .
\end{align*}
Thus, choosing $\Psi_3(t, \, a_1,\, a_4) := a_1 \,  \exp(C \, t^{1-\frac{1}{p}} \,a_4)$, we have $ \frac{1}{m(t)} \leq \Psi_3$ and $\Psi_3(0, \, a_1, \,  a_4) = a_1$ is independent on $a_4$. Similarly $M(t) \leq M_0 \, \exp(C \, t^{1-\frac{1}{p}} \, \|v^*\|_{W^{2,1}_p(Q_t)})$. Moreover, again due to the Proposition \ref{solonnikov2},
\begin{align*}
&  \|\nabla \varrho(t)\|_{L^p(\Omega)} \leq \sqrt{3} \, [\phi_{6,t}^*]^{(2+\frac{1}{p})\sqrt{3}} \, (\|\nabla \varrho_0\|_{L^p(\Omega)} +\sqrt{3} \,  \|\varrho_0\|_{L^{\infty}(\Omega)} \, \|v_{x,x}\|_{L^{p,1}(Q_t)})\\
& \leq \sqrt{3} \, [\phi_{6,t}^*]^{(2+\frac{1}{p})\sqrt{3}} \, (\|\nabla \varrho_0\|_{L^p(\Omega)} + \|\varrho_0\|_{L^{\infty}(\Omega)} \, \sqrt{3} \, t^{1-\frac{1}{p}} \, \|v_{x,x}\|_{L^{p}(Q_t)})\\
& \leq \sqrt{3} \, \exp(C \, (2+\frac{1}{p})\sqrt{3} \, t^{1-\frac{1}{p}} \, \|v^*\|_{W^{2,1}_p(Q_t)}) \, (\|\nabla \varrho_0\|_{L^p(\Omega)} + \|\varrho_0\|_{L^{\infty}(\Omega)} \, \sqrt{3} \, t^{1-\frac{1}{p}} \, \|v_{x,x}\|_{L^{p}(Q_t)}) \, .
\end{align*}
We define $\Psi_4(t, \, a_1, \ldots, a_4) :=  \exp(C \, (2+\tfrac{1}{p})\sqrt{3} \, t^{1-\tfrac{1}{p}} \, a_4) \, (a_3 + a_2 \, \sqrt{3} \, t^{1-\tfrac{1}{p}} \, a_4)$. Then we clearly can show that $ \|\nabla \varrho(t)\|_{L^p(\Omega)} \leq\sqrt{3} \,  \Psi_4(t, \, m_0^{-1}, \, M_0, \, \|\nabla \varrho_0\|_{L^p}, \, \mathscr{V}(t; \, v^*))$. As required, the value of $\Psi_4(0, \, a_1, \ldots,a_4) = a_3$ is independent on $a_4$.

For $\alpha = 1-\tfrac{3}{p}$, the Sobolev embedding yields $\|v^*_x(t)\|_{C^{\alpha}(\Omega)} \leq C \,  \|v^*(t)\|_{W^{2,p}(\Omega)}$. Thus
\begin{align*}
&  [\varrho(t)]_{C^{\alpha}(\Omega)} \leq 3^{\frac{\alpha}{2}}\, [\phi_{6,t}^*]^{1+\alpha} \, ( [\varrho_0]_{C^{\alpha}(\Omega)} + \sqrt{3} \,  \|\varrho_0\|_{L^{\infty}(\Omega)} \, \int_0^t [v^*_x(s)]_{C^{\alpha}(\Omega)} \, ds)  \,\\
&  \leq 3^{\frac{\alpha}{2}}\, \exp((1+\alpha) \, C \, t^{1-\frac{1}{p}} \, \|v^*\|_{W^{2,1}_p(Q_t)}) \, ( [\varrho_0]_{C^{\alpha}(\Omega)} + \sqrt{3} \,  \|\varrho_0\|_{L^{\infty}(\Omega)} \,  C \, t^{1-\frac{1}{p}} \, \|v^*\|_{W^{2,1}_{p}(Q_t)}) \, .
 \end{align*}
 Moreover $ [\varrho_0]_{C^{\alpha}(\Omega)} \leq C \, \|\nabla \varrho_0\|_{L^p(\Omega)}$. Thus for $\Psi^{(1)}_5(t, \, a_1, \ldots,a_4) = \exp((1+\alpha)\, C \, t^{1-\frac{1}{p}} \, a_4) \, (a_3 + a_2 \, a_4\, t^{1-\frac{1}{p}})$, we find $[\varrho(t)]_{C^{\alpha}(\Omega)} \leq C \,\Psi_5^{(1)}(t, \, m_0^{-1}, \, M_0, \, \|\nabla \varrho_0\|_{L^p}, \, \mathscr{V}(t; \, v^*))$. Note that $\Psi^{(1)}_5(0) = a_3$. For $x \in \Omega$ and $\alpha < 1$, we can invoke Proposition \ref{solonnikov2} to estimate $[\varrho(x)]_{C^{\alpha}(0,t)}$.
 If $\alpha < 1-\frac{1}{p}$, H\"older's inequality implies for $r = \frac{(1-\alpha) \, p-1}{(1-\alpha) \, p} > 0$
 \begin{align*}
  \|v_x^*\|_{L^{\infty,\frac{1}{1-\alpha}}(Q_t)} \leq C \, \|v^*\|_{L^{\frac{1}{1-\alpha}}(0,t; \, W^{2,p}(\Omega))} \leq C \, t^r \, \|v^*\|_{W^{2,1}_p(Q_t)} \, .
 \end{align*}
Moreover, for all $\alpha \leq 1-\frac{3}{p}$, we have $\int_0^t [v_x^*(\tau)]_{C^{\alpha}(\Omega)}\, d\tau  \leq C \, t^{1-\frac{1}{p}} \,  \|v^*\|_{W^{2,1}_p(Q_t)}$. For $\alpha = 1-\frac{3}{p}$ we define $r = \frac{2}{3}$ and we see that
\begin{align*}
&  [\varrho(x)]_{C^{\alpha}(0,t)} \leq c \,  \|\varrho_0\|_{C^{\alpha}(\Omega)} \, \exp(C \,  t^{1-\frac{1}{p}} \,\|v^*\|_{W^{2,1}_p(Q_t)}) \, \Big(C \, t^{\frac{2}{3}} \, \|v^*\|_{W^{2,1}_p(Q_t)}\\
& \phantom{[\varrho(x)]_{C^{\alpha}(0,t)}  } \quad  + \|v^*\|_{L^{\infty}(Q_t)}^{\alpha} \, \exp(C \, \alpha\, t^{1-\frac{1}{p}} \,\|v^*\|_{W^{2,1}_p(Q_t)}) \, (1+ C \, t^{1-\frac{1}{p}} \,  \|v^*\|_{W^{2,1}_p(Q_t)}) \Big) \, .
\end{align*}
%
We can estimate $\|v^*\|_{L^{\infty}(Q_t)} \leq \sup_{s \leq t} \|v^*\|_{W^{2-\frac{2}{p}}_p(\Omega)}$ and $\|\varrho_0\|_{C^{\alpha}(\Omega)} \leq M_0 + C \, \|\nabla \varrho_0\|_{L^p(\Omega)}$. Then we define
\begin{align*}
  & \Psi_5^{(2)}(t, \, a_1, \ldots, a_4) \\
  & \quad = (a_3+ C \, a_2) \, \exp(C\, t^{1-\tfrac{1}{p}} \, a_4) \, (t^{\frac{2}{3}} \, a_4 + a_4^\alpha \, \exp( C \, \alpha \, t^{1-\tfrac{1}{p}} \,a_4) \, (1 +  C \, t^{1-\tfrac{1}{p}} \,  a_4))  \, ,
\end{align*}
and find that $ [\varrho(x)]_{C^{\alpha}(0,t)} \leq  \Psi^{(2)}_5$. The value $\Psi^{(2)}_5(0, \, a_1,\ldots,a_4) = (a_3 + C \, a_2) \, a_4^{\alpha}$ is not yet independent of $a_4$. But for arbitrary $0< \alpha^{\prime} < 1-\tfrac{3}{p}$, it also follows that $[\varrho(x)]_{C^{\alpha^{\prime}}(0,t)} \leq C \,  t^{1-\tfrac{3}{p}-\alpha^{\prime}} \, \Psi_5^{(2)}$.

The function $\Psi^{(3)}_5(t, \, a_1, \ldots, a_4) := t^{1-\frac{3}{p}-\alpha^{\prime}} \, \Psi_5^{(2)}(t, \, a_1, \ldots,a_4)$ now satisfies $\Psi_5^{(3)}(0, \, a) = 0$ independently on $a_4$. Moreover $[\varrho(x)]_{C^{\alpha^{\prime}}(0,t)} \leq \Psi_5^{(3)}$.
Thus for $\beta = 1-\tfrac{3}{p}$, we have $[\varrho]_{C^{\beta,\frac{\beta}{2}}(Q_T)} \leq  C  \, (\Psi_5^{(1)} + \Psi_5^{(3)}) =: C \, \Psi_5$. Clearly $\Psi_5(0, \, a_1,\ldots, a_4) = \Psi_5^{(1)}(0, \, a_1, \ldots,a_4) = a_3$. We are done.
\end{proof}

\section{The continuity estimate for $\mathcal{T}$}\label{contiT}

We now want to combine the Propositions \ref{A1linmain} and \ref{solonnikov1} with the linearisation of the continuity equation in Proposition \ref{solonnikov2} to study the fixed point map $\mathcal{T}$ described at the beginning of Section \ref{twomaps} and defined by the equations \eqref{linearT1}, \eqref{linearT2}, \eqref{linearT3}. For a given $v^* \in W^{2,1}_p(Q_T; \, \mathbb{R}^3)$ and $q^* \in W^{2,1}_p(Q_T; \, \mathbb{R}^{N-1})$, we introduce the notation $\mathscr{V}^*(t) := \mathscr{V}(t; \, q^*) +  \mathscr{V}(t; \, v^*)$. To begin with, we need to control the growth of the lower order terms in \eqref{linearT1}, \eqref{linearT2}, \eqref{linearT3}.
\begin{lemma}\label{RightHandControl}
Consider the compound $\Gamma := Q \times \mathbb{R}^{N-1} \times \mathbb{R}_+ \times \mathbb{R}^3 \times \mathbb{R}^{N-1\times 3} \times \mathbb{R}^{3\times 3} \times \mathbb{R}^{3}$, and a function $G$ defined on $\Gamma$. For $u^* = (q^*, \, \varrho^*, \,  v^*) \in \mathcal{X}_{T,+}$ we define $G^* := G(x, \, t, \, u^*,\, D^1_xu^*)$. Assume that the function $G$ is satisfying for all $t \leq T$ and $x \in \Omega$ the growth conditions
\begin{align*}
& | G(x, \, t, \, u^*,\, D^1_xu^*)| \\
& \quad \leq c_1((m^*(t))^{-1}, \, |u^*|) \, \Big(|\bar{G}(x, \, t)| + |q^*_x|^{r_1} + |v^*_x|^{r_1} + |\bar{H}(x ,\, t)| \,  (|\varrho^*_x| + |q^*_x|)\Big)   \, ,
\end{align*}
in which $1\leq r_1 < 2-\frac{3}{p} + \frac{3}{(5-p)^+}$ and $\bar{G} \in L^p(Q_T)$, $\bar{H} \in L^{\infty,p}(Q_T)$ are arbitrary and $c_1$ is a continuous, increasing function of two positive arguments. Then, there is a continuous function $\Psi_{G} = \Psi_G(t, \, a_1, \ldots, a_5)$ defined for all non-negative arguments such that
\begin{align*}
 \|G^*\|_{L^p(Q_t)} \leq \Psi_{G}(t, \, (m^*(t))^{-1}, \, M^*(t), \, \|(q^*(0), \, v^*(0))\|_{W^{2-\frac{2}{p}}_p(\Omega)}, \, \|\nabla \varrho^*\|_{L^{p,\infty}(Q_t)}, \, \mathscr{V}^*(t)) \, .
\end{align*}
The function $\Psi_{G}$ is increasing in all arguments and $\Psi_{G}(0, \, a_1,\ldots,a_5) = 0$ for all $a \in [\mathbb{R}_+]^5$.
\end{lemma}
\begin{proof}
With the abbreviation $c_1^* := c_1((m^*(t))^{-1}, \, \|q^*\|_{L^{\infty}(Q_t)} + \|v^*\|_{L^{\infty}(Q_t)} + M^*(t))$, we have by assumption
\begin{align*}
 \|G^*\|_{L^p(Q_t)} \leq & c_1^* \, (\|\bar{G}\|_{L^p(Q_t)} + \||q^*_x| + |v^*_x|\|_{L^{pr_1}(Q_t)}^{r_1} + \|(|\varrho^*_x| + |q^*_x|) \, \bar{H}\|_{L^p(Q_t)})\\
  \leq & c_1^* \, (\|\bar{G}\|_{L^p(Q_t)} + \||q^*_x| + |v^*_x|\|_{L^{pr_1}(Q_t)}^{r_1} + \||q^*_x| + |\varrho^*_x|\|_{L^{p,\infty}(Q_t)} \, \|\bar{H}\|_{L^{\infty,p}(Q_t)}) \, .
\end{align*}
Thanks to the Remark \ref{parabolicspace}, the gradients $q^*_x, \, v^*_x$ belong to $L^r(Q_T)$ for $r = 2p-3+ \frac{3p}{(5-p)^+}$, and for all $t \leq T$ the inequality $\|q^*_x\|_{L^r(Q_t)} \leq \|q^*_x\|_{L^{\infty,2p-3}(Q_t)}^{\frac{2p-3}{r}} \, \|q^*_x\|_{L^{\frac{3p}{(5-p)^+},\infty}(Q_t)}^{\frac{3p}{r(5-p)^+}}$ is valid.

Thus, if $r_1 \, p < r$, we obtain that $\|q^*_x\|_{L^{pr_1}(Q_t)} \leq C \, t^{1-\frac{pr_1}{r}} \, |\Omega|^{1-\frac{pr_1}{r}} \, \mathscr{V}(t; \, q)$.
Moreover, since $p < \frac{3p}{(5-p)^+}$, we have $\|q^*\|_{L^{p, \infty}(Q_t)} \leq C \, \sup_{s \leq t} \|q^*(s)\|_{W^{2-\frac{2}{p}}(\Omega)}$ with $C$ depending only on $\Omega$. The terms containing $v^*_x$ are estimated the same way.
Overall, we obtain for $G^*$
\begin{align*}
 \|G^*\|_{L^p(Q_t)} \leq C_1^* \, (&\|\bar{G}\|_{L^p(Q_t)} + t^{r_1\, (1-\frac{pr_1}{r})} \, [\mathscr{V}^*(t)]^{r_1}+\|\bar{H}\|_{L^{\infty,p}(Q_t)} \, (\|\varrho^*_x\|_{L^{p,\infty}(Q_t)} + \mathscr{V}^*(t)) ) \, ,
\end{align*}
in which $C_1^* = C_1((m^*(t))^{-1}, \, \|q^*\|_{L^{\infty}(Q_t)} + \|v^*\|_{L^{\infty}(Q_t)} + M^*(t))$. Invoking the Lemma \ref{HOELDERlemma} to estimate $\|q^*\|_{L^{\infty}(Q_t)} \leq \|q^0\|_{L^{\infty}(\Omega)} + t^{\gamma} \, \mathscr{V}^*(t)$ and the same for $v^*$, we see that this estimate possesses the structure claimed by the Lemma.
\end{proof}
We are now ready to establish the a final estimate that allows to obtain the self-mapping property.
\begin{prop}\label{estimateself}
There is a continuous function $\Psi_8 = \Psi_8(t, \, a_1,\ldots, a_5)$ on $[0, \, + \infty[ \times \mathbb{R}^5_+$, increasing in all arguments, such that for the pair $(q, \, v) = \mathcal{T}(q^*, \, v^*)$ the following estimate is valid:
\begin{align*}
 \mathscr{V}(t) \leq \Psi_8(t, \, m_0^{-1}, \,  M_0, \, \|\nabla \varrho_0\|_{L^p(\Omega)}, \,  \|(q^0, \, v^0)\|_{W^{2-\frac{2}{p}}_p(\Omega)}, \, \mathscr{V}^*(t)) \, .
\end{align*}
Moreover, for all $\eta \geq 0$
\begin{align*}
& \Psi_8\Big(0, \,  m_0^{-1}, \,  M_0, \, \|\nabla \varrho_0\|_{L^p(\Omega)}, \,  \|q^0\|_{W^{2-\frac{2}{p}}_p(\Omega)},\, \|v^0\|_{W^{2-\frac{2}{p}}_p(\Omega)}, \, \eta\Big) \\
& =
 \Psi^0_1(m_0^{-1}, \, M_0, \, \|q^0\|_{C^{1-\frac{3}{p}}(\Omega)}) \, (1+\|\nabla \varrho_0\|_{L^p(\Omega)})^{\tfrac{2p}{p-3}} \, \|q^0\|_{W^{2-\frac{2}{p}}_p(\Omega)}\\
& +
\left( \frac{1}{\min\{1, \, m_0\}} \right)^{\tfrac{2p}{p-3}}\, \left( \frac{M_0}{m_0} \right)^{\tfrac{p+1}{p}} \, (1+\|\nabla \varrho_0\|_{L^p(\Omega)})^{\tfrac{2p}{p-3}} \, \|v^0\|_{W^{2-\frac{2}{p}}_p(\Omega)} \, .
\end{align*}
\end{prop}
\begin{proof}
We first apply the Proposition \ref{A1linmain} with $\varrho^* = \varrho$. It follows that 
\begin{align*}
  \mathscr{V}(t; \, q) \leq & C \, \Psi_1(t, \, m(t)^{-1}, \, M(t),  \, \|q^*(0)\|_{ C^{\beta}(\Omega)}, \, 
\mathscr{V}(t; \, q^*), \, [\varrho]_{C^{\beta,\frac{\beta}{2}}(Q_t)}, \, \|\nabla \varrho\|_{L^{p,\infty}(Q_t)}) \,  \times\\ 
& \times ((1+[\varrho]_{C^{\beta,\frac{\beta}{2}}(Q_t)})^{\tfrac{2}{\beta}} \,   \|q^0\|_{W^{2-\frac{2}{p}}_p(\Omega)} + \|g^*\|_{L^p(Q_t)}) \, . 
\end{align*}
Due to the Corollary \ref{normrhomain} and the choice $\varrho^* = \varrho$, we have for $\beta := 1-\frac{3}{p}$
\begin{align*}
\max\{ \frac{1}{m(t)}, \, M(t)\} \leq & \Psi_3(t, \, \max\{m_0^{-1}, \, M_0\}, \, \mathscr{V}(t; \, v^*)) =: \Psi_3(t, \ldots) \, , \\
  \|\nabla \varrho^*\|_{L^{p,\infty}(Q_t)} \leq & \Psi_4(t, \, m_0^{-1}, \, M_0, \, \, \|\nabla \varrho_0\|_{L^p(\Omega)}, \, \mathscr{V}(t; \, v^*)) =: \Psi_4(t, \ldots)\\
 [\varrho^*]_{C^{\beta,\frac{\beta}{2}}(Q_t)}  \leq & \Psi_5(t, \, m_0^{-1}, \, M_0, \, \, \|\nabla \varrho_0\|_{L^p(\Omega)}, \, \mathscr{V}(t; \, v^*)) =: \Psi_5(t, \ldots) \, .
 \end{align*}
Moreover, we can apply the Lemma \ref{RightHandControl} to the right-hand defined in \eqref{A1right}. (Choose $G = g$, $r_1 = 1$, $\bar{H}(x, \, t) := |\tilde{b}(x, \, t)|$ and $\bar{G}(x, \, t) = |\tilde{b}_x(x, \, t)|$.) It follows that
\begin{align*}
 \|g^*\|_{L^p(Q_t)} \leq & \Psi_g( t, \,\Psi_3(t, \, \ldots), \, \Psi_3(t, \, \ldots), \, \|(q^0, \, v^0)\|_{W^{2-\frac{2}{p}}_p(\Omega)},\, \Psi_4(t, \, \ldots), \, \mathscr{V}^*(t))\,\\
  =: & \Psi_g(t, \ldots) \, .
\end{align*}
Combining all these estimates we can bound the quantity $ \mathscr{V}(t; \, q) $ by the function
\begin{align*}
\Psi_8^{(1)} := & \Psi_1( t, \,  \Psi_3(t, \, \ldots)   , \,\Psi_3(t, \, \ldots),  \, \|q^0\|_{C^{\beta}(\Omega)}, \, 
\mathscr{V}(t; \, q^*), \,  \Psi_5(t, \,\ldots), \, \Psi_4(t, \, \ldots)) \, \times \\
& \times \Big( (1+\Psi_5(t, \, \ldots))^{\tfrac{2}{\beta}} \, \|q^0\|_{W^{2-\frac{2}{p}}_p(\Omega)} + \Psi_g(t, \ldots) \Big) \, .
\end{align*}
Since we can apply the inequalities $\mathscr{V}(t; \, v^*), \, \mathscr{V}(t; \, q^*) \leq \mathscr{V}(t)$, we reinterpret the latter expression as a function $\Psi^{(1)}_{8} := \Psi^{(1)}_8(t, \, m_0^{-1}, \,  M_0, \, \|\nabla \varrho_0\|_{L^p(\Omega)}, \,  \|(q^0, \, v^0)\|_{W^{2-\frac{2}{p}}_p(\Omega)}, \, \mathscr{V}(t))$.

Moreover, it $t = 0$, we can use the estimates proved in the Proposition \ref{A1linmain} and the Corollaries \ref{normvmain} and \ref{normrhomain}. Recall in particular that $\Psi_1(0, \, a_1, \ldots, a_6) = \Psi_1^0(a_1, \, a_2, \, a_3)$. Moreover, $\Psi_3(t, \, M_0, \, \eta) = M_0$. Thus, since $\Psi_5(0, \, a_1, \, a_2, \, a_3) = a_3$, and $\Psi_g(0, \ldots) = 0$ (see Lemma \ref{RightHandControl}) we can compute that
\begin{align}\label{formulatpsi6null}
& \Psi^{(1)}_8(0, \, m_0^{-1}, \,  M_0, \, \|\nabla \varrho_0\|_{L^p(\Omega)}, \,  \|(q^0, \, v^0)\|_{W^{2-\frac{2}{p}}_p(\Omega)}, \, \mathscr{V}(t)) \\
& \quad= \Psi^0_1(m_0^{-1}, \, M_0, \, \|q^0\|_{C^{1-\frac{3}{p}}(\Omega)}) \, (1+\|\nabla \varrho_0\|_{L^p(\Omega)})^{\tfrac{2p}{p-3}} \, \|q^0\|_{W^{2-\frac{2}{p}}_p(\Omega)} \nonumber\, .
\end{align}
We next apply the Corollary \ref{normvmain} with $\varrho^* = \varrho$ and $f = f^*$, to obtain that
\begin{align*}
   \mathscr{V}(t; \, v) \leq & C \, \Psi_2(t, \, m(t)^{-1}, \, M(t), \, \sup_{s \leq t} [\varrho(s)]_{C^{\alpha}(\Omega)}) \, (1+\sup_{s \leq t} [\varrho(s)]_{C^{\alpha}(\Omega)})^{\frac{2}{\alpha}}\, \times\\
   & \times ( \|v^0\|_{W^{2-\frac{2}{p}}_p(\Omega)}+ \|f^*\|_{L^p(Q_t)})   \, .
\end{align*}
We apply the Lemma \ref{RightHandControl} to $G = f$ (recall \eqref{A3right}, and choose $r_1 = 1$, $|\bar{H}(x, \, t)| = 1$ and $\bar{G}(x, \, t) := |\tilde{b}(x, \, t)| + |\bar{b}(x, \, t)|$ in the statement of Lemma \ref{RightHandControl}). For $\alpha = 1-\frac{3}{p}$, we estimate $\mathscr{V}(t; \, v)$ above by
\begin{align*}
\Psi_2(& t, \, \Psi_3(t, \, \ldots) \,, \Psi_3(t, \, \ldots), \, \Psi_5(t, \, \ldots)) \, (1+ \Psi_5(t, \, \ldots))^{\frac{2}{\alpha}}\, (\|v^0\|_{W^{2-\frac{2}{p}}_p(\Omega)}+ \Psi_f(t,\ldots) ) \, .  
\end{align*}
We reinterpret this function as a $\Psi^{(2)}_{8,t}$ of the same arguments, and we note that
\begin{align*}
 & \Psi^{(2)}_8(0, \, m_0^{-1}, \,  M_0, \, \|\nabla \varrho_0\|_{L^p(\Omega)}, \,  \|(q^0, \, v^0)\|_{W^{2-\frac{2}{p}}_p(\Omega)}, \, \eta) \\
 & = \Psi^0_2(m_0^{-1}, \, M_0, \, \|\nabla \varrho_0\|_{L^p(\Omega)}) \, (1+\|\nabla \varrho_0\|_{L^p(\Omega)})^{\tfrac{2p}{p-3}} \, \|v^0\|_{W^{2-\frac{2}{p}}_p(\Omega)} \\
 & = \left( \frac{1}{\min\{1, \, m_0\}} \right)^{\tfrac{2p}{p-3}}\, \left( \frac{M_0}{m_0} \right)^{\tfrac{p+1}{p}} \, (1+\|\nabla \varrho_0\|_{L^p(\Omega)})^{\frac{2p}{p-3}} \, \|v^0\|_{W^{2-\frac{2}{p}}_p(\Omega)}
\end{align*}
The claim follows.
\end{proof}
\begin{prop}\label{ContiEststate}
We adopt the assumptions of the Theorem \ref{MAIN2}.
For a given pair $(q^*, \, v^*) \in \mathcal{Y}_T$, we define a map $\mathcal{T}(q^*, \, v^*) = (q, \, v)$ via solution to the equations \eqref{linearT1}, \eqref{linearT2}, \eqref{linearT3} with homogeneous boundary conditions \eqref{lateralq}, \eqref{lateralv} and initial conditions $(q^0, \, \varrho_0, \, v^0)$. Then, there are $0 < T_0 \leq T$ and $\eta_0 > 0$ depending on the data via the vector $R_0 := (m_0^{-1},\, M_0, \, \|\nabla \varrho_0\|_{L^p(\Omega)}, \, \|(q^0, \, v^0)\|_{W^{2-\frac{2}{p}}_p(\Omega)})$ such that $\mathcal{T}$ maps the ball with radius $\eta_0$ in $\mathcal{Y}_{T_0}$ into itself.
\end{prop}
\begin{proof}
We apply the Lemma \ref{selfmapT} with $\Psi(t, \, R_0, \, \eta) := \Psi_8(t, \, R_0, \, \eta)$ from Lemma \ref{estimateself}, and the claim follows. 
\end{proof}

\section{Fixed point argument and proof of the theorem on short-time well-posedness}\label{FixedPointIter}

Starting from $(q^1, \, v^1) = 0$, we consider a fixed point iteration $(q^{n+1}, \, v^{n+1}) := \mathcal{T} \, (q^{n}, \, v^{n})$ for $ n \in \mathbb{N}$. 

Recall that this means first considering $\varrho^{n+1} \in W^{1,1}_{p,\infty}(Q_T)$ solution to
\begin{align*}
\partial_t \varrho^{n+1} + \divv(\varrho^{n+1} \, v^n) = 0 \text{ in } Q_T\, , \quad \varrho^{n+1}(x, \, 0) = \varrho_0(x) \text{ in } \Omega \, .
\end{align*}
Then we introduce $(q^{n+1}, \, v^{n+1}) \in W^{2,1}_p(Q_T; \, \mathbb{R}^{N-1}) \times W^{2,1}_p(Q_T; \, \mathbb{R}^{3})$ via solution in $Q_{T}$ to
\begin{align*}
& R_q(\varrho^{n+1}, \, q^n) \, \partial_t q^{n+1} - \divv (\widetilde{M}(\varrho^{n+1}, \, q^n) \, \nabla q^{n+1}) = - \divv (\widetilde{M}(\varrho^{n+1}, \, q^n) \, \tilde{b}(x, \,t) ) \\
& \quad +(R_{\varrho}(\varrho^{n+1}, \, q^n) \, \varrho^{n+1} - R(\varrho^{n+1}, \, q^n)) \, \divv v^n - R_q(\varrho^{n+1}, \, q^n) \, v^n \cdot \nabla q^n + \tilde{r}(\varrho^{n+1}, \, q^n)\, , \nonumber\\[0.2cm]
& \varrho^{n+1} \, \partial_t v^{n+1}  - \divv \mathbb{S}(\nabla v^{n+1}) = - \nabla P(\varrho^{n+1}, \, q^n) - \varrho^{n+1} \, (v^n \cdot \nabla) v^n \nonumber\\
& \quad \phantom{\varrho^{n+1} \, \partial_t v^{n+1}  - \divv \mathbb{S}(\nabla v^{n+1}) = \, } + \tilde{b}(x, \,t) \cdot R(\varrho^{n+1}, \, q^n) + \varrho^{n+1}\, \bar{b}(x, \,t) \, .
\end{align*}
with boundary conditions $\nu \cdot \nabla q^{n+1} = 0$, $v^{n+1} = 0$ on $S_T$ and initial data $q^{n+1}(x, \, 0) = q^0(x)$ and $v^{n+1}(x, \, 0) = v^0(x)$ in $\Omega$.
Recalling \eqref{Vfunctor}, we define $\mathscr{V}^{n+1}(t) := \mathscr{V}(t; \, q^{n+1}) + \mathscr{V}(t; \, v^{n+1})$. Since obviously $\mathscr{V}^1(t) \equiv 0$, the Prop. 
\ref{ContiEststate} implies the existence of parameters $T_0, \, \eta_0 > 0$ such that there holds uniform estimates
\begin{align}\label{uniform}
\sup_{n \in \mathbb{N}} \mathscr{V}^{n}(T_0) \leq \eta_0 \, , \quad  \sup_{n \in \mathbb{N}} \|\varrho_{n}\|_{W^{1,1}_{p,\infty}(Q_{T_0})} \leq C_0 \, .
\end{align}
In the Theorem \ref{iter} below, we obtain that the fixed-point iteration yields strongly convergence subsequences in $L^2(Q_{t,t+t_1})$ for the components of $q^n$, $\varrho_n$ and $v^n$ and the gradients $q^n_x$ and $v^n_x$. Here $0 < t_1 \leq T_0$ is a fixed number and $t \in [0, \, T_0-t_1]$ is arbitrary. Thus, we obtain the convergence in $L^2(Q_{T_0})$ of these functions. The passage to the limit in the approximation scheme is then a straightforward exercise, since we can rely on a uniform bound in $\mathcal{X}_{T_0}$. This step shall therefore be spared.

We next prove sufficient convergence properties of the sequence $\{(q^{n}, \, \varrho^n, \, v^n)\}_{n\in \mathbb{N}}$ by means of contractivity estimates in a lower--order space. This estimate also guarantees the uniqueness. The proof is unfortunately lengthy due to the complex form of the PDE system, but it is elementary in essence and might be skipped.
\begin{theo}\label{iter}
For $n \in \mathbb{N}$, we define
\begin{align*}
 r^{n+1} & :=  q^{n+1} - q^n, \quad \sigma^{n+1} := \varrho^{n+1} - \varrho^{n}, \quad w^{n+1} := v^{n+1} - v^n\\
 e^{n+1} & := |r^{n+1}| + |w^{n+1}| \, .
 \end{align*}
 Then there are $k_0, \, p_0 > 0$ and $0 < t_1 \leq T_0$ such that for all $t \in [0, \, T_0 -t_1]$, the quantity
\begin{align*} 
 E^{n+1}(t) := & k_0 \, \sup_{\tau \in [t, \, t+t_1]}  ( \|e^{n+1}(\tau)\|_{L^2(\Omega)}^2 + \|\sigma^{n+1}(\tau)\|_{L^2(\Omega)}^2) \\
 & + p_0 \, \int_{Q_{t,t+t_1}}  (|\nabla r^{n+1}|^2 + |\nabla w^{n+1}|^2)  \, dxd\tau
 \end{align*}
satisfies $ E^{n+1}(t) \leq \frac{1}{2} \, E^{n}(t)$ for all $n \in \mathbb{N}$.
\end{theo}
\begin{proof}
To be shorter, denote $R^n := R(\varrho^{n+1}, \, q^n)$, $\widetilde{M}^n := \widetilde{M}(\varrho^{n+1}, \, q^n)$, $P^n := P(\varrho^{n+1}, \, q^n)$. For simplicity, we also define $g^n :=  (R_{\varrho}^n \, \varrho^{n+1} - R^n) \, \divv v^n - R_q^n \, v^n  \cdot \nabla q^n + \tilde{r}(\varrho^{n+1}, \, q^n)$. 
The differences $r^{n+1}$, $\sigma^{n+1}$ and $w^{n+1}$ solve
\begin{align}\label{difference}
 & R_q^n \partial_t r^{n+1} - \divv (\widetilde{M}^n \, \nabla r^{n+1}) = \\
& \quad  + g^n - g^{n-1} + (R_q^{n-1}-R_q^n) \, \partial_t q^n-\divv((\widetilde{M}^{n-1} -\widetilde{M}^n) \, (\nabla q^n - \tilde{b}(x, \, t)) \, ,\nonumber\\
& \label{difference2} \partial_t \sigma^{n+1} + \divv (\sigma^{n+1} \, v^n + \varrho_n \, w^n) = 0 \, ,\\
\label{difference3}
 & \varrho^{n+1} \, \partial_t w^{n+1} - \divv \mathbb{S}(\nabla w^{n+1}) = (R^n - R^{n-1}) \cdot \tilde{b}(x,t)- \nabla (P^n-P^{n-1}) \\
&  - \sigma^{n+1} \, [\partial_t v^n + (v^{n} \cdot \nabla) v^n - \bar{b}(x, \, t)] - \varrho_{n} \, [(w^{n} \cdot \nabla )v^n +(v^{n-1} \cdot \nabla )w^n] \nonumber\, .
\end{align}
together with the boundary conditions $\nu \cdot \nabla r^{n+1} = 0$ and $w^{n+1} = 0$ on $S_{T_0}$ and homogeneous initial conditions. We multiply in \eqref{difference} with $r^{n+1}$ and make use of the formula
\begin{align*}
\frac{1}{2} \,  \partial_t  (R_q^n \, r^{n+1} \cdot r^{n+1}) =  R_q^n \partial_t r^{n+1} \cdot r^{n+1}+ \frac{1}{2} \, \partial_t R^n_q \, r^{n+1} \cdot r^{n+1} \, .
\end{align*}
We introduce the abbreviation $a^n(r^{n+1}, \, r^{n+1}) := \frac{1}{2} \,  R_q^n \, r^{n+1} \cdot r^{n+1}$. After integration over $\Omega$, and using the Gauss divergence theorem, we obtain that
\begin{align*}
& \frac{d}{dt}\, \int_{\Omega} a^n(r^{n+1}, \, r^{n+1}) \, dx + \int_{\Omega} \widetilde{M}^n \nabla r^{n+1} \cdot \nabla r^{n+1} \, dx \\
& \quad = \int_{\Omega} [g^n - g^{n-1} + (R_q^{n-1}-R_q^n) \, \partial_t q^n] \cdot r^{n+1} \, dx\\
& \qquad + \int_{\Omega} (\widetilde{M}^{n-1} -\widetilde{M}^n) \, (\nabla q^n-\tilde{b}) \cdot \nabla r^{n+1} \, dx + \int_{\Omega}  \frac{1}{2} \, \partial_t R^n_q \, r^{n+1} \cdot r^{n+1} \, dx \, .
\end{align*}
On the interval $[0, \, T_0]$, the \emph{a priori} bounds \eqref{uniform} ensure that $\widetilde{M}^n = \widetilde{M}(\varrho^{n+1}, \, q^n)$ has a smallest eigenvalue strictly bounded away from zero. Thus $\widetilde{M}^n \nabla r^{n+1} \cdot \nabla r^{n+1} \geq \lambda_0 \, |\nabla r^{n+1}|^2$. Invoking the Young inequality and standard steps
 \begin{align}\label{energy}
& \frac{d}{dt}\, \int_{\Omega} a^n(r^{n+1}, \, r^{n+1}) \, dx + \frac{\lambda_0}{2} \, \int_{\Omega} |\nabla r^{n+1}|^2 \, dx \nonumber\\
& \leq  \int_{\Omega} [|g^n - g^{n-1}| + |R_q^{n-1}-R_q^n|\, |\partial_t q^n|] \, |r^{n+1}| \, dx\nonumber\\
& +\frac{1}{2\lambda_0} \, \int_{\Omega} |\widetilde{M}^{n-1} -\widetilde{M}^n|^2 \, (|\nabla q^n|^2+|\tilde{b}|^2) \, dx + \int_{\Omega} \frac{1}{2} \, |\partial_t R^n_q|\,  |r^{n+1}|^2 \, dx \, .
\end{align}
We want to estimate the differences $g^n - g^{n-1}$. To do it shorter, we shall denote $K_0$ a generic number depending possibly on $\inf_{n\in \mathbb{N}, \, (x,t) \in Q_{T_0}} \varrho_n(x,t)$ and on $\sup_{n\in\mathbb{N}} \|(q^n, \, \varrho_n, \, v^n)\|_{L^{\infty}(Q_{T_0})}$. These quantities are bounded independently on $n$ due to the choice of $T_0$; $K_0$ might moreover depend on the $C^2-$norm of the maps $R$ and $\widetilde{M}$ over the range of $(\varrho_n, \, q^n)$ on $Q_{T_0}$. This range is contained in a compact $\mathcal{K}$ of $\mathbb{R}_+ \times \mathbb{R}^{N-1}$. Thus $|R(\varrho^{n+1}, \, q^n) -R(\varrho_{n}, \, q^{n-1})| \leq \|R\|_{C^2(\mathcal{K})} \, (|\sigma^{n+1}| + |r^n|) \leq K_0 \,  (|\sigma^{n+1}| + |r^n|)$. By means of these reasoning, we readily show that 
\begin{align*}
|g^n - g^{n-1}| \leq K_0 \, \Big[(1+|v^n_x| + |v^n| \, |q^n_x|) \, (|\sigma^{n+1}| + |r^n|) + |w^n_x| +  |q^n_x| \, |w^n| +  |v^n| \, | r_x^n|\Big] \, .
\end{align*}
Similarly we estimate $| \widetilde{M}^{n-1} -\widetilde{M}^n|^2 \leq K_0 \,  (|\sigma^{n+1}|^2 + |r^n|^2)$ and $|\partial_t R^n_q| \leq K_0 \, (|\varrho^{n+1}_t| + |q^n_t|)$.

We rearrange terms, and we recall that $e^n := |r^n|+ |w^n|$. From \eqref{energy}, we obtain the estimate
\begin{align}\label{energy2}
& \frac{d}{dt}\, \int_{\Omega} a^n(r^{n+1}, \, r^{n+1}) \, dx + \frac{\lambda_0}{2} \, \int_{\Omega} |\nabla r^{n+1}|^2 \, dx \nonumber\\
& \leq K_0 \int_{\Omega} |r^{n+1}| \, (e^n+|\sigma^{n+1}|) \, (1+|v^n_x| + |q^n_x| + |q^n_t|) \, dx + K_0 \, \int_{\Omega} (|\varrho^{n+1}_t| + |q^n_t|) \,  |r^{n+1}|^2 \, dx  \nonumber\\
& \quad + K_0 \, \int_{\Omega} |r^{n+1}| \, (|w^n_x| + |r^n_x|) \, dx  + K_0 \, \int_{\Omega} (e^n + |\sigma^{n+1}|)^2 \, (|q^n_x|^2+|\tilde{b}|^2) \, dx \, .
\end{align}
To transform the right-hand we apply H\"older's inequality, the Sobolev embedding theorem and Young's inequality according to the schema
\begin{align}\begin{split}\label{schema}
\int_{\Omega} a \, b \, c \, dx \leq & \|a\|_{L^3} \, \|b\|_{L^6} \, \|c\|_{L^2}\\
\leq & C \, \|a\|_{L^3} \, (\|\nabla b\|_{L^2} + \|b\|_{L^2}) \, \|c\|_{L^2}\\
\leq & \frac{\lambda_0}{4} \, \|\nabla b\|_{L^2}^2 + C^2 \, (\frac{1}{\lambda_0}+\frac{1}{4}) \, \|a\|_{L^3}^2 \, \|c\|_{L^2}^2 + \|b\|_{L^2}^2 \, .
\end{split}
\end{align}
We apply this first with $a = 1+ |v^n_x| + |q^n_x| + |q^n_t|$ and $b = r^{n+1}$ and $c = e^n + |\sigma^{n+1}|$. Thus,
\begin{align*}
& \int_{\Omega} |r^{n+1}| \, (e^n + |\sigma^{n+1}|) \, (1+|v^n_x| + |q^n_x| + |q^n_t|) \, dx  \leq  \frac{\lambda_0}{4} \, \|\nabla r^{n+1}\|_{L^2}^2 \\
& \quad  + C^2 \, (\frac{1}{\lambda_0} + \frac{1}{4}) \, \|1+|v^n_x| + |q^n_x| + |q^n_t|\|_{L^3}^2 \, (\|e^n\|_{L^2}^2+\|\sigma^{n+1}\|^2_{L^2}) + \|r^{n+1}\|_{L^2}^2 \, .
\end{align*}
We choose next $a = |\varrho^{n+1}_t| + |q^n_t|$ and $b = r^{n+1} = c$, to get
\begin{align*}
  \int_{\Omega} (|\varrho^{n+1}_t| + |q^n_t|) \,  |r^{n+1}|^2 \, dx \leq & \frac{\lambda_0}{4} \, \|\nabla r^{n+1}\|_{L^2}^2 \\
& + [C^2 \,  (\frac{1}{\lambda_0}+\frac{1}{4}) \, \||\varrho^{n+1}_t| + |q^n_t|\|_{L^3}^2 + 1] \, \|r^{n+1}\|_{L^2}^2   \, .
 \end{align*}
 Employing Young's inequality we find for $\delta > 0$ arbitrary that
 \begin{align*}
 \int_{\Omega} (e^n + |\sigma^{n+1}|)^2 \, (|q^n_x|^2 + |\tilde{b}|^2)\, dx & \leq (\|q^n_x\|_{L^{\infty}(\Omega)}^2 + \|\tilde{b}\|_{L^{\infty}(\Omega)}^2) \, (\|e^n\|_{L^2}^2 +\|\sigma^{n+1}\|_{L^2}^2) \, , \\
 K_0 \, \int_{\Omega} |r^{n+1}| \, (|w^n_x| + |r^n_x|) \, dx & \leq \delta \, \int_{\Omega} (|w^n_x|^2 + |r^n_x|^2) \,  dx + \frac{K^2_0}{4\delta} \, \int_{\Omega} |r^{n+1}|^2 \, dx \, .
 \end{align*}
From \eqref{energy2} we deduce the inequality
  \begin{align}\label{energy3}
& \frac{d}{dt}\, \int_{\Omega} a^n(r^{n+1}, \, r^{n+1}) \, dx + \frac{\lambda_0}{4} \, \int_{\Omega} |\nabla r^{n+1}|^2 \, dx \\
& \leq D(t) \,  (\|e^n\|_{L^2}^2 + \|\sigma^{n+1}\|_{L^2}^2)+ D^{(1)}_{\delta}(t) \, \|r^{n+1}\|_{L^2}^2 + \delta \, \int_{\Omega} (|w^n_x|^2 + |r^n_x|^2) \,  dx \, , \nonumber
\end{align}
in which the coefficients $D$ and $D^{(1)}_{\delta}$ satisfy
\begin{align*}
D(t) & \leq K_0 \, (|\Omega|^{\frac{2}{3}} +\|v^n_x(t)\|_{L^3}^2 + \|q^n_x(t)\|_{L^3}^2 + \|q^n_t(t)\|_{L^3}^2 + \|q^n_x(t)\|_{L^{\infty}}^2 + \|\tilde{b}(t)\|_{L^{\infty}}^2 ) \, ,\\
D^{(1)}_{\delta}(t) & \leq K_0 \, (\|\varrho^{n+1}_t(t)\|_{L^3}^2 + \|q^n_t(t)\|_{L^3}^2 + \delta^{-1}) \, .
\end{align*}
Next we multiply \eqref{difference2} with $\sigma^{n+1}$, integrate over $\Omega$, and this yields
\begin{align*}
\frac{1}{2} \, \frac{d}{dt} \int_{\Omega} |\sigma^{n+1}|^2 \, dx = - \frac{1}{2} \,  \int_{\Omega} \divv v^n \, (\sigma^{n+1})^2 \, dx - \int_{\Omega} \divv( \varrho_n \, w^n) \, \sigma^{n+1} \, dx\, ,\\
 \frac{1}{2} \, \frac{d}{dt} \int_{\Omega} |\sigma^{n+1}|^2 \, dx \leq \int_{\Omega} [\frac{1}{2} \, |v^n_x| \, |\sigma^{n+1}|^2 + K_0 \, |w^n_x| \, |\sigma^{n+1}| + |\varrho^{n}_x| \, |w^n| \,  |\sigma^{n+1}|] \, dx \, .
\end{align*}
We note that $\int_{\Omega} \frac{1}{2} \, |v^n_x| \, (\sigma^{n+1})^2 \, dx\leq \frac{1}{2} \, \|v^n_x\|_{L^{\infty}(\Omega)} \, \|\sigma^{n+1}\|_{L^2}^2$, and
employing Young's inequality we see that $\int_{\Omega}  K_0 \, |w^n_x| \, |\sigma^{n+1}|\, dx \leq \delta \, \int_{\Omega} |w^n_x|^2 \, dx + \frac{K^2_0}{4\delta} \,  \|\sigma^{n+1}\|_{L^2}^2$. As already seen 
 \begin{align*}
 \int_{\Omega} |\varrho^{n}_x| \, |w^n| \,  |\sigma^{n+1}| \, dx  \leq \delta \, \|\nabla w^n\|_{L^2}^2 + \frac{C^2}{4} \, \left(\frac{1}{\delta}+1\right) \, \|\varrho^{n}_x\|_{L^3}^2 \, \| \sigma^{n+1}\|_{L^2}^2 + \|w^n\|_{L^2}^2 \, ,
\end{align*}
allowing us to conclude that
\begin{align}\label{energy4}
  \frac{1}{2} \, \frac{d}{dt} \int_{\Omega} |\sigma^{n+1}|^2 \, dx \leq 2\delta \, \int_{\Omega} |w^n_x|^2 \, dx + D^{(2)}_{\delta}(t) \, \|\sigma^{n+1}\|_{L^2}^2  + \|e^n\|_{L^2}^2 \, ,
\end{align}
in which $ D^{(2)}_{\delta}(t) \leq K_0 \, \delta^{-1} \, (1 + \|\varrho^{n}_x(t)\|_{L^3(\Omega)}^2+\|v^{n}_x(t)\|_{L^{\infty}(\Omega)}^2 )$. Finally, we multiply \eqref{difference3} with $w^{n+1}$ and obtain that
\begin{align*}
& \frac{\varrho^{n+1}}{2}\, \partial_t |w^{n+1}|^2 - \divv \mathbb{S}(\nabla w^{n+1}) \cdot w^{n+1} = - \nabla (P^n-P^{n-1})\cdot w^{n+1} + (R^n -R^{n-1}) \tilde{b} \cdot w^{n+1}\\
& - \sigma^{n+1} \, [\partial_t v^n+(v^{n} \cdot \nabla) v^n) - \bar{b}] \cdot w^{n+1}  - \varrho_{n} \, [(w^{n} \cdot \nabla )v^n+(v^{n-1} \cdot \nabla )w^n]\cdot w^{n+1}  \, .
\end{align*}
After integration over $\Omega$,
\begin{align*}
& \frac{1}{2} \, \frac{d}{dt} \int_{\Omega} \varrho^{n+1} \, |w^{n+1}|^2 \, dx + \int_{\Omega}  \mathbb{S}(\nabla w^{n+1}) \cdot \nabla w^{n+1} \, dx \\
& =   \frac{1}{2} \, \int_{\Omega} \partial_t \varrho^{n+1} \, |w^{n+1}|^2 \, dx + \int_{\Omega} (P^n-P^{n-1}) \, \divv w^{n+1} \, dx + \int_{\Omega}(R^n -R^{n-1}) \tilde{b}\cdot w^{n+1} \, dx \\
& - \int_{\Omega} \{\sigma^{n+1} \, [\partial_t v^n + (v^{n} \cdot \nabla) v^n - \bar{b}] - \varrho_{n} \, [(w^{n} \cdot \nabla )v^n+(v^{n-1} \cdot \nabla )w^n]\}\cdot w^{n+1} \, dx \, .
 \end{align*}
We use $ \int_{\Omega}  \mathbb{S}(\nabla w^{n+1}) \cdot \nabla w^{n+1} \, dx  \geq \nu_0 \, \int_{\Omega}  |\nabla w^{n+1}|^2 \, dx$. We estimate
\begin{align*}
\left| \int_{\Omega} (P^n-P^{n-1}) \, \divv w^{n+1} \, dx \right | & \leq \frac{\nu_0}{2} \, \int_{\Omega} |\nabla w^{n+1}|^2 \, dx + \frac{1}{2\nu_0} \, \int_{\Omega} |P^n-P^{n-1}|^2 \, dx\\
& \leq \frac{\nu_0}{2} \, \int_{\Omega} |\nabla w^{n+1}|^2 \, dx + \frac{K^2_0}{2\nu_0} \, \int_{\Omega} (|\sigma^{n+1}|^2+|r^n|^2) \, dx \, .
\end{align*}
Further,
\begin{align*}
|(R^n -R^{n-1}) \tilde{b}\cdot w^{n+1}| & \leq K_0 \, (|\sigma^{n+1}| + |r^n|) \, |\tilde{b}| \,  |w^{n+1}| \, ,\\
 |\sigma^{n+1} \, (\partial_t v^n + (v^{n} \cdot \nabla) v^n - \bar{b}) \cdot w^{n+1}| & \leq K_0 \, |\sigma^{n+1}| \, |w^{n+1}| \, (|v^n_t| + |v^n_x| + |\bar{b}|) \, ,\\
 |\varrho_{n} \, [(w^{n} \cdot \nabla )v^n+(v^{n-1} \cdot \nabla )w^n]\cdot w^{n+1}| & \leq K_0 \, |w^{n+1}| \, (|w^n| \,|v^n_x| + |w^n_x|) \, .
\end{align*}
Thus,
\begin{align*}
 & \frac{1}{2} \, \frac{d}{dt} \int_{\Omega} \varrho^{n+1} \, |w^{n+1}|^2 \, dx + \frac{\nu_0}{2} \, \int_{\Omega} |\nabla w^{n+1}|^2  \, dx\\
 & \leq  \frac{1}{2} \, \int_{\Omega} |\partial_t \varrho^{n+1}| \, |w^{n+1}|^2 \, dx
 +  \frac{K^2_0}{2\nu_0} \, \int_{\Omega} (|\sigma^{n+1}|^2+|r^n|^2) \, dx + K_0 \, \int_{\Omega}|r^n| \, |\tilde{b}| \,  |w^{n+1}| \, dx  \\
&  + K_0 \,  \int_{\Omega} [|\sigma^{n+1}| \, |w^{n+1}| \, (|v^n_t| + |v^n_x| +|\tilde{b}|+ |\bar{b}|) + (|v^n_x| \, |w^n| + |w^n_x|) \, |w^{n+1}|] \, dx\, .
 \end{align*}
 By means of \eqref{schema} and Young's inequality, we can also show that
 \begin{align*}
& K_0\,   \int_{\Omega} [\sigma^{n+1} \, |w^{n+1}| \, (|v^n_t| + |v^n_x|+|\tilde{b}|+ |\bar{b}|) \, dx \leq \frac{\nu_0}{8} \, \|\nabla w^{n+1}\|_{L^2}^2 \\
& \quad + C^2K^2_0  \, (\frac{2}{\nu_0}+\frac{1}{4}) \, \||v^n_t| + |v^n_x|+|\tilde{b}|+ |\bar{b}|\|_{L^3}^2 \, \|\sigma^{n+1}\|_{L^2}^2 +  \|w^{n+1}\|_{L^2}^2 \, ,\\
& K_0 \, \int_{\Omega}|r^n| \, |\tilde{b}| \,  |w^{n+1}| \, dx \leq \frac{\nu_0}{8} \, \|\nabla w^{n+1}\|_{L^2}^2 +C^2K^2_0 \, (\frac{2}{\nu_0}+\frac{1}{4}) \, \|\tilde{b}\|_{L^3}^2 \, \|r^{n}\|_{L^2}^2 +  \|w^{n+1}\|_{L^2}^2 \, ,\\
&  K_0 \,  \int_{\Omega}  |v^n_x| \, |w^n| \, |w^{n+1}| \, dx \leq \frac{\nu_0}{8} \, \|\nabla w^{n+1}\|_{L^2}^2 +C^2K^2_0 \, (\frac{2}{\nu_0}+\frac{1}{4}) \, \|v^n_x\|_{L^3}^2 \, \|w^n\|_{L^2}^2 +  \|w^{n+1}\|_{L^2}^2 \, , \\
& \int_{\Omega} |\partial_t \varrho^{n+1}| \, |w^{n+1}|^2 \, dx \leq \frac{\nu_0}{8} \, \|\nabla w^{n+1}\|_{L^2}^2 + \left(\frac{2C^2}{\nu_0} \, \|\varrho^{n+1}_t\|_{L^3}^2 + 1\right)  \|w^{n+1}\|_{L^2}^2 \, ,\\
& K_0 \,  \int_{\Omega} |w^n_x| \, |w^{n+1}| \, dx \leq \delta \, \int_{\Omega} |w^n_x|^2 \, dx + \frac{K^2_0}{4\delta} \, \int_{\Omega} |w^{n+1}|^2 \, dx \, .
\end{align*}
Overall, we obtain for the estimation of \eqref{difference3} that
\begin{align}\label{energy5}
& \frac{1}{2} \, \frac{d}{dt} \int_{\Omega} \varrho^{n+1} \, |w^{n+1}|^2 \, dx + \frac{\nu_0}{2} \, \int_{\Omega} |\nabla w^{n+1}|^2  \, dx \leq \delta \, \int_{\Omega} |w^n_x|^2 \, dx\nonumber\\
& \quad \qquad \quad +  D^{(3)}(t) \, (\|e^n\|_{L^2}^2 + \|\sigma^{n+1}\|_{L^2}^2) + D^{(4)}_{\delta}(t) \, \|w^{n+1}\|_{L^2}^2 \, ,
 \end{align}
 in which $D^{(3)}(t) \leq K_0 \, (\|v^n_t\|_{L^3}^2 + \|v^n_x|\|_{L^3}^2+\|\tilde{b}\|_{L^3}^2+ \|\bar{b}\|_{L^3}^2)$ and $D^{(4)}_{\delta}(t) \leq K_0 \, (\|\varrho^{n+1}_t\|_{L^3}^2 + \delta^{-1})$.
 
We add the three inequalities \eqref{energy3}, \eqref{energy4} and \eqref{energy5} and get
\begin{align}\label{energy6}
& \frac{d}{dt}\, \int_{\Omega} \{a^n(r^{n+1}, \, r^{n+1}) +\tfrac{1}{2} \, |\sigma^{n+1}|^2 + \tfrac{1}{2} \, \varrho_n \, |w^{n+1}|^2\} \, dx \nonumber \\
& + \frac{\lambda_0}{2} \, \int_{\Omega} |\nabla r^{n+1}|^2 \, dx +  \frac{\nu_0}{2} \, \int_{\Omega} |\nabla w^{n+1}|^2  \, dx \nonumber\\
& \leq 4 \, \delta \, \int_{\Omega} (|\nabla r^{n}|^2 + |\nabla w^{n}|^2)  \, dx + F_{\delta}(t) \, (\|e^n\|_{L^2}^2 +\|\sigma^{n+1}\|_{L^2}^2)  + F^{(1)}_{\delta}(t) \, \|e^{n+1}\|_{L^2}^2 \, .
\end{align}
In this inequality we have introduced $F_{\delta}(t) := 1 + D(t) + D^{(2)}_{\delta}(t) + D^{(3)}(t)$, and $F_{\delta}^{(1)}(t) :=  D^{(1)}_{\delta}(t) +  D^{(4)}_{\delta}(t)$. These definitions and the inequalities above show that
\begin{align*}
 F_{\delta}(t) \leq  K_0 \, \Big[ & (\|q^n_x\|_{L^3} + \|q^n_t\|_{L^3} +\|q^n_x\|_{L^{\infty}} +\|v^n_x\|_{L^3} + \|v^n_t\|_{L^3}+\|\varrho^n_x\|_{L^3})^2\\
& + \|\bar{b}\|_{L^3}^2+ \|\tilde{b}\|_{L^{\infty}}^2 +  \|\tilde{b}\|_{L^{3}}^2 + \delta^{-1}\Big] \, .
\end{align*}
Consequently, due to embedding properties of $W^{2,p}(\Omega)$, it follows for $s \in [0, \, T_0]$ arbitrary and for $0< t_1 \leq T_0$ and $t \leq T_0-t_1$ that
\begin{align}\label{energy7}
& |F_{\delta}(s)| \leq C \, K_0 \, [\|q^n(s)\|_{W^{2,p}}^2 + \|v^n(s)\|_{W^{2,p}}^2 + \|\varrho^n_x(s)\|_{L^{p}}^2  + \|\tilde{b}(s)\|_{W^{1,p}}^2 + \|\bar{b}(s)\|^2_{L^p} + \delta^{-1}] \, , \nonumber\\
& \int_{t}^{t+t_1} F_{\delta}(s) \, ds  \leq \tilde{K}_0 \, \{ t_1^{1-\frac{2}{p}} \, [\|q^n\|_{W^{2,1}_p(Q_{T_0})}^2 + \|v^n\|_{W^{2,1}_p(Q_{T_0})}^2 +\|\tilde{b}\|_{W^{1,0}_p(Q_{T_0})}^2 + \|\bar{b}\|^2_{L^p(Q_{T_0})} ] \nonumber \\
&\phantom{\int_{t}^{t+t_1} F_{\delta}(s) \, ds  \leq  C \, K_0 \,}  + t_1 \, [\|\varrho^n_x\|_{L^{p,\infty}(Q_{T_0})}^2 + \, \delta^{-1}]\}\nonumber\\
 & \phantom{\int_{t}^{t+t_1} F_{\delta}(s) \, ds } \leq C_0 \, (1+\delta^{-1}) \, t_1^{1-\frac{2}{p}} \, .
\end{align}
Here we use the uniform bounds \eqref{uniform}. 
Similarly we show that $F^{(1)}_{\delta}(s) \leq K_0 \, [\|q^n_t\|_{L^{3}(\Omega)} + \|\varrho^{n+1}_t\|_{L^{3}(\Omega)} + \delta^{-1}]$ to show that
\begin{align}\label{energy8}
 \int_{t}^{t+t_1} F^{(1)}_{\delta}(s) \, ds & \leq \tilde{K}_0 \,\{ t_1^{1-\frac{2}{p}} \, \|q^n_t\|_{L^p(Q_{T_0})}^2 + t_1 \, [\|\varrho^{n+1}_t\|_{L^{p,\infty}(Q_{T_0})}^2+ \delta^{-1}]\}\nonumber\\
 & \leq C_1 \, (1+\delta^{-1}) \, t_1^{1-\frac{2}{p}}  \, .
\end{align}
We integrate \eqref{energy6} over $[t, \, \tau]$ for $t_1 \leq T_0$, $t \leq T_0 - t_1$ and $t \leq \tau \leq t+t_1$ arbitrary. Note that
\begin{align*}
 & \int_{\Omega} \{a^n(r^{n+1}, \, r^{n+1}) +\tfrac{1}{2} \, |\sigma^{n+1}|^2 + \tfrac{1}{2} \, \varrho_n \, |w^{n+1}|^2\}(\tau) \, dx\\
 & \quad \geq \frac{1}{2} \, \int_{\Omega} \{\lambda_{\inf}(R_q^n) \, |r^{n+1}|^2 + |\sigma^{n+1}|^2 + \inf_{Q_{T_0}} \varrho_n \, |w^{n+1}|^2\}(\tau)\} \, dx\\
 & \quad \geq \frac{1}{2} \, \min\{1, \, \lambda_{\inf}(R_q^n), \,  \inf_{Q_{T_0}} \varrho_n\} \, (\|e^{n+1}(\tau)\|_{L^2}^2 + \|\sigma^{n+1}(\tau)\|_{L^2}^2) \, .
\end{align*}
Invoking \eqref{uniform}, there is a uniform $k_0>0$ such that $\frac{1}{2} \, \min\{1, \, \lambda_{\inf}(R_q^n), \,  \inf_{Q_{T_0}} \varrho_n\} \geq k_0 > 0$. We also define $p_0 := \min\{\lambda_0, \, \nu_0\}$. This shows the inequality
\begin{align}\label{energy9}
& k_0 \, ( \|e^{n+1}(\tau)\|_{L^2}^2 + \|\sigma^{n+1}(\tau)\|_{L^2}^2) + \frac{p_0}{2} \int_{Q_{t,\tau}}  (|\nabla r^{n+1}|^2 + |\nabla w^{n+1}|^2) \nonumber\\
& \leq  \delta \,  \int_{Q_{t,\tau}} (|\nabla r^{n}|^2 + |\nabla w^{n}|^2)\nonumber\\
&  + \int_{t}^{\tau} \, F_{\delta}(s) \, (\|e^n(s)\|_{L^2}^2 +\|\sigma^{n+1}(s)\|_{L^2}^2)  \, ds + \int_{t}^{\tau} \, F^{(1)}_{\delta}(s) \, \|e^{n+1}(s)\|_{L^2}^2 \, ds \, .
\end{align}
Thus, taking the supremum over all $\tau \in [t, \, t+t_1]$ yields
\begin{align*}
& k_0 \, \sup_{t \leq \tau \leq t + t_1}  ( \|e^{n+1}(\tau)\|_{L^2}^2+ \|\sigma^{n+1}(\tau)\|_{L^2}^2) \leq \delta \,  \int_{Q_{t,t+t_1}} (|\nabla r^{n}|^2 + |\nabla w^{n}|^2) \\
& +  \int_{t}^{t+t_1} \, F_{\delta}(s) \, ds \,  \sup_{t \leq \tau \leq t_1} (\|e^n(\tau)\|_{L^2}^2 +\|\sigma^{n+1}(\tau)\|_{L^2}^2)  + \int_{t}^{t+t_1} \, F^{(1)}_{\delta}(s) \, ds \, \sup_{t \leq \tau \leq t+t_1} \|e^{n+1}(\tau)\|_{L^2}^2 \, .
\end{align*}
On the other hand, choosing $\tau = t+t_1$ in \eqref{energy9} shows that also $\frac{\min\{\lambda_0, \, \nu_0\}}{2} \int_{Q_{t,t+t_1}}  (|\nabla r^{n+1}|^2 + |\nabla w^{n+1}|^2)  \, dx$ is estimated above by the same right-hand. Thus
\begin{align*}
& k_0 \, \sup_{t \leq \tau \leq t + t_1}   (\|e^{n+1}(\tau)\|_{L^2}^2 +\|\sigma^{n+1}(\tau)\|_{L^2}^2   )+ \frac{p_0}{2}\,  \int_{Q_{t,t+t_1}}  (|\nabla r^{n+1}|^2 + |\nabla w^{n+1}|^2) \\
& \leq 2  \, \delta \,  \int_{Q_{t,t+t_1}} (|\nabla r^{n}|^2 + |\nabla w^{n}|^2)  + 2 \, \int_{t}^{t+t_1} \, F^{(1)}_{\delta}(s) \, ds \, \sup_{t \leq \tau \leq t+ t_1} \|e^{n+1}(\tau)\|_{L^2}^2 \\
& \qquad + 2 \, \int_{t}^{t+t_1} \, F_{\delta}(s) \, ds \,  \sup_{t \leq \tau \leq t+ t_1} (\|e^n(\tau)\|_{L^2}^2 ++\|\sigma^{n+1}(\tau)\|_{L^2}^2   )  \, .
\end{align*}
We choose $\delta_0 = \frac{p_0}{8}$ and $0 < t_1 < T_0-t$ such that $ 2 \, \int_{t}^{t+t_1} \, F^{(1)}_{\delta_0}(t) \, dt \leq \frac{k_0}{2}$. In view of \eqref{energy8}, it is sufficient to satisfy the condition $C_1 \, \left(1+ \frac{8}{p_0}\right) \, t_1^{1-\frac{2}{p}} \leq \frac{k_0}{4}$. Then
\begin{align*}
& \frac{k_0}{2} \,  \sup_{t \leq \tau \leq t + t_1}  ( \|e^{n+1}(\tau)\|_{L^2}^2 +\|\sigma^{n+1}(\tau)\|_{L^2}^2   ) + \frac{p_0}{2} \int_{Q_{t,t+t_1}}  (|\nabla r^{n+1}|^2 + |\nabla w^{n+1}|^2) \\
& \leq \frac{p_0}{4} \, \int_{Q_{t,t+t_1}} (|\nabla r^{n}|^2 + |\nabla w^{n}|^2) + 2 \, \int_{t}^{t+t_1} \, F_{\delta_0}(s) \, ds \,  \sup_{t \leq \tau \leq t + t_1} ( \|e^n(\tau)\|_{L^2}^2 +\|\sigma^{n+1}(\tau)\|_{L^2}^2   )\, .
\end{align*}
By requiring that $C_0 \, \left(1+ \frac{8}{p_0}\right) \, t_1^{1-\frac{2}{p}} \leq \frac{k_0}{8}$, we choose $t_1$ such that $ \int_{t}^{t+t_1} \, F_{\delta_0}(s) \, ds \leq \frac{k_0}{8}$ (use \eqref{energy7}). It follows that
\begin{align*}
 & \frac{k_0}{4} \,  \sup_{t \leq \tau \leq t + t_1}  ( \|e^{n+1}(\tau)\|_{L^2}^2 +\|\sigma^{n+1}(\tau)\|_{L^2}^2   ) + \frac{p_0}{2} \, \int_{Q_{t,t+t_1}}  (|\nabla r^{n+1}|^2 + |\nabla w^{n+1}|^2)\\
 & \leq \frac{k_0}{4} \,  \sup_{t \leq \tau \leq t+ t_1}   \|e^{n}(\tau)\|_{L^2}^2 + \frac{p_0}{4} \,  \int_{Q_{t,t+t_1}}  (|\nabla r^{n}|^2 + |\nabla w^{n}|^2)\, . 
\end{align*}
The claim follows.
\end{proof}
In order to complete the proof of the Theorems \ref{MAIN}, \ref{MAIN2} it remains to investigate the characterisation of the maximal existence time $T^*$. 
\begin{lemma}\label{MAXEX}
 Suppose that $u = (q, \, \varrho, \, v) \in \mathcal{X}_{t}$ is a solution to $\widetilde{\mathscr{A}}(u) = 0$ and $u(0) = u_0$ for all $t < T^*$. If for some $\alpha > 0$ the quantity $\mathscr{N}(t) := \|q\|_{C^{\alpha,\frac{\alpha}{2}}(Q_{t})} + \|\nabla q\|_{L^{\infty,p}(Q_{t})} + \|v\|_{L^{z \, p,p}(Q_{t})} + \int_{0}^{t} [\nabla v(s)]_{C^{\alpha}(\Omega)} \, ds$ is finite for $t \nearrow T^*$, then it is possible to extend the solution to a larger time interval.
\end{lemma}
\begin{proof}
To show this claim we first note that the components of $v_x$ have all spatial mean-value zero over $\Omega$ due to the boundary condition \eqref{lateralv}. Thus, the inequalities $\|v_x(s)\|_{L^{\infty}(\Omega)} \leq c_{\Omega} \, [v_x(s)]_{C^{\alpha}(\Omega)}$ and $\|v_x\|_{L^{\infty,1}(Q_t)} \leq c_{\Omega} \, \int_0^t [ v_x(s)]_{C^{\alpha}(\Omega)} \, ds$ are valid. Invoking the Proposition \ref{solonnikov2}, we thus see that $(m(t))^{-1}, \, M(t)$ and $\sup_{s \leq t} [\varrho(s)]_{C^{\alpha}(\Omega)}$ are all bounded by a function of $\int_0^t [ v_x(s)]_{C^{\alpha}(\Omega)} \, ds$, thus also by a function of $\mathcal{N}(t)$.
Invoking further the estimates of Proposition \ref{solonnikov2}, we also see that 
\begin{align*}
\|\varrho_x(s)\|_{L^p(\Omega)} \leq & \phi(R_0, \, \|v_x\|_{L^{\infty,1}(Q_s)}) \, (1+ \int_{0}^s\|v_{x,x}(\tau)\|_{L^{p}(\Omega)} \, d\tau) \\
\leq & \phi(R_0, \, \mathcal{N}(s)) \, (1+ \mathscr{V}(s; \, v)) \, ,
\end{align*}
for all $s \geq 0$, with a function $\phi$ increasing in its arguments. Next we apply the Corollary \ref{normvmain}. Due to the fact that $(m(t))^{-1}, \, M(t)$ and $\sup_{s \leq t} [\varrho(s)]_{C^{\alpha}(\Omega)}$ are bounded by a function of $\mathcal{N}(t)$, this yields $\mathscr{V}(t; \, v) \leq  \phi(t, \, \mathcal{N}(t)) \, (\|f\|_{L^p(Q_t)} + \|v^0\|_{W^{2-\frac{2}{p}}_p(\Omega)})$. We recall the form \eqref{A3right} of the function $f$, and estimate
\begin{align*}
 |f(x, \, t)| \leq & |\nabla \varrho| \, \sup_{Q_t} |R_{\varrho}(\varrho, \, q)| +|\nabla q| \,  \sup_{Q_t} |R_{q}(\varrho, \, q)| \\
 & + c \, (|v(x, \, t)| \, |v_x(x, \, t)| + |\bar{b}(x, \, t)| + |\tilde{b}(x, \, t)|) \,  \sup_{Q_t}  \varrho \, .
\end{align*}
We can bound the coefficients via $\sup_{Q_t} |R_{\varrho}(\varrho, \, q)| \leq \phi(M(t), \, \|q\|_{L^{\infty}(Q_t)}) \leq \phi(\mathcal{N}(t))$, etc.
Therefore, we can show that
\begin{align*}
 \|f\|_{L^p(Q_t)}^p \leq \phi(\mathcal{N}(t)) \, (\|\nabla \varrho\|_{L^p(Q_t)}^p + \|\nabla q\|_{L^p(Q_t)}^p + \|v \, \nabla v\|_{L^p(Q_t)}^p + \|\tilde{b}\|_{L^p(Q_t)}^p + \|\bar{b}\|_{L^p(Q_t)}^p) \, . 
\end{align*}
Using the abbreviation $A_0(t) :=  \|\tilde{b}\|_{L^p(Q_t)}^p + \|\bar{b}\|_{L^p(Q_t)}^p   + \|v^0\|_{W^{2-\frac{2}{p}}_p(\Omega)}$ we obtain, after straightforward computations
\begin{align*}
 \mathscr{V}^p(t; \, v) \leq  \phi(t, \, \mathcal{N}(t)) \, ( & \|\nabla \varrho\|_{L^p(Q_t)}^p +  \|v \, \nabla v\|_{L^p(Q_t)}^p +   \|\nabla q\|_{L^p(Q_t)}^p + A_0(t))  \, .
\end{align*}
As shown, we have $\|\nabla \varrho\|_{L^p(Q_t)}^p \leq \phi(R_0, \, \|v_x\|_{L^{\infty,1}(Q_s)})  \, \int_0^t (1+ \mathscr{V}(s; \, v))^p \, ds$. Recall the continuity of $W^{2-\frac{2}{p}}_p \subset L^{\frac{3p}{(5-p)^+}}$ (cf. Rem. \ref{parabolicspace}). Choosing $z = \frac{3}{p-2}$ if $3 < p < 5$, $z > 1$ arbitrary if $p = 5$ and $z = 1$ if $p > 5$, we are thus able to also show by means of H\"older's inequality that $\|v \,  v_x\|_{L^p(Q_t)}^p \leq \int_0^t \|v(s)\|^p_{L^{z\, p}} \, \mathscr{V}^p(s; \, v) \, ds$. 

Invoking the Gronwall Lemma yields $\mathscr{V}^p(t; \, v) \leq \phi(t, \, R_0, \,\mathcal{N}(t)) \,  (\|\nabla q\|_{L^p(Q_t)}^p + A_0(t))$. Since $\|\nabla q\|_{L^p(Q_t)}$ is also controlled by $t$ and $\mathcal{N}(t)$, we obtain that $\mathscr{V}^p(t; \, v) \leq \phi(t, \, R_0, \,\mathcal{N}(t))$. Obviously we now have also $\|\nabla \varrho\|_{L^{p,\infty}(Q_t)}^p \leq \phi(t, \, R_0, \,\mathcal{N}(t)) $. The Corollary \ref{normrhomain} yields that $\|\varrho\|_{C^{\beta, \frac{\beta}{2}}(Q_t)} \leq \phi(t, \, R_0, \,\mathcal{N}(t))$ for $\beta = 1-\frac{3}{p}$.

To show the final claim, we reconsider the inequality \eqref{ControlProducts0} in the proof of Corollary \ref{A1linUMFORMsecond}. Note that a solution $(q, \,\varrho, \, v)$ is a fixed-point of $\mathcal{T}$, so that this inequality is valid with $q^* = q$, $\varrho^* = \varrho$ and $v^* =v$. The factors $\phi^*_{1,t}, \, \phi^*_{2,t}$ are increasing functions of $(m(t))^{-1}$, $M(t)$, $\|q\|_{L^{\infty}(Q_t)}$ and $[q]_{C^{\beta,\frac{\beta}{2}}(Q_t)}$, $[\varrho]_{C^{\beta,\frac{\beta}{2}}(Q_t)}$.
With the preliminary considerations in this proof, we thus can state that
\begin{align*}
 \mathscr{V}^p(t, \, q) \leq   \phi(t, \, \mathcal{N}(t)) \, (& \|q^0\|^p_{W^{2-\frac{2}{p}}_p(\Omega)} + \|q\|^p_{W^{1,0}_p(Q_t)} + \|g\|_{L^p(Q_t)}^p \nonumber\\
 & + \|\nabla \varrho \cdot \nabla q\|^p_{L^p(Q_t)} + \|\nabla q \cdot \nabla q\|^p_{L^p(Q_t)})  \, , 
\end{align*}
Invoking \eqref{A1right} and the fact that $\|v\|_{W^{2,1}_p(Q_t)} + \|\nabla \varrho\|_{L^{p,\infty}(Q_t)}^p$ and, by definition $\|\nabla q\|_{L^p(Q_t)}$ are all bounded by $\mathcal{N}(t)$, we can obtain the inequality
\begin{align*}
 \mathscr{V}^p(t; \, q) \leq \phi(t, \, D_0, \, \mathcal{N}(t)) \, (1 + \int_{0}^t \|\nabla q(s)\|^p_{L^{\infty}(\Omega)} \, (\|\nabla \rho(s)\|_{L^p}^p  +\|\nabla q(s)\|_{L^p}^p   ) \, ds \, .
\end{align*}
Thus, if $ \|\nabla q(s)\|^p_{L^{\infty}(\Omega)}$ is integrable in time, we obtain by means of Grownall an independent estimate in terms of $\mathcal{N}(t)$. The claim follows.
\end{proof}

\section{Estimates for the solutions to the second linearisation}\label{contiT1}

We now consider the equations \eqref{linearT1second}, \eqref{linearT2second}, \eqref{linearT3second} underlying the definition of the map $\mathcal{T}^1$. Here the data is a pair $(r^*, \, w^*) \in \phantom{}_0\mathcal{Y}_T$, and we want to find the image $(r, \, w)$ in the same space as well as $\sigma \in \phantom{}_0W^{1,1}_{p,\infty}(Q_T)$ by solving these equations. The solvability will not be discussed, since it can be easily obtained by linear continuation using the estimates. We shall therefore go directly for the estimates.

The first point consists in obtaining estimates for solutions to a perturbed continuity equation. Precisely for this point, we need to assume more regularity of the function $\hat{\varrho}_0$.
\begin{lemma}\label{PerturbConti}
Assume that $\hat{\varrho}^0 \in W^{2,0}_p(Q_T)$, that $v, \, w \in W^{2,1}_p(Q_T; \, \mathbb{R}^3)$ and that $\sigma \in W^{1,1}_{p,\infty}(Q_T)$ solves $\partial_t \sigma +\divv(\sigma \, v + \hat{\varrho}^0 \, w) = 0$ in $Q_T$ with $\sigma(x, \, 0) = 0$ in $\Omega$. Then there are constants $c, \, C > 0$ depending only on $\Omega$, such that for all $s \leq T$ we have
\begin{align*}
 \|\sigma(s)\|_{W^{1,p}(\Omega)}^p \leq & C \, \exp(c \, \int_0^s [\|v_x\|_{L^{\infty}(\Omega)} + \|v_{x,x}\|_{L^p(\Omega)} + 1]  \, ds) \, \times \\
 & \times  (\|\hat{\varrho}^0\|_{W^{2,0}_p(Q_s)}^p \, \|w\|_{L^{\infty}(Q_s)}^p + \|\hat{\varrho}^0\|_{W^{1,1}_{p,\infty}(Q_s)}^p \,  \|w\|_{W^{2,0}_p(Q_s)}^p) \, .
\end{align*}
\end{lemma}
\begin{proof}
 After some obvious technical steps, we can show that the components $z_i := \sigma_{x_i}$ ($i=1,2,3$) of the gradient of $\sigma$ satisfy, in the sense of distributions,
 \begin{align*}
  \partial_t z_i + \divv(z_i \, v) = -\divv(\sigma \, v_{x_i}) - \divv(\hat{\varrho}^0_{x_i} \, w + \hat{\varrho}^0 \, w_{x_i}) =: -\divv(\sigma \, v_{x_i})  + R_i \, .
 \end{align*}
The right-hand side is bounded in $L^p(Q_t)$, and the velocity $v$ belongs to $W^{2,1}_p(Q_t)$. Thus, $z_i$ is also a renormalised solution to the latter equation. Without entering the details of this notion, the following identity is valid in the sense of distributions:
\begin{align*}
 \partial_t f(z) + \divv(f(z) \, v) + (z \cdot f_z(z) - f(z)) \, \divv v = \sum_{i=1}^3 f_{z_i}(z) \, (-\divv(\sigma \, v_{x_i}) + R_i)
 \end{align*}
for every globally Lipschitz continuous function $f \in C^1(\mathbb{R}^3)$. We integrate the latter identity over $Q_t$. Recall that $\sigma(x, \, 0) = 0$ in $\Omega$ by assumption. If $f(0) = 0$, we then obtain that
\begin{align*}
 \int_{\Omega} f(z(t)) \, dx + \int_{Q_t} (z \cdot f_z(z) - f(z)) \, \divv v \, dxds = \int_{Q_t} f_{z}(z) \cdot (-\divv(\sigma \, v_{x}) + R)  \, dxds \, .
\end{align*}
By means of a standard procedure, we approximate the function $f(z) = |z|^p$ by means of a sequence of smooth Lipschitz continuous functions $\{f_m\}$. This yields
\begin{align*}
  \int_{\Omega} |z(t)|^p \, dx + (p-1) \, \int_{Q_t} |z|^p \, \divv v \, dxds = p \, \int_{Q_t} |z|^{p-2} \, z \cdot (-\divv(\sigma \, v_{x}) + R)  \, dxds \, .
\end{align*}
The estimates below will establish that all members in the latter identity are finite. We first use H\"older's inequality and note that
\begin{align*}
\left| \int_{Q_t} |z|^{p-2} \, z \cdot \divv(\sigma \, v_{x})\right| \leq &  \int_{Q_t} |z|^{p-1} \, (|z| \, |v_x| + |\sigma| \, |v_{x,x}|)\\
 \leq & \int_{Q_t} |z|^{p} \, |v_x| \, dxds + \int_{0}^t \|v_{x,x}\|_{L^p(\Omega)} \, \|z\|_{L^p(\Omega)}^{p-1} \, \|\sigma\|_{L^{\infty}(\Omega)} \, ds \, .
\end{align*}
Next, we recall that for a solution to $ \partial_t \sigma + \divv(\sigma \, v + \hat{\varrho}_0 \, w) = 0$, the integral $\int_{\Omega} \sigma(t, \, x) \, dx$ is conserved and equal to zero. Due to the Poincar\'e inequality, we therefore have $\|\sigma(t)\|_{L^p(\Omega)} \leq c_0 \, \|z(t)\|_{L^p(\Omega)}$ and, by the Sobolev embedding, also that $\|\sigma(t)\|_{L^{\infty}(\Omega)}  \leq \tilde{c}_0 \, \|z(t)\|_{L^p(\Omega)}$. Thus,
\begin{align*}
\left| \int_{Q_t} |z|^{p-2} \, z \cdot \divv(\sigma \, v_{x})\right| \leq \int_0^t (\|v_x\|_{L^{\infty}(\Omega)} + \tilde{c}_0 \, \|v_{x,x}\|_{L^p(\Omega)}) \, \|z\|_{L^p(\Omega)}^{p} \, ds \, .
\end{align*}
Moreover, by Young's inequality,
\begin{align*}
& \int_{Q_t} |z|^{p-2} \, z \cdot R \, dxds  \leq \int_{0}^t \|z\|_{L^p(\Omega)}^{p} \, ds + c_p \, \int_{0}^t \|R\|_{L^p(\Omega)}^p \, ds \leq \int_{0}^t \|z\|_{L^p(\Omega)}^{p} \, ds \\
& \qquad + c_p \, \int_{0}^t [\|w\|_{L^{\infty}(\Omega)}^p \, \|\hat{\varrho}^0_{x,x}\|_{L^p(\Omega)}^p + 2 \, \|\hat{\varrho}^0_{x} \, w_x\|^p_{L^p(\Omega)} + \|\hat{\varrho}^0\|_{L^{\infty}(\Omega)}^p \, \|w_{x,x}\|_{L^p(\Omega)}^p] \, ds \, .
\end{align*}
Further,
\begin{align*}
 \int_{0}^t \|w\|_{L^{\infty}(\Omega)}^p \, \|\hat{\varrho}^0_{x,x}\|_{L^p(\Omega)}^p \, ds &\leq \|w\|_{L^{\infty}(Q_t)}^p \, \|\hat{\varrho}^0\|_{W^{2,0}_p(Q_t)}^p\, ,\\
 \int_{0}^t \|\hat{\varrho}^0_{x} \, w_x\|^p_{L^p(\Omega)} & \leq \|w_x\|_{L^{\infty,p}(Q_t)}^p \, \|\hat{\varrho}^0_x\|_{L^{p,\infty}(Q_t)}^p
 \leq C \, \|\hat{\varrho}^0_x\|_{L^{p,\infty}(Q_t)}^p \, \|w\|_{W^{2,0}_p(Q_t)}^p \, ,\\
 \int_{0}^t  \|\hat{\varrho}^0\|_{L^{\infty}(\Omega)}^p \, \|w_{x,x}\|_{L^p(\Omega)}^p \, ds &  \leq \|\hat{\varrho}^0\|_{L^{\infty}(Q_t)}^p \, \|w\|_{W^{2,0}_p(Q_t)}^p  \, .
\end{align*}
Thus,
\begin{align*}
  \int_{\Omega} |z(t)|^p \, dx \leq & (p-1) \, \int_0^t [\|v_x\|_{L^{\infty}(\Omega)} + \tilde{c}_0 \, \|v_{x,x}\|_{L^p(\Omega)} + p^{\prime}] \, \|z(s)\|_{L^p(\Omega)}^{p} \, ds\\
  & + p \, c_p \, [\|\hat{\varrho}^0\|_{W^{2,0}_p(Q_t)}^p \, \|w\|_{L^{\infty}(Q_t)}^p + \|\hat{\varrho}^0\|_{W^{1,1}_{p,\infty}(Q_t)}^p \,  \|w\|_{W^{2,0}_p(Q_t)}^p] \, .
\end{align*}
The claim follows by means of the Gronwall Lemma.
\end{proof}
We need next an estimate for the operators $(g^1)^{\prime}$ and $(f^1)^{\prime}$ from the right-hand side of \eqref{equationdiff1}, \eqref{equationdiff3}.
\begin{lemma}\label{RightHandControlsecond}
Let $\hat{u}_0 := (\hat{q}^0, \, \hat{\varrho}^0, \, \hat{v}^0) \in \mathcal{X}_{T,+}$ with $\hat{\varrho}^0 \in W^{2,0}_p(Q_T)$. Let $(r^*, \, w^*) \in \phantom{}_0\mathcal{Y}_T$, and $u^* := (\hat{q}^0 + r^*, \mathscr{C}(\hat{v}^0+w^*), \, \hat{v}^0+w^*) \in \mathcal{X}_{T,+}$ (cf. \eqref{ustar}). Let $(r, \, w) \in \phantom{}_0\mathcal{Y}_T$, and denote $\sigma $ the function obtained via solution of \eqref{linearT2second} with $v^* = \hat{v}^0+w^*$. We define $\bar{u} := (r, \, \sigma, \, w) \in \phantom{}_0\mathcal{X}_T$. Then the operators $(g^1)^{\prime}$ and $(f^1)^{\prime}$ in the right-hand of \eqref{equationdiff1}, \eqref{equationdiff3} satisfy
\begin{align*}
& \|(g^1)^{\prime}(u^*,\, \hat{u}^0) \, \bar{u}\|_{L^p(Q_t)}^p + \|(f^1)^{\prime}(u^*, \, \hat{u}^0) \, \bar{u}\|_{L^p(Q_t)}^p  \leq K_2^*(t) \,  \int_{0}^t \mathscr{V}^p(s) \, K^*_1(s) \, ds \, ,
\end{align*}
with functions $K^*_1 \in L^1(0,T)$ and $K_2^* \in L^{\infty}(0,T)$. There is a function $\Phi^* = \Phi^*(t, \, a_1,\ldots,a_5)$ defined for all $t, \, a_1, \ldots, a_5 \geq 0$, continuous and increasing in all arguments, such that for all $t \leq T$
\begin{align*}
 \|K^*\|_{L^1(0,t)} +  \|K^*_2\|_{L^{\infty}(0,t)} \leq \Phi^*(t, \, \mathscr{V}^*(t),\,  \|\hat{u}^0\|_{\mathcal{X}_t}, \, \|\hat{\varrho}^0\|_{W^{2,0}_p(Q_t)}, \, \|\tilde{b}\|_{W^{1,0}_p(Q_t)}, \, \|\bar{b}\|_{L^p(Q_t)}) \, .
\end{align*}
Here we used the abbreviations $\mathscr{V}(t) := \mathscr{V}(t; \, r) + \mathscr{V}(t; \, w)$ and $\mathscr{V}^*(t) := \mathscr{V}(t; \, r^*) + \mathscr{V}(t; \, w^*)$.
\end{lemma}
\begin{proof}
 At first we estimate $(g^1)^{\prime}$. Starting from \eqref{Arightlinear2}, we obtain by elementary means that
\begin{align*}
|(g^1)^{\prime}(u^*, \, \hat{u}^0) \, \bar{u}| \leq & |g^1_q(u^*, \, \hat{u}^0)| \, |r| +  |g^1_{\varrho}(u^*, \, \hat{u}^0)| \, |\sigma| + |g^1_{v}(u^*, \, \hat{u}^0)| \,|w|\\
& + |g^1_q(u^*, \, \hat{u}^0)| \, |r_x| +  |g^1_{\varrho}(u^*, \, \hat{u}^0)| \, |\sigma_x| + |g^1_{v}(u^*, \, \hat{u}^0)| \,|w_x| \, .
\end{align*}
We define $z = \frac{3p}{3-(5-p)^+}$, and by means of H\"older's inequality we obtain first that
\begin{align*}
 \|(g^1)^{\prime}(u^*, \, \hat{u}^0) \, \bar{u}\|_{L^p(Q_t)}^p \leq & \int_{0}^t \{\|g^1_q\|_{L^p(\Omega)}^p \, \|r\|_{L^{\infty}(\Omega)}^p + \|g^1_{q_x}\|_{L^z(\Omega)}^p \, \|r_x\|_{L^{\frac{3p}{(5-p)^+}}(\Omega)}^p\} \, ds \\
 & +\int_{0}^t \{\|g^1_v\|_{L^p(\Omega)}^p \, \|w\|_{L^{\infty}(\Omega)}^p + \|g^1_{v_x}\|_{L^z(\Omega)}^p \, \|w_x\|_{L^{\frac{3p}{(5-p)^+}}(\Omega)}^p\} \, ds\\
 & +  \int_{0}^t \{\|g^1_{\varrho}\|_{L^p(\Omega)}^p \, \|\sigma\|_{L^{\infty}(\Omega)}^p + \|g^1_{\varrho_x}\|_{L^{\infty}(\Omega)}^p \, \|\sigma_x\|_{L^p(\Omega)}^p\} \, ds \, .
 \end{align*}
Making use of the embeddings $W^{2-\frac{2}{p}}_p \subset L^{\frac{3p}{(5-p)^+}}$ and of $W^{1,p} \subset L^{\infty}(\Omega)$ (recall also that the means of $\sigma$ over $\Omega$ is zero at every time!), we show that
\begin{align*}
 \|(g^1)^{\prime}(u^*, \, \hat{u}^0) \, \bar{u}\|_{L^p(Q_t)}^p \leq & \int_{0}^t \sup_{\tau \leq s} \{\|r(\tau)\|_{W^{2-\frac{2}{p}}_p(\Omega)}^p + \|w(\tau)\|_{W^{2-\frac{2}{p}}_p(\Omega)}^p \} \, K_1(s) \, ds \\
& + \int_{0}^t K_2(s)\, \|\sigma_x(s)\|_{L^p(\Omega)}^p \, ds \, , \\
K_1(s) := & \|g^1_q(s)\|_{L^p(\Omega)}^p + C \,  \|g^1_{q_x}(s)\|_{L^z(\Omega)}^p  + \|g^1_v(s)\|_{L^p(\Omega)}^p +C \,  \|g^1_{v_x}(s)\|_{L^z(\Omega)}^p \, ,\\
K_2(s) :=  & C\, \|g^1_{\varrho}(s)\|_{L^p(\Omega)}^p + \|g^1_{\varrho_x}(s)\|_{L^{\infty}(\Omega)}^p \, .
\end{align*}
We invoke the Lemmas \ref{gedifferential}, \ref{gedifferentialun} to see that $K_1$ and $K_2$ are integrable functions and their norm are controlled by the data. Recall also that the minimum and the maximum of the function $\varrho^* := \mathscr{C}(\hat{v}^0+w^*)$, which enter the estimates via the coefficients, are controlled by a function of $\mathscr{V}(t; \, \hat{v}^0+w^*)$.

For the terms containing $\sigma_x$, we use the result of Lemma \ref{PerturbConti}. It yields for $s \leq t$ in particular that 
\begin{align*}
  \|\sigma(s)\|_{W^{1,p}(\Omega)} \leq &  K_3(s) \, \|w\|_{L^{\infty}(Q_s)}^p + K_4(s) \,  \|w\|_{W^{2,0}_p(Q_s)}^p \, ,\\
 K_3(s) := & C \, \exp(c \, \int_0^s (\|v^*_x\|_{L^{\infty}(\Omega)} + \|v^*_{x,x}\|_{L^p(\Omega)} + 1) \, d\tau) \,  \|\hat{\varrho}^0\|_{W^{2,0}_p(Q_s)}^p \, ,\\
 K_4(s) := & C \, \exp(c \, \int_0^s (\|v^*_x\|_{L^{\infty}(\Omega)} + \|v^*_{x,x}\|_{L^p(\Omega)} + 1) \, d\tau) \,  \|\hat{\varrho}^0\|_{W^{1,1}_{p,\infty}(Q_s)}^p \, .
\end{align*}
 We obtain that
\begin{align*}
\int_{0}^t K_2(s) \, \|\sigma_x(s)\|_{L^p(\Omega)}^p \, ds \leq  \max\{K_3(t), \,K_4(t)\} \, \int_{0}^t K_2(s) \, [\|w\|_{L^{\infty}(Q_s)}^p + \|w\|_{W^{2,0}_p(Q_s)}^p] \, ds \, .
\end{align*}
Overall, since $\|w\|_{L^{\infty}(Q_s)} \leq c \, \sup_{\tau \leq s} \|w(\tau)\|_{W^{2-\frac{2}{p}}_p(\Omega)}$, we obtain that
\begin{align*}
  \|(g^1)^{\prime}(u^*, \, \hat{u}^0) \, \bar{u}\|_{L^p(Q_t)}^p \leq & \int_{0}^t \sup_{\tau \leq s} \{\|r(\tau)\|_{W^{2-\frac{2}{p}}_p(\Omega)}^p + \|w(\tau)\|_{W^{2-\frac{2}{p}}_p(\Omega)}^p \} \, K_1(s) \, ds\\
& + \max\{K_3(t), \,K_4(t)\} \, \int_{0}^t K_2(s) \, [\|w\|_{L^{\infty}(Q_s)}^p + \|w\|_{W^{2,0}_p(Q_s)}^p] \, ds\\
\leq & c \, \max\{1, \, K_3(t), \,K_4(t)\} \, \int_{0}^t \mathscr{V}^p(s; \, w) \, (K_1(s)+K_2(s)) \, ds  \, .
\end{align*}
We can prove a similar result for $f_1^{\prime}$. This finishes to prove the estimate.
 \end{proof}

\section{Existence of a unique fixed-point of $\mathcal{T}^1$}\label{FixedPointT1}

We are now in the position to prove the continuity estimate for $\mathcal{T}^1$. We assume that $(r, \, \sigma, \, w)$ satisfy the equations \eqref{linearT1second},\eqref{linearT2second}, \eqref{linearT3second} with data $(r^*, \, w^*)$. We apply the Proposition \ref{A1linmain} to \eqref{linearT1second}, and making use of the fact that $r(0, \, x) =0$ in $\Omega$, we get an estimate
\begin{align}\label{difference4}
 \mathscr{V}(t; \, r) \leq & C \, \Psi_{1,t} \, \|g^1\|_{L^p(Q_t)} \leq C \, \Psi_{1,t} \, (\|\hat{g}^0\|_{L^p(Q_t)} + \|(g^1)^{\prime}(u^*,\hat{u}^0) \, \bar{u}\|_{L^p(Q_t)})  \, .
\end{align}
Here $\Psi_{1,t} = \Psi_1(t, \, (m^*(t))^{-1}, \, M^*(t),  \, \|q^0\|_{ W^{2-\frac{2}{p}}_p(\Omega)}, \, 
\mathscr{V}(t; \, q^*), \, [\varrho^*]_{C^{\beta,\frac{\beta}{2}}(Q_t)}, \, \|\nabla \varrho^*\|_{L^{p,\infty}(Q_t)})$, and $\bar{u} := (r, \, \sigma, \, w)$. We then apply the Proposition \ref{normvmain} to \eqref{linearT3second}, and we obtain that
\begin{align}\label{difference5}
 \mathscr{V}(t; \, w) \leq & C \, \tilde{\Psi}_{2,t} \, \|f^1\|_{L^p(Q_t)} \leq C \, \tilde{\Psi}_{2,t} \, (\|\hat{f}^0\|_{L^p(Q_t)} + \|(f^1)^{\prime}(u^*,\hat{u}^0) \, \bar{u}\|_{L^p(Q_t)}) \, .
\end{align}
Here $\tilde{\Psi}_{2,t} = \Psi_2(t, \, (m^*(t))^{-1}, \, M^*(t), \, \sup_{s\leq t} [\varrho^*(s)]_{C^{\alpha}(\Omega)}) \, (1+ \sup_{s\leq t} [\varrho^*(s)]_{C^{\alpha}(\Omega)})^{\frac{2}{\alpha}}$.

We next raise both \eqref{difference4} and \eqref{difference5} to the $p-$ power, add both inequalities, and get for the function $\mathscr{V}(t) := \mathscr{V}(t; \, r) + \mathscr{V}(t; \, w)$ an inequality 
\begin{align*}
 \mathscr{V}^p(t) \leq  C \, (\Psi_{1,t}^p + \tilde{\Psi}_{2,t}^p) \, (& \|\hat{g}^0\|^p_{L^p(Q_t)} + \|\hat{f}^0\|^p_{L^p(Q_t)}\\
 & +  \|(g^1)^{\prime}(u^*, \, \hat{u}^0) \, \bar{u}\|_{L^p(Q_t)}^p + \|(f^1)^{\prime}(u^*, \, \hat{u}^0) \, \bar{u}\|_{L^p(Q_t)}^p)  \, .
\end{align*}
Then we apply Lemma \ref{RightHandControlsecond} and find
\begin{align*}
  \mathscr{V}^p(t) \leq & C \, (\Psi_{1,t}^p + \tilde{\Psi}_{2,t}^p)  \, (\|\hat{g}^0\|^p_{L^p(Q_t)} + \|\hat{f}^0\|^p_{L^p(Q_t)} + K_2^*(t) \, \int_0^t K^*_1(s) \,   \mathscr{V}^p(s) \, ds) \, .
\end{align*}
The Gronwall Lemma implies that
\begin{align*}
\mathscr{V}^p(t) \leq  C \, (\Psi_{1,t}^p + \tilde{\Psi}_{2,t}^p) \, \exp(C \, (\Psi_{1,t}^p + \tilde{\Psi}_{2,t}^p) \, K^*_{2}(t) \, \int_0^t K^*_1(s)\, ds) \, (\|\hat{g}^0\|^p_{L^p(Q_t)} + \|\hat{f}^0\|^p_{L^p(Q_t)}) \, .
\end{align*}
We thus have proved the following continuity estimate:
\begin{prop}\label{estimateselfsecond}
 Suppose that $(r^*, \, w^*), \, (r, \, w) \in \phantom{}_0\mathcal{Y}_T$ are solutions to $(r, \, w) = \mathcal{T}^1(r^*, \, w^*)$. Then there is a continuous function $\Psi^9$ increasing in its arguments such that, for all $t \leq T$,
 \begin{align*}
  \mathscr{V}(t) \leq & \Psi_9(t, \, \|\hat{u}^0\|_{\mathcal{X}_t}+ \|\hat{\varrho}^0\|_{W^{2,0}_p(Q_t)} + \|\tilde{b}\|_{W^{1,0}_p(Q_t)}+ \|\bar{b}\|_{L^p(Q_t)}, \, \mathscr{V}^*(t)) \times \\
&  \times (\|\hat{g}^0\|_{L^p(Q_t)} + \|\hat{f}^0\|_{L^p(Q_t)}) \, .
 \end{align*}
\end{prop}
We are now in the position to prove a self-mapping property for sufficiently 'small data' applying the Lemma \ref{selfmapTsecond}.
\begin{lemma}
There is $R_1 > 0$ such that if $\|\hat{g}^0\|_{L^p(Q_T)} + \|\hat{f}^0\|_{L^p(Q_T)} \leq R_1$, the map $\mathcal{T}^1$ is well defined and possesses a unique fixed-point.
\end{lemma}
\begin{proof}
We apply the Lemma \ref{selfmapTsecond} with $\Psi(T, \, R_0, \, R_1, \, \eta) := \Psi_9(T, \, R_0, \, \eta) \, R_1$. Here $R_0 = \|\hat{u}^0\|_{\mathcal{X}_T}+ \|\hat{\varrho}^0\|_{W^{2,0}_p(Q_T)} + \|\tilde{b}\|_{W^{1,0}_p(Q_T)} + \|\bar{b}\|_{L^p(Q_T)}$. 

Thus, there is $R_1 > 0$ such that if $\|\hat{g}^0\|_{L^p(Q_T)} + \|\hat{f}^0\|_{L^p(Q_T)} \leq R_1$, we can find $\eta_0 > 0$ such that $\mathcal{T}^1$ maps the set $\{\bar{u} \in \phantom{}_0\mathcal{Y}_T \, : \,   \|\bar{u}\|_{\mathcal{Y}_T} \leq \eta_0\}$ into itself. 

Consider the iteration $\bar{u}^{n+1} := \mathcal{T}^1(\bar{u}^n)$ starting at $\bar{u}^n = 0$. The sequences $(r^{n}, \, \sigma^n, \, w^n)$, and thus also $(\hat{q}^0+r^n, \, \mathscr{C}(\hat{v}^0 + w^n ), \, \hat{v}^0 + w^n)$, are uniformly bounded in $\mathcal{X}_T$. We show the contraction property with respect to the same lower-order norm than in Theorem \ref{iter}. There are $k_0, \, p_0 > 0$ such that the quantities 
 \begin{align*}
 &  E^n(t) :=   p_0 \, \int_{t}^{t+t_1} \{|\nabla (r^{n}-r^{n-1})|^2 + |\nabla (w^{n}-w^{n-1})|^2\} \, dxds\\
  &+  k_0 \, \sup_{\tau \in [t, \, t + t_1]} \{\|(r^n-r^{n-1})(\tau)\|_{L^2(\Omega)}^2 + \|(\sigma^n-\sigma^{n-1})(\tau)\|_{L^2(\Omega)}^2 +\|(w^n-w^{n-1})(\tau)\|_{L^2(\Omega)}^2\}
 \end{align*}
satisfy $E^{n+1}(t) \leq\tfrac{1}{2} \, E^{n}(t)$ for some fixed $t_1 > 0$ and every $t \in [0, \, T-t_1]$. 
\end{proof}
In order to finish the proof of Theorem \ref{MAIN3}, we want to show how to make $\|\hat{g}^0\|_{L^p(Q_T)} + \|\hat{f}^0\|_{L^p(Q_T)}$ small. We observe that $\hat{g}^0 = \widetilde{ \mathscr{A}}^1(\hat{u}^0)$ and that $\hat{f}^0 = \mathscr{A}^3(\hat{u}^0)$. Thus, if an equilibrium solution to $\mathscr{A}(u^{\text{eq}}) = 0$ is at hand, we can expect that $\mathscr{A}(\hat{u}^0) = \mathscr{A}(\hat{u}^0) - \mathscr{A}(u^{\text{eq}})$ will remain small if the initial data are near to the equilibrium solution.

We thus consider $u^{\text{eq}} = (q^{\text{eq}}, \, \varrho^{\text{eq}}, \, v^{\text{eq}}) \in W^{2,p}(\Omega; \, \mathbb{R}^{N-1}) \times W^{1,p}(\Omega) \times W^{2,p}(\Omega; \, \mathbb{R}^{3})$ an equilibrium solution. This means that the equations \eqref{massstat}, \eqref{momentumsstat} are valid with the vector $\rho^{\text{eq}}$ of partial mass densities obtained from $q^{\text{eq}}$ and $\varrho^{\text{eq}}$ by means of the transformation of Section \ref{changevariables}.
\begin{lemma}
Suppose that $u^{\text{eq}} \in W^{2,p}(\Omega; \, \mathbb{R}^{N-1}) \times W^{1,p}(\Omega) \times W^{2,p}(\Omega; \, \mathbb{R}^{3})$ is an equilibrium solution. Moreover, we assume that the initial data $u^0$ belongs to $\text{Tr}_{\Omega \times\{0\}} \, \mathcal{X}_T$. We assume that the components $\varrho^{\text{eq}}, \, \varrho^0$ and $v^0$ of $u^{\text{eq}}$ and $u^0$ possess the additional regularity 
\begin{align}\label{moreregu}
 \varrho^{\text{eq}}, \, \varrho^0 \in  W^{2,p}(\Omega), \quad  v^{\text{eq}} \in W^{3,p}(\Omega; \, \mathbb{R}^3), \, v^0 \in W^{2,p}(\Omega; \, \mathbb{R}^3)   \, .
\end{align}
Then, there exists $R_1 > 0$ such that if $\|u^{\text{eq}}-u^0\|_{\text{Tr}_{\Omega \times\{0\}} \, \mathcal{X}_T} \leq R_1$, then there is a unique global solution $u \in \mathcal{X}_T$ to $\mathscr{A}(u) = 0$ and $u(0) = u_0$.
\end{lemma}
\begin{proof}
We denote $u^1 :=  u^{\text{eq}}-u^0 \in \text{Tr}_{\Omega \times\{0\}} \, \mathcal{X}_T$. We find extensions $\hat{q}^1 \in W^{2,1}_p(Q_T; \, \mathbb{R}^{N-1})$ and $\hat{v}^1 \in W^{2,1}_p(Q_T; \, \mathbb{R}^{3})$ with continuity estimates. For instance, we can extend the components of $q^1, \, v^1$ to elements of $W^{2-2/p}_{p}(\mathbb{R}^3)$, and then solve Cauchy-problems for the heat equation to extend the functions. Since the assumption \eqref{moreregu} moreover guarantees that $v^1 \in W^{2,p}(\Omega)$, this procedure yields even $\hat{v}^1 \in W^{4,2}_p(Q_T; \, \mathbb{R}^3)$ at least (cf. \cite{ladu}, Chapter 4, Paragraph 3, inequality (3.3)).

The definitions $\hat{q}^{\text{eq}}(x, \, t) := q^{\text{eq}}(x)$ and $\hat{v}^{\text{eq}}(x, \, t) := v^{\text{eq}}(x)$ provide extensions of $\hat{q}^{\text{eq}} \in W^{2,\infty}_{p,\infty}$ and $\hat{v}^{\text{eq}}$ in $W^{3,\infty}_{p,\infty}$. We define
\begin{align*}
 \hat{q}^0 := \hat{q}^{\text{eq}} + \hat{q}^1 \in W^{2,1}_p(Q_T; \, \mathbb{R}^{N-1}), \quad \hat{v}^0 := \hat{v}^{\text{eq}} + \hat{v}^1 \in W^{2,1}_p(Q_T; \, \mathbb{R}^{N-1}) \cap W^{3,0}_p(Q_T; \, \mathbb{R}^3) \, ,
\end{align*}
satisfying
\begin{align}\label{distance}
&  \|\hat{q}^0 - \hat{q}^{\text{eq}}\|_{W^{2,1}_p(Q_T)} + \|\hat{v}^0 - \hat{v}^{\text{eq}}\|_{W^{2,1}_p(Q_T)} \leq C \, (\|q^1\|_{W^{2-\frac{2}{p}}_p(\Omega)} +  \|v^1\|_{W^{2-\frac{2}{p}}_p(\Omega)}) = C \, R_1 \, ,\\
& \label{extendbetter}
 \|\hat{v}^0\|_{W^{3,0}_p(Q_T; \, \mathbb{R}^3)} \leq C \, (\|v^{\text{eq}}\|_{W^{3,p}(\Omega)} + \|v^0\|_{W^{2,p}(\Omega)}) \, .
\end{align}
In order to extend $\varrho^0$, we solve $\partial_t \hat{\varrho}^0 + \divv( \hat{\varrho}^0 \, \hat{v}^0) = 0$ with initial condition $\hat{\varrho}^0 = \varrho^0$. We clearly obtain by these means an extension of class $W^{1,1}_{p,\infty}(Q_T)$. Moreover, due to \eqref{extendbetter}, we can show that $\hat{\varrho}^0 \in W^{2,0}_p(Q_T)$ (use the representation formula at the beginning of the proof of Prop. \ref{solonnikov2}).
If we next choose the extension $ \hat{\varrho}^{\text{eq}}(x, \, t) := \varrho^{\text{eq}}(x) \in W^{2,\infty}_{p,\infty}(Q_T)$, then by definition of the equilibrium solution we have $\divv (\hat{\varrho}^{\text{eq}} \, \hat{v}^{\text{eq}}) = 0$ in $Q_T$ and $\partial_t \hat{\varrho}^{\text{eq}} = 0$.

Thus, the difference $\hat{\varrho}^1 := \hat{\varrho}^0 - \hat{\varrho}^{\text{eq}}$ is a solution to $ \partial_t \hat{\varrho}^1 + \divv (\hat{\varrho}^1 \, \hat{v}^0) = - \divv(\hat{\varrho}^0 \, \hat{v}^{1})$. Since $\hat{\varrho}^0 \in W^{1,1}_{p,\infty}(Q_T) \cap W^{2,0}_p(Q_T)$ by construction, the estimate of Lemma \ref{PerturbConti} applies, and invoking also \eqref{distance} this gives
\begin{align*}
 \|\hat{\varrho}^1\|_{W^{1,1}_{p,\infty}(Q_T)} \leq & C \, \exp(c \, \int_0^T [\|\hat{v}^0_x\|_{L^{\infty}(\Omega)} + \|\hat{v}^0_{x,x}\|_{L^p(\Omega)} + 1]  ds) \times\\
 & \times
  (\|\hat{\varrho}^0\|_{W^{2,0}_p(Q_T)}^p \, \|\hat{v}^1\|_{L^{\infty}(Q_T)}^p + \|\hat{\varrho}^0\|_{W^{1,1}_{p,\infty}(Q_T)}^p \,  \|\hat{v}^1\|_{W^{2,0}_p(Q_T)}^p)\\
   \leq & C_T \, \|\hat{v}^1\|_{W^{2,1}_p(Q_T)} \leq C_T \, \, R_1 \, .
\end{align*}
The latter and \eqref{distance} now entail that
\begin{align*}
 \|\hat{q}^0 - \hat{q}^{\text{eq}}\|_{W^{2,1}_p(Q_T)} + \|\hat{v}^0 - \hat{v}^{\text{eq}}\|_{W^{2,1}_p(Q_T)} + \|\hat{\varrho}^0-\hat{\varrho}^{\text{eq}}\|_{W^{1,1}_{p,\infty}(Q_T)} \leq C \, R_1 \, .
\end{align*}
Here $C$ is allowed to depend on $T$ and all data in their respective norm. Now recalling the Lemma \ref{IMAGESPACE} we can verify that
\begin{align*}
\widetilde{ \mathscr{A}}(\hat{u}^0) =  & \widetilde{ \mathscr{A}}(\hat{u}^{\text{eq}} + \hat{u}^1) = \widetilde{ \mathscr{A}}(\hat{u}^{\text{eq}} + \hat{u}^1) - \widetilde{ \mathscr{A}}(\hat{u}^{\text{eq}})= \int_{0}^1  \widetilde{ \mathscr{A}}^{\prime}(\hat{u}^{\text{eq}} + \theta \, \hat{u}^1) \, d\theta \, \hat{u}^1 \, .
\end{align*}
Thus $ \|\widetilde{ \mathscr{A}}(\hat{u}^0)\|_{\mathcal{Z}_T} \leq C \, R_1$. The definitions of $\hat{g}^0$ and $\hat{f}^0$ in \eqref{Arightlinear2} show that
\begin{align*}
 \|\hat{g}^0\|_{L^p(Q_T)} + \|\hat{f}^0\|_{L^p(Q_T)} =  \|\widetilde{ \mathscr{A}}^1(\hat{u}^0)\|_{L^p(Q_T)} + \|\widetilde{ \mathscr{A}}^3(\hat{u}^0)\|_{L^p(Q_T)} \leq C \, R_1 \, . 
\end{align*}
The claim follows from the Lemma \ref{estimateselfsecond}.
\end{proof}

\appendix

\section{Examples of free energies}\label{Legendre}

1. We consider first $h(\rho) := \sum_{i=1}^N n_i \, \ln \frac{n_i}{n^{\text{ref}}}$ where for $i = 1,\ldots,N$, the mass and number densities are related via $m_i \, n_i =  \rho_i$ with a positive constant $m_i > 0$. We want to show that $h$ is a \emph{Legendre function} on $\mathbb{R}^N_+$. It is at first clear that $h$ is continuously differentiable on $\mathbb{R}^N_+$, and we even have $h \in C^{\infty}(\mathbb{R}^N_+)$. The strict convexity of $h$ is obviously inherited from the strict convexity of $t \mapsto t \, \ln t$ on $\mathbb{R}_+$. The gradient of $h$ is given by 
\begin{align*}
 \partial_{\rho_i} h (\rho) = \frac{1}{m_i} \, (1+ \ln \frac{n_i}{n^{\text{ref}}}) \, .
\end{align*}
Thus $\lim_{k \rightarrow \infty} |\nabla_{\rho} h(\rho^k)| = + \infty$ whenever $\{\rho^k\}_{k \in \mathbb{N}}$ is a sequence of points approaching the boundary of $\mathbb{R}^N_+$. Overall we have shown that $h$ is a continuously differentiable, strictly convex, essentially smooth function on $\mathbb{R}^N_+$ (where \emph{essentially smooth} precisely means the blow-up of the gradient on the boundary). Functions satisfying these properties are called of \emph{Legendre type} (cf. \cite{rockafellar}, page 258).
Moreover, we can directly show that the gradient of $h$ is surjective onto $\mathbb{R}^N$, since the equations $\partial_{\rho_i} h = \mu_i$ have the unique solution $\rho_i = m_i \, n^{\text{ref}} \, e^{m_i \, \mu_i -1}$ for arbitrary $\mu \in \mathbb{R}^N$.

2. The second example is $h(\rho) = F\left(\sum_{i=1}^N n_i \, \bar{v}_i^{\text{ref}}\right) + \sum_{i=1}^N n_i \, \ln \frac{n_i}{n}$, with the total number density $n = \sum_{j=1}^N n_j$. Here $F$ is a given convex function of class $C^2(\mathbb{R}_+)$. We assume that
\begin{itemize}
\item $F^{\prime\prime}(t) > 0$ for all $t > 0$;
\item $F^{\prime}(t) \rightarrow -\infty$ for $t \rightarrow 0$;
\item $\frac{1}{t} \, F(t) \rightarrow + \infty$ for $t \rightarrow +\infty$. 
\end{itemize}
In other words, $F$ is a co-finite function of Legendre type on $\mathbb{R}_+$. The numbers $\bar{v}^{\text{ref}}_i$ are positive constants. Choosing $\bar{v}_1^{\text{ref}} = \ldots = \bar{v}_N^{\text{ref}} = 1$ and $F(t) = t \, \ln t$ we recover the preceding example. 

The function $h$ is clearly of class $C^2(\mathbb{R}^N_+)$. We compute the derivatives
\begin{align*}
 \partial_{\rho_i}h(\rho) = & F^{\prime}(v \cdot \rho) \, v_i + \frac{1}{m_i} \, \ln \frac{n_i}{n}\, ,\\
 \partial_{\rho_i,\rho_j}h(\rho) = & F^{\prime\prime}(v \cdot \rho) \, v_i \, v_j+ \frac{1}{m_i\, m_j} \, (\frac{\delta_{i,j}}{n_j} - \frac{1}{n}) \, ,
\end{align*}
in which we have for simplicity set $v_i := \bar{v}^{\text{ref}}_i/m_i$. For $\xi \in \mathbb{R}^N$, we verify that 
\begin{align*}
 D^2h(\rho) \xi \cdot \xi = F^{\prime\prime}(v \cdot \rho) \, (v \cdot \xi)^2 + \sum_{i=1}^N \left(\frac{\xi_i}{\sqrt{n_i} \, m_i}\right)^2 - \frac{1}{n} \, \left(\sum_{i=1}^N \frac{\xi_i}{m_i}\right)^2 \, .
\end{align*}
With the Cauchy-Schwarz inequality, we see that $\sum_{i=1}^N \left(\frac{\xi_i}{\sqrt{n_i} \, m_i}\right)^2 - \frac{1}{n} \, \left(\sum_{i=1}^N \frac{\xi_i}{m_i}\right)^2 \geq 0$, with equality only if $\xi_i = \lambda \, n_i \, m_i$ for some $\lambda \in \mathbb{R}$. In this case however, we have $\xi \cdot v = \lambda \, \sum_{i=1}^N \rho_i \, v_i$, so that $ D^2h(\rho) \xi \cdot \xi = \lambda^2 \, F^{\prime\prime}(v \cdot \rho) \, (v \cdot \rho)^2 \geq 0$, with equality only if $\lambda = 0$. This proves that $  D^2h(\rho) \xi \cdot \xi > 0$ for all $\xi \in \mathbb{R}^N \setminus \{0\}$, which implies the strict convexity. 

In order to show that $h$ is essentially smooth, we consider a sequence $\{\rho^k\}_{k \in \mathbb{N}}$ approaching the boundary of $\mathbb{R}^N_+$. We first consider the case that $\rho^k$ does not converge to zero. In this case, we clearly have $\inf_{k \in \mathbb{N}}\{n^k, \, v \cdot \rho^k\} \geq c_0$ for some positive constant $c_0$. Thus
\begin{align*}
|\nabla_{\rho} h(\rho^k)| \geq & \frac{1}{\max m} \, \sup_{i=1,\ldots,N} |\ln n_i^k| - |v| \, \sup_{k} |F^{\prime}(v \cdot \rho^k)| - \frac{1}{\min m} \sup_{k} |\ln n^k|\\
\geq & \frac{1}{\max m} \, \sup_{i=1,\ldots,N} |\ln n_i^k| - C \rightarrow + \infty \, .
\end{align*}
The second case is that $\rho^k$ converges to zero. In this case the fractions $\frac{n^k_i}{n^k}$ might remain all bounded. But our assumptions on $F$ guarantee that $F^{\prime}(v \cdot \rho^k) \rightarrow - \infty$ so that $|\nabla_{\rho} h(\rho^k)| \rightarrow + \infty$. Thus, the function $h$ is essentially smooth, and a function of Legendre type on $\mathbb{R}^N_+$.

It remains to prove that $\nabla_{\rho} h$ is a surjective mapping. We first verify that $h$ is co-finite. In the present context, it is sufficient to show that $\lim_{\lambda \rightarrow +\infty} \frac{h(\lambda \, y)}{\lambda} = + \infty$ for all $y \in \mathbb{R}^N_+$. This follows directly from the fact that $\lim_{t \rightarrow +\infty} \frac{F(t)}{t} = + \infty$. We then infer the surjectivity of $\nabla_{\rho} h$ form Corollary 13.3.1 in \cite{rockafellar}. 

3. Similar arguments allow to deal with the case $h(\rho) =  \sum_{i=1}^N K_i \, n_i\, \bar{v}_i^{\text{ref}} \, ((n_i\, \bar{v}_i^{\text{ref}})^{\alpha_i-1} + \ln (n_i\, \bar{v}_i^{\text{ref}})) +  k_B \, \theta \, \sum_{i=1}^N n_i \, \ln \frac{n_i}{n}$.

\section{Proof of the Lemma \ref{A1linUMFORM}}\label{technique}

The argument is based on covering $Q_t$, $t >0$, with sufficiently small sets and localising the problem therein. This is essentially carried over by standard techniques of meshing, so we will spare these rather technical considerations. For every $r > 0$, we can find $m = m(r) \in \mathbb{N}$ and, for each $j= 1,\ldots,m$, a point $(x^j, \, t_j) \in Q_t$ and sets $Q^j$ that possess the following properties:
\begin{itemize}
 \item $\overline{Q_t} \subset \bigcup_{j=1}^m Q^j$;
  \item $\sup_{(x, \, t) \in Q^j} |t-t_j| \leq c \, r$ and $\sup_{(x, \, t) \in Q^j} |x-x^j| \leq c \, \sqrt{r}$;
  \item $Q^j$ intersects a finite number, not larger than some $m_0 \in \mathbb{N}$, of elements of the collection $Q^1, \ldots, Q^m$. Here $m_0$ is independent on $r$ and $t$.
 \end{itemize}
 For $j = 1,\ldots,m$, we can moreover choose a non-negative function $\eta^j \in C^{2,1}(Q^j)$ with support in $Q^j$. The family $\eta^1,\ldots,\eta^m$ is assumed to nearly provide a partition of unity, that is, to possess the following properties:
\begin{align*}\begin{split}
& c_0 \leq \sum_{j=1}^m \eta^j(x, \, t) \leq C_0 \text{ for all } (x, \, t) \in \overline{Q}_t\\
& \|\eta_x\|_{L^{\infty}(Q^j)} \leq C_1 \, r^{-\frac{1}{2}}, \quad \|\eta_t\|_{L^{\infty}(Q^j)} + \|\eta_{x,x}\|_{L^{\infty}(Q^j)} \leq C_2 \, r^{-1} \, .
\end{split}
\end{align*}
Here $c_0 > 0$ and $C_i$ ($i=0,1,2$) are constants independent on $r$ and $t$. Moreover we can also enforce that $\nu \cdot \nabla \eta^j = 0$ on $S^j = : Q^j \cap (\partial \Omega \times [0, \, + \infty[)$. 
We let $\Omega^j := Q^j \cap (\Omega \times\{0\})$. 

After inversion of $R_q^*$, the vector field $q$ satisfies the equations \eqref{petrovskisyst}, that is
\begin{align}\label{Petrovski2}
 q_t - \underbrace{[R_q^*]^{-1} \, \widetilde{M}^*}_{=:A^*} \, \triangle q = [R_q^*]^{-1} \, g + [R_q^*]^{-1} \, \nabla \widetilde{M}^* \cdot \nabla q \, =: \tilde{g} \, .
\end{align}
Multiplying \eqref{Petrovski2} with $\eta^j$, we next derive the identities
\begin{align}\label{transformee}
 \eta^j \, q_t - \eta^j \, A^j \, \triangle q  = & \eta^j \, \tilde{g} + \eta^j \, (A^* - A^j  )\, \triangle q, \qquad \quad  A^j :=  [R_q^*]^{-1} \, \widetilde{M}^*(x^j, \, t_j) \, . 
\end{align}
Making use of the Lemma \ref{ALGEBRA}, the eigenvalues $p_1^j, \ldots, p_{N-1}^j$ of $A^j$ are real and strictly positive. Recalling \eqref{EVPROD}, we have on $[0, \, t]$ the bound
\begin{align*}
\frac{\lambda_{0}(t, \, \widetilde{M}^*)}{\lambda_{1}(t, \, R_q^*)}  \leq p_i^j \leq   \frac{\lambda_{1}(t, \, \widetilde{M}^*)}{\lambda_{0}(t, \, R_q^*)} \, .
\end{align*}
Further, there exists a basis $\xi^{1}, \, \ldots, \xi^{N-1} \in \mathbb{R}^{N-1}$ of eigenvectors of $A^j$.

For $i = 1,\ldots, N-1$, we multiply the equation \eqref{transformee} with $\xi^i$. For $u_{i} := \xi^{i} \cdot q$ ($i = 1,\ldots, N-1$) we therefore obtain that
$\eta^j \, (u_{i,t} - p_i^j \, \triangle u_i) = \eta^j \, \xi^i \cdot (\tilde{g} + (A^* - A^j  )\, \triangle q )$. We define $\tilde{u}^j_i := u_i \, \eta^j$, and for this function we obtain that
\begin{align*}
\tilde{u}^j_{i,t} - p_i^j \, \triangle \tilde{u}^j_i = & h^j_i := \eta^j \, \xi^i \cdot (\tilde{g} + (A^* - A^j  )\, \triangle q ) + \eta^j_t \, u_i  - p_i^j \, (2 \,  u_{i,x}  \cdot \eta^j_x  + \triangle \eta^j \, u_i) \, .
\end{align*}
Recall that $\nu \cdot \nabla q = 0$ on $S_T$ and $q(\cdot, \, 0) = q_0$ in $\Omega$. Due to our restrictions on the choice of $\eta^j$, we then readily compute that $\nu \cdot \nabla \tilde{u}^j_i = \nu \cdot \nabla \eta^j \, \xi^i \cdot q =  0$ on $S_t$. Moreover, $\tilde{u}^i_j(0) = \eta^j(0) \, \xi^i \cdot q_0 =: \tilde{u}^{0,j}_i$ in $\Omega$. Since $q^0$ satisfies the initial compatibility condition, also $\tilde{u}^{0,j}_i$ is a compatible data. Standard results for the heat equation now yield for arbitrary $t \leq T$
\begin{align}\label{estimatetildeuij}
 \|\tilde{u}^j_i\|_{W^{2,1}_p(Q_t)} \leq C_1 \, \frac{\max\{1, \, p^j_i\}}{\min\{1, \, p^j_i\}} \, (p^j_i \, \|\tilde{u}^{0,j}_i\|_{W^{2-\frac{2}{p}}_p(\Omega)} + \|h^j_i\|_{L^p(Q_t)}) \, ,
\end{align}
where $C_1$ depends only on $\Omega$ (see the Remark \ref{Tindependence}). In order to estimate $\|h^j_i\|_{L^p(Q_t)}$, we introduce $Q^j_t := Q^j \cap Q_t$ and observe that
\begin{align}\label{estimatehij1}
 \|\eta^j \, \xi^i \cdot \tilde{g}\|_{ L^p(Q_t)} \leq & C_0 \, \| \tilde{g}\|_{ L^p(Q^j_t)} \nonumber \, , \\
\|(\eta^j_t -  p_i^j \,\triangle \eta^j) \, u_i\|_{ L^p(Q_t)} \leq & C \, (1+\lambda_{\max}(A_j)) \, r^{-1} \, \|q\|_{L^p(Q^j_t)} \, ,\nonumber\\
 2 \, p_i^j \, \|u_{i,x}  \cdot \eta^j_x\|_{ L^p(Q_t)} \leq & C \,  \, \lambda_{\max}(A_j) \,  r^{-1/2} \, \|q_x\|_{L^p(Q^j_t)} \leq C \, \lambda_{\max}(A_j) \, (1+r^{-1}) \,\|q_x\|_{L^p(Q^j_t)} \, , \nonumber\\
\|\eta^j \, \xi^i \, (A^* - A^j  )\, \triangle q \|_{L^p(Q_t)} \leq & C \, \|A^* - A^j\|_{L^{\infty}(Q^j_t)} \, \|D^2q\|_{L^p(Q^j_t)} \, .
\end{align}
Since, by definition, $A^j = [R_q^*]^{-1} \, \widetilde{M}^*(x^j, \, t_j) = A^*(x^j, \, t_j)$, we have
\begin{align}\label{estimatehoelderprod}
\|A^* - A^j\|_{L^{\infty}(Q^j_t)} \leq \Big[[R_q^*]^{-1} \, \widetilde{M}^*\Big]_{C^{\beta,\frac{\beta}{2}}(Q_t)} \, r^{\frac{\beta}{2}} \, .
\end{align}
We call $F: \mathbb{R}_+ \times \mathbb{R}^{N-1} \rightarrow  \mathbb{R}^{N-1}\times \mathbb{R}^{N-1}$ the map $(s, \, \xi) \mapsto [R_q(s, \, \xi)]^{-1} \, \widetilde{M}(s, \, \xi)$. The derivatives of $F$ satisfy the estimates
\begin{align*}
&  |\partial_s F(\varrho^*, \, q^*)| \leq \frac{\lambda_{\max}( \widetilde{M}^*)}{\lambda_{\min}^2(R_q^*)} \, |R_{q,\varrho}(\varrho^*, \, q^*)| + \frac{1}{\lambda_{\min}(R_q^*)} \, |\widetilde{M}_{\varrho}(\varrho^*, \, q^*)| \, ,\\
& |\partial_{\xi}F(\varrho^*, \, q^*)| \leq \frac{\lambda_{\max}( \widetilde{M}^*)}{\lambda_{\min}^2(R_q^*)} \, |R_{q,q}(\varrho^*, \, q^*)| + \frac{1}{\lambda_{\min}(R_q^*)} \, |\widetilde{M}_{q}(\varrho^*, \, q^*)| \, ,\\
 & |\partial_s F(\varrho^*, \, q^*)| + |\partial_{\xi}F(\varrho^*, \, q^*)| \leq \left(\frac{\lambda_{1}(t,\, \widetilde{M}^*)}{\lambda_{0}^2(t, \, R_q^*)} + \frac{1}{\lambda_{0}(t, \, R_q^*)}\right) \, (L^*(t, \, R_q) + L^*(t, \, \widetilde{M})) =: \ell^*_t \,  .
\end{align*}
By standard arguments, we have
\begin{align}\label{estimatehoelderprod2}
\Big[[R_q^*]^{-1} \, \widetilde{M}^*\Big]_{C^{\beta,\frac{\beta}{2}}(Q_t)} \leq \ell^*_t \, ([\varrho^*]_{C^{\beta,\frac{\beta}{2}}(Q_t)} + [q^*]_{C^{\beta,\frac{\beta}{2}}(Q_t)}) \, \, .
\end{align}
Combining \eqref{estimatetildeuij}, \eqref{estimatehij1}, \eqref{estimatehoelderprod} and \eqref{estimatehoelderprod2}, we get
\begin{align*} 
  & \|\tilde{u}^j_i\|_{W^{2,1}_p(Q_t)} \leq  \tilde{C}_1\,  \frac{\max\{1, \, \lambda_{\max}(A^j)\}}{\min\{1, \, \lambda_{\min}(A^j)\}}\,  \times \\
& \Big( \big[(1+\lambda_{\max}(A^j)) \, (1+r^{-1}) \, \{\|q^0\|_{W^{2-\frac{2}{p}}_p(\Omega^j; \, \mathbb{R}^{N-1})} +\|q\|_{W^{1,0}_p(\tilde{Q}^j)} \} +  \| \tilde{g}\|_{ L^p(\tilde{Q}^j)}\big] \\
 & \quad + \ell^*_t \, ([\varrho^*]_{C^{\beta,\frac{\beta}{2}}(Q_t)} + [q^*]_{C^{\beta,\frac{\beta}{2}}(Q_t)}) \, r^{\frac{\beta}{2}} \, \|D^2q\|_{L^p(\tilde{Q}^j)} \Big)\, .
\end{align*}
Recall now that $\tilde{u}^j_i = \eta^j \, \xi^i \cdot q = \xi^i \cdot w^j$ with $w^j = \eta^j \, q$. Here $\{\xi^i\}$ are eigenvectors of $A^j$ and form a basis of $\mathbb{R}^{N-1}$. It is moreover shown in the Lemma \ref{ALGEBRA} that there are orthonormal vectors $v^1,\ldots,v^{N-1}$ such that $\xi^i = B^j \, v^i$ for $i =1,\ldots,N-1$, where $B^j := [R_q(\varrho^*(x^j,t_j), \, q^*(x^j,\, t_j))]^{\frac{1}{2}}$. Thus, $w^j \cdot \xi^i = B^j w^j \cdot v^i$. For any norm $\|\cdot \| $ on vector fields of length $N-1$ defined on $Q_t$ we then have
\begin{align*}
& \|q\, \eta^j\| = \|w^j\| = \|[B^j]^{-1} \,  B^j w^j \| \leq  |[B^j]^{-1}|_{\infty} \, \|B^j w^j\| = |[B^j]^{-1}|_{\infty} \, \|\sum_{i=1}^{N-1} (v^i \cdot B^j w^j) \, v^i\| \\
& \quad \leq  c \, |[B^j]^{-1}|_{\infty}  \, \sum_{i=1}^{N-1} \|B^j \, w^j \cdot v^i\| = c \, |[B^j]^{-1}|_{\infty}  \, \sum_{i=1}^{N-1} \|w^j \cdot \xi^i\|\\
& \quad \leq  \frac{c_1}{[\lambda_{\min}(R_q(\varrho^*(x^j,t_j), \, q^*(x^j,\, t_j)))]^{\frac{1}{2}}}  \, \sum_{i=1}^N   \|\tilde{u}^j_i\|\leq 
\frac{c_1}{[\lambda_{0}(t, \, R_q^*)]^{\frac{1}{2}}}  \, \sum_{i=1}^N   \|\tilde{u}^j_i\| \, .
\end{align*}
Choosing $\|\cdot \| = \|\cdot \|_{W^{2,1}_p(Q_t; \, \mathbb{R}^{N-1})}$, it follows that
\begin{align*}
  & \|q \, \eta^j\|_{W^{2,1}_p(Q_t; \, \mathbb{R}^{N-1})} \leq \bar{C}_1 \, 
  \frac{\max\{1, \, \lambda_{\max}(A^j)\}}{[\lambda_{0}(t, \, R_q^*)]^{\frac{1}{2}}\, \min\{1, \, \lambda_{\min}(A^j)\}} \times\\
  & \Big(\, \big[(1+\lambda_{\max}(A^j)) \, (1+r^{-1}) \, \{\|q^0\|_{W^{2-\frac{2}{p}}_p(\Omega^j; \, \mathbb{R}^{N-1})} +\|q\|_{W^{1,0}_p(Q^j_t)} \} +  \| \tilde{g}\|_{ L^p(Q^j_t)}\big]\\
 & \quad +\ell^*_t \, ([\varrho^*]_{C^{\beta,\frac{\beta}{2}}(Q_t)} + [q^*]_{C^{\beta,\frac{\beta}{2}}(Q_t)}) \, r^{\frac{\beta}{2}} \, \|D^2q\|_{L^p(Q^j_t)}\Big) \, .
\end{align*}
We estimate $\lambda_{\min}(A^j) \geq \lambda_0([R_q^*]^{-1}\widetilde{M}^*)$ etc. We recall that for each $j$ there are at most $m_0$ indices $i_1,\ldots,i_{m_0} \neq j$ such that $Q^j \cap Q^{i_k} \neq \emptyset $. Thus $\sum_{j=1}^m  \|f\|_{L^p(Q^j_t)} \leq (m_0 + 1) \, \|f\|_{L^p(Q_t)}$. We easily verify that
\begin{align*}
\|q \, \eta^j\|_{W^{2,1}_p(Q_t)} \geq & \|q_t \, \eta^j\|_{L^p(Q_t)} +\sum_{0\leq \alpha\leq 2} \|D^{\alpha}_xq \, \eta^j\|_{L^p(Q_t)} - c \, r^{-\frac{1}{2}} \, \|q_x\|_{L^p(Q^j_t)} - c \, r^{-1} \, \|q\|_{L^p(Q^j_t)} \, . 
\end{align*}
After summing up for $j = 1,\ldots,m$ and using the properties of our covering again, we obtain
\begin{align*}
& \|q\|_{W^{2,1}_p(Q_t; \, \mathbb{R}^{N-1}))} \leq\bar{C}_1 \, 
   \frac{ \max\{1, \, \lambda_{1}(t, \, [R_q^*]^{-1}\widetilde{M}^*)\}}{\lambda_{0}^{\frac{1}{2}}(t, \, R_q^*)\, \min\{1, \, \lambda_{0}(t, \, [R_q^*]^{-1}\widetilde{M}^*)\}}\, \times\\
 & \Big( \big[(1+\lambda_{1}(t, \, [R_q^*]^{-1}\widetilde{M}^*)) \, (1+r^{-1}) \, \{\|q^0\|_{W^{2-\frac{2}{p}}_p(\Omega; \, \mathbb{R}^{N-1})} +\|q\|_{W^{1,0}_p(Q_t)} \} +  \| \tilde{g}\|_{ L^p(Q_t)}\big]\\
 & \quad + \, \ell^*_t \, ([\varrho^*]_{C^{\beta,\frac{\beta}{2}}(Q_t)} + [q^*]_{C^{\beta,\frac{\beta}{2}}(Q_t)}) \, r^{\frac{\beta}{2}} \, \|D^2q\|_{L^p(Q_t)} \Big)\, .
  \end{align*}
Since $r$ is a free parameter and $t$ fixed, we can choose 
\begin{align*}
 r^{\frac{\beta}{2}} := \frac{1}{2} \, \min\{1, \frac{\lambda_{0}^{\frac{1}{2}}(t, \, R_q^*)\, \min\{1, \, \lambda_{0}(t, \, [R_q^*]^{-1}\widetilde{M}^*)\}}{\bar{C}_1 \, \max\{1, \, \lambda_{1}(t, \, [R_q^*]^{-1}\widetilde{M}^*)\}  \, \ell_t^* \, ([\varrho^*]_{C^{\beta}(Q_t)} + [q^*]_{C^{\beta}(Q_t)})}\}
\end{align*}
to obtain an estimate
\begin{align}\label{qcontrol}
& \frac{1}{2} \,  \|q\|_{W^{2,1}_p(Q_t; \, \mathbb{R}^{N-1})} \leq \bar{C}_1\, \frac{\max\{1, \, \lambda_{1}(t, \, [R_q^*]^{-1}\widetilde{M}^*)\}}{\lambda_{0}^{\frac{1}{2}}(t, \, R_q^*)\, \min\{1, \, \lambda_{0}(t, \, [R_q^*]^{-1}\widetilde{M}^*)\}}\, \times \nonumber\\
&  [(1+\lambda_{1}(t, \, [R_q^*]^{-1}\widetilde{M}^*)) \, (1+r^{-1}) \, \{\|q^0\|_{W^{2-\frac{2}{p}}_p(\Omega; \, \mathbb{R}^{N-1})} +\|q\|_{W^{1,0}_p(Q_t)} \} +  \| \tilde{g}\|_{ L^p(Q_t)} ] \, .
\end{align}
Due to the definition of $\tilde{g}$ in \eqref{Petrovski2}, and since $\nabla \widetilde{M}^* = \widetilde{M}_{\varrho}(\varrho^*, \, q^*) \, \nabla \varrho^* + \widetilde{M}_{q}(\varrho^*, \, q^*) \, \nabla q^*$ it follows that
\begin{align*}
 \|\tilde{g}\|_{L^p(Q_T)}  \leq & \frac{1}{\lambda_0(t, \, R_q^*)} \, (\|g\|_{L^p(Q_t)} + \|\nabla \widetilde{M}^* \cdot \nabla q\|_{L^p(Q_t)})\\
 \leq & \frac{1}{\lambda_0(t, \, R_q^*)} \, (\|g\|_{L^p(Q_t)} + L(t, \, \widetilde{M}^*)\,  [\|\nabla \varrho^* \cdot \nabla q\|_{L^p(Q_t)} + \|\nabla q^* \cdot \nabla q\|_{L^p(Q_t)}]) \, .
\end{align*}
We define $\phi^*_{0,t}$ and $\phi^*_{1,t}$ via \eqref{LinftyFACTORS}.  Due to Lemma \ref{rhonewlemma} and Lemma \ref{Mnewlemma} on the coefficients $R$ and $\widetilde{M}$, we see that $\phi^*_{0,t}$ and $\phi^*_{1,t}$ are bounded by a continuous function of $\|\varrho^*\|_{L^{\infty}(Q_t)}$, of $\|q^*\|_{L^{\infty}(Q_t)}$ and of $[m^*(t)]^{-1}$. Moreover $\phi^*_{0,t}$ and $\phi^*_{1,t}$ is determined only by the eigenvalues of $R_q^*$ and $\widetilde{M}^*$ and their Lipschitz constants over the range of $\varrho^*$, $q^*$. In order to obtain a control on $\sup_{s\leq t} \|q(s)\|_{W^{2-\frac{2}{p}}_p(\Omega; \, \mathbb{R}^{N-1})}$, we apply the inequality (3) of the paper \cite{solocompress} which yields
\begin{align*}
 \sup_{s\leq t} \|q(s)\|_{W^{2-\frac{2}{p}}_p(\Omega; \, \mathbb{R}^{N-1})} \leq \|q^0\|_{W^{2-\frac{2}{p}}_p(\Omega; \, \mathbb{R}^{N-1})} + C \, \|q\|_{W^{2,1}_p(Q_t; \, \mathbb{R}^{N-1})}
\end{align*}
for some constant $C$ independent on $t$, and combine it with \eqref{qcontrol}.
\begin{rem}\label{Tindependence}
 Consider the problem $\lambda \, \partial_t u - \triangle u = f$ in $Q_t$ with $u(0, \, x) = u_0(x)$ in $\Omega$ and $\nu(x)\cdot \nabla u = 0$ on $S_t$. Then $\|u\|_{W^{2,1}_p(Q_t)} \leq C_{1} \, \frac{\max\{1, \, \lambda\}}{\min\{1,\lambda\}}  \, (\|f\|_{L^p(Q_t)} + \|u_0\|_{W^{2-2/p}_p(\Omega)})$ with $C_{1}$ depending only on $\Omega$.
\end{rem}
\begin{proof}
We find an extension $\hat{u}_0 \in W^{2,1}_p(Q_{\infty})$ for $u_0$ such that $\|\hat{u}_0\|_{ W^{2,1}_p(Q_{\infty})} \leq c \, \|u_0\|_{W^{2-\frac{2}{p}}_p(\Omega)}$. 

We then look for the solution $v$ to $\lambda \, \partial_t v - \triangle v = (f - \lambda \, \partial_t \hat{u}_0 + \triangle \hat{u}_0) \, \chi_{[0,t]} =: g$ in $\Omega \times \mathbb{R}_+$ with $v(x, \, 0) = 0$ in $\Omega$ and $\nu(x)\cdot \nabla v = 0$ on $\partial \Omega \times \mathbb{R}_+$. In order to solve this problem, we scale time defining $\tilde{v}(s, \, x) := v(\lambda \, s, \, x)$. Clearly $\partial_s \tilde{v} - \triangle \tilde{v} = \tilde{g}$ in $\Omega \times \mathbb{R}_+$ with $\tilde{v}(x, \, 0) = 0$ in $\Omega$ and $\nu(x)\cdot \nabla \tilde{v} = 0$ on $\partial \Omega \times \mathbb{R}_+$. Thus $\|\tilde{v}\|_{W^{2,1}_p(\Omega \times \mathbb{R}_+)} \leq C_1 \, \|\tilde{g}\|_{L^p(\Omega \times \mathbb{R}_+)}$ with $C_1$ depending only on $\Omega$. We rescale time, to obtain
 \begin{align*}
  \lambda^{1+1/p} \, \|\partial_tv\|_{L^p} + \lambda^{1/p} \, \sum_{0 \leq |\alpha|\leq 2} \|D^{\alpha}_x v\|_{L^p} \leq C_1 \, \lambda^{1/p} \, \|g\|_{L^p(\Omega \times \mathbb{R}_+)}\, .
 \end{align*}
By the uniqueness theorem for the heat equation, we must have $u - \hat{u}_0 = v$ in $Q_t$. Since $g = 0$ on $]t, \, + \infty[$, it follows that
 \begin{align*}
\|u\|_{W^{2,1}_p(Q_t)}  \leq\|\hat{u}_0\|_{ W^{2,1}_p(Q_{t})} +  C_1 \, \frac{1}{\min\{\lambda, \, 1\}} \, ( \|f\|_{L^p(Q_t)} + \lambda \, \|\partial_t \hat{u}_0\|_{L^p(Q_t)} + \|\triangle \hat{u}_0\|_{L^p(Q_t)}) \, .
 \end{align*}
 The claim follows.
\end{proof}

\section{Auxiliary statements}

\begin{lemma}\label{ALGEBRA}
Suppose that $A, \, B \in \mathbb{R}^{N\times N}$ are two positive definite symmetric matrices. Then $A \, B$ possesses only real, strictly positive eigenvalues and
\begin{align*}
 \lambda_{\min}(A) \, \lambda_{\min}(B) \leq \lambda_{\min}(A \, B) \leq \lambda_{\max}(A \, B) \leq \lambda_{\max}(A) \, \lambda_{\max}(B) \
\end{align*}
Moreover, there are orthonormal vectors $\eta^1, \ldots, \eta^N \in \mathbb{R}^N$ such that the vectors $\xi^i := A^{\frac{1}{2}} \, \eta^i$ ($i = 1,\ldots,N$) define a basis of eigenvectors of $A \, B$ for $\mathbb{R}^N$.
\end{lemma}
\begin{proof}
 Define $C := A^{\frac{1}{2}} \, B \, A^{\frac{1}{2}}$. Since $A^{\frac{1}{2}}$ is symmetric, and moreover the matrix $B$ is positive, it follows that $C$ is symmetric and positive. Thus, since $A \, B = A^{\frac{1}{2}} \, C \, A^{-\frac{1}{2}}$, the eigenvalues of $A \, B$ are the ones of $C$. Choose $\eta^1, \ldots, \eta^N \in \mathbb{R}^N$ an orthonormal basis of eigenvectors for $C$. Then $\xi^i := A^{\frac{1}{2}} \, \eta^i$ is an eigenvector of $A \, B$.
\end{proof}
\begin{lemma}\label{HOELDERlemma}
For $0 \leq \beta < \min\{1, 2-\frac{5}{p}\}$ we define
\begin{align*}
\gamma := \begin{cases}
            \frac{1}{2} \, (2 - \frac{5}{p} -\beta) & \text{ for } 3<p<5\\
            (1-\beta) \,  \frac{p-1}{3+p} & \text{ for } 5 \leq p 
            \end{cases}
            \end{align*}
Then, there is $C = C(t)$ bounded on finite time intervals such that $C(0) = C_0$ depends only on $\Omega$ and for all $q^* \in W^{2,1}_{p}(Q_t)$
 \begin{align*}
 \|q^*\|_{C^{\beta,\frac{\beta}{2}}(Q_t)} \leq \|q^*(0)\|_{C^{\beta}(\Omega)} +  C(t) \,  t^{\gamma} \, [\|q^*\|_{W^{2,1}_{p}(Q_t)} + \|q^*\|_{C([0,t]; \, W^{2-\frac{2}{p}}_p(\Omega)})]  \, .
\end{align*}
\end{lemma}
\begin{proof}
For $r = \frac{3p}{(5-p)^+}$ and $\theta := \frac{3}{3+p-(5-p)^+}$ the Gagliardo-Nirenberg inequality yields
\begin{align*}
 \|u\|_{L^{\infty}(\Omega)} \leq C_1 \, \|\nabla u\|_{L^r(\Omega)}^{\theta} \, \|u\|_{L^p(\Omega)}^{1-\theta} + C_2 \, \|u\|_{L^p(\Omega)} \, .
\end{align*}
We apply this inequality to a difference $u = a(t_2) - a(t_1)$ for $ 0 < t_1 \leq t_2 \leq t$. By elementary means $a(t_2) - a(t_1) = \int_{t_1}^{t_2} a_t(s) \, ds$ and $\|a(t_2) - a(t_1)\|_{L^p(\Omega)} \leq (t_2-t_1)^{1-\frac{1}{p}} \, \|\partial_ta\|_{L^{p}(Q_t)}$. This yields
\begin{align*}
  \|a(t_2) - a(t_1)\|_{L^{\infty}(\Omega)} \leq & C_1 \, [2\|\nabla a\|_{L^{r,\infty}(Q_t)}]^{\theta} \, \|\partial_t a\|_{L^{p}(Q_t)}^{1-\theta} \, (t_2-t_1)^{(1-\theta) \, (1-\frac{1}{p})} \\
  & + C_2 \, (t_2-t_1)^{1-\frac{1}{p}} \, \|\partial_t a\|_{L^{p}(Q_t)} \, .
\end{align*}
We define $\delta := (1-\theta) \, (1-\frac{1}{p})$, and we make use of the continuity of $W^{2-\frac{2}{p}}(\Omega) \subset W^{1,r}(\Omega)$ and we see that
\begin{align}\label{differencetime}
\sup_{t_2\neq t_1}  \frac{\|a(t_2) - a(t_1)\|_{L^{\infty}(\Omega)}}{|t_2-t_1|^{\delta}} & \leq  \|\partial_t a\|_{L^{p}(Q_t)}^{1-\theta} \, (2C_1 \, C \, \|a\|_{C([0,t]; \, W^{2-\frac{2}{p}}(\Omega))}^{\theta} + C_2 \, t^{\theta \, (1-\frac{1}{p})} \, \|\partial_t a\|_{L^{p}(Q_t)}^{\theta}) \nonumber\\
& \leq C(t) \, (\|a\|_{W^{2,1}_{p}(Q_t)} +  \|a\|_{C([0,t]; \, W^{2-\frac{2}{p}}(\Omega))}) \, .
\end{align}
Now we consider a function $u = u(x, \, s)$ such that $u(x, \, 0) = 0$. Using \eqref{differencetime} and the embedding $W^{2-\frac{2}{p}}_p(\Omega) \subset W^{1,r}(\Omega) \subset C^{\alpha}(\Omega)$ valid for $\alpha := \min\{1, \, 2-\frac{5}{p}\}$
\begin{align*}
 \|u(s)\|_{C^0(\Omega)} = \|u(s)\|_{L^{\infty}(\Omega)} \leq C(t) \, (\|u\|_{W^{2,1}_{p}(Q_t)}+  \|u\|_{C([0,t]; \, W^{2-\frac{2}{p}}(\Omega))}) \, s^{\delta}\\
 \|u(s)\|_{C^{\alpha}(\Omega)} \leq C \, \|u(s)\|_{W^{1,r}(\Omega)} \leq C \,  \|u\|_{C([0,s]; \, W^{2-\frac{2}{p}}_{p}(\Omega))} \, .
\end{align*}
Introduce $\alpha := \min\{1, \, 2-\frac{5}{p}\}$. First making use of interpolation inequalities (\cite{lunardi}, Example 1.25 with Corollary 1.24) and find for all $0 \leq \beta \leq \alpha \leq 1$ and $u \in C^1(\Omega)$
\begin{align}\label{interpo}
 \|u\|_{C^{\beta}(\Omega)} \leq c \,  \|u\|_{C^{\alpha}(\Omega)}^{\frac{\beta}{\alpha}} \, \|u\|_{C^0(\Omega)}^{1-\frac{\beta}{\alpha}} \, .
\end{align}
Thus, for $b := (1-\frac{\beta}{\alpha}) \, \delta$, it follows from \eqref{differencetime} and \eqref{interpo} that for all $s \leq t$
\begin{align*}
 \|u(s)\|_{C^{\beta}(\Omega)} \leq C \, C(t)^{1-\frac{\beta}{\alpha}} \, (\|u\|_{W^{2,1}_{p}(Q_t)}+  \|u\|_{C([0,t]; \, W^{2-\frac{2}{p}}(\Omega))}) \, s^{b} \, .
\end{align*}
For a function $q^* \in W^{2,1}_{p}(Q_t)$, this induces for all $\beta < \alpha$ a bound
\begin{align*}
\sup_{s\leq t} \|q^*(s)\|_{C^{\beta}(\Omega)} \leq & \|q^*(0)\|_{C^{\beta}(\Omega)} + \sup_{s\leq t} \|q^*(s)-q^*(0)\|_{C^{\beta}(\Omega)}\\
\leq & \|q^*(0)\|_{C^{\beta}(\Omega)} + C(t) \, t^{b} \, (\|q^*\|_{W^{2,1}_{p}(Q_t)} +  \|q^*\|_{C([0,t]; \, W^{2-\frac{2}{p}}(\Omega))}) \, .
\end{align*}
Moreover, we observe that $\beta < \alpha =  \min\{1, \, 2-\frac{5}{p}\}$ always implies that $\beta/2 < \delta$. Thus, invoking \eqref{differencetime} again
\begin{align*}
 \sup_{x\in \Omega} [q^*(x)]_{C^{\beta/2}([0,t])} = & \sup_{x\in \Omega} \sup_{t_2\neq t_1} \frac{|q^*(x, \, t_2) - q^*(x, \, t_1)|}{|t_2-t_1|^{\beta/2}}\\
 \leq & C(t) \, (\|q^*\|_{W^{2,1}_{p}(Q_t)} +\|q^*\|_{C([0,t]; \, W^{2-\frac{2}{p}}(\Omega))})   \, t^{\delta-\beta/2}  \, .
\end{align*}
Thus for all $\beta < \alpha$ we get
\begin{align*}
 \|q^*\|_{C^{\beta, \, \frac{\beta}{2}}(Q_t)} \leq \|q^*(0)\|_{C^{\beta}(\Omega)} + C(t) \, t^{\gamma} \, (\|q^*\|_{W^{2,1}_{p}(Q_t)} +\|q^*\|_{C([0,t]; \, W^{2-\frac{2}{p}}(\Omega))}) \, . 
\end{align*}
with $\gamma$ being the minimum of $\delta - \beta/2$ and $b = (1-\frac{\beta}{\alpha}) \, \delta$. The computation of the exponent is straightforward. 
\end{proof}
We now prove some properties of the lower--order operators defined in \eqref{A1right}, \eqref{A3right}. Consider first $g = g(x, \, t, \, q, \,\varrho, \, v, \, \nabla q,  \, \nabla v,\, \nabla \varrho)$ with $g$ defined by \eqref{A1right}. The vector field $g$ is defined on the compound $\Gamma := Q \times \mathbb{R}^{N-1} \times \mathbb{R}_+ \times \mathbb{R}^3 \times \mathbb{R}^{N-1\times 3} \times \mathbb{R}^{3\times 3} \times \mathbb{R}^{3}$ and assumes values in $\mathbb{R}^{N-1}$. As a vector field, $g$ belongs to $C^1(\Gamma; \, \mathbb{R}^{N-1})$, because the maps $R$ and $\widetilde{M}$ are of class $C^2$ (Lemma \ref{rhonewlemma} and Lemma \ref{Mnewlemma}). The derivatives possess the following expressions
 \begin{align*}
 g_{\varrho} = & R_{\varrho,\varrho} \, \varrho \, \divv v - R_{\varrho,q} \, v \cdot \nabla q - \widetilde{M}_{\varrho,\varrho} \, \nabla \varrho \, \tilde{b}- \widetilde{M}_{\varrho,q}  \, \nabla q \, \tilde{b} - \widetilde{M}_{\varrho} \, \divv \tilde{b} + \tilde{r}_{\varrho} \, , \\
 g_q = & (R_{\varrho,q} \, \varrho - R_q) \, \divv v - R_{q,q} \, v \cdot \nabla q - \widetilde{M}_{\varrho,q} \, \nabla \varrho \, \tilde{b}
 - \widetilde{M}_{q,q}  \, \nabla q \, \tilde{b} - \widetilde{M}_{q} \, \divv \tilde{b} + \tilde{r}_{q}\, , \\
 g_v = & - R_q \, \nabla q\, , \quad  g_{\varrho_x} = - \widetilde{M}_{\varrho} \, \tilde{b}\, , \quad g_{q_x} = - R_q \, v  - \widetilde{M}_{q} \, \tilde{b}\, , \quad g_{v_x} =  (R_{\varrho} \, \varrho - R) \, \text{I}_{3\times 3}\, ,  
\end{align*} 
in which all non-linear functions $R, \, \widetilde{M}$, $\tilde{r}$ and their derivatives are evaluated at $\varrho, \, q$.
 
We next want to study the Nemicki operator $(q, \,\varrho, \, v) \mapsto g(x, \, t, \, q, \,\varrho, \, v, \, \nabla q, \, \nabla \varrho, \, \nabla v)$ on $\mathcal{X}_{T,+}$. A boundedness estimate can be obtained for this operator from $\mathcal{X}_{T,+}$ into $L^p(Q_T; \, \mathbb{R}^{N-1})$ via the Lemma \ref{RightHandControl} (this was applied for instance in the proof of Proposition \ref{estimateself}). We can apply the same tool to the derivatives. Choosing $G = g_{\varrho}$ in Lemma \ref{RightHandControl} with  $r_1 = 1$, $\bar{G}(x, \,t) = |\tilde{b}_x(x, \, t)|$ and $\bar{H}(x, \,t) := |\tilde{b}(x, \, t)|$, we obtain a boundedness estimate for $g_{\varrho}$ as operator between $\mathcal{X}_{T,+}$ and $ L^p(Q_T; \, \mathbb{R}^{N-1})$. With obvious choices, we treat the derivatives $g_q$ and $g_v$ in the same way. Due to the simpler expressions, we obtain for the other derivatives continuity estimates:
\begin{align*}
\|g_{q_x}\|_{L^{\infty,p}(Q_T)}\leq & c_1((m(T))^{-1}, \, M(T), \, \|q\|_{L^{\infty}(Q_T)}) \, (\|v\|_{L^{\infty,p}(Q_T)} + \|\tilde{b}\|_{L^{\infty,p}(Q_T)}) \, ,\\  
\|g_{\varrho_x}\|_{L^{\infty,p}(Q_T)}\leq & c_1((m(T))^{-1}, \, M(T), \, \|q\|_{L^{\infty}(Q_T)}) \,   \|\tilde{b}\|_{L^{\infty,p}(Q_T)} \, ,\\
\|g_{v_x}\|_{L^{\infty}(Q_T)} \leq & c_1((m(T))^{-1}, \, M(T), \, \|q\|_{L^{\infty}(Q_T)}) \, .
\end{align*}
We also remark that if $u, \, u^*$ are two points in $\mathcal{X}_{T,+}$ and we expand $g(x, \, t, \, u, \, D_xu) = g(x, \, t, \, u^*, \, D_xu^*) + g^{\prime}(u, \, u^*) \, (u-u^*)$ (cf. \eqref{Arightlinear}), then the operators $g^{\prime}(u, \, u^*) = \int_{0}^1 g^{\prime}(x, \, t, \theta \, u + (1-\theta) \, u^*, \, \theta \, D_xu + (1-\theta) \, D_xu^*) \, d\theta$ satisfy similar estimates. 

We consider next $f = f(x, \, t, \, q, \,\varrho, \, v, \, \nabla q,  \, \nabla v,\, \nabla \varrho)$ with $f$ defined by \eqref{A3right}; $f$ belongs to $C^1(\Gamma; \, \mathbb{R}^{3})$. The derivatives possess the following expressions
 \begin{align*}
 f_{\varrho} = & - P_{\varrho,\varrho} \, \nabla \varrho - P_{\varrho,q} \, \nabla q - (v\cdot \nabla) v + R_{\varrho} \, \tilde{b} + \bar{b} \, , \\
 f_q = & -P_{\varrho,q} \, \nabla \varrho - P_{q,q} \, \nabla q + R_q \, \tilde{b}\, , \\
 f_v = & - \varrho \, \nabla v\, , \quad f_{\varrho_x} = - P_{\varrho}\, , \quad f_{q_x} =  - P_q \, , \quad  f_{v_x} = - \varrho \, v \, . 
\end{align*} 
We discuss these derivatives as Nemicki operators on $\mathcal{X}_{T,+}$ with similar arguments as in the case of $g$. We resume our conclusions in the following Lemma.
\begin{lemma}\label{gedifferential}
Adopt the assumptions of Theorem \ref{MAIN}. The maps $g$ and $f$ are defined on $\mathcal{X}_{T,+}$ by the expressions \eqref{A1right} and \eqref{A3right}. Then, $g$ and $f$ are continuously differentiable at every $u^* = (q^*, \, \varrho^*, \, v^*) \in \mathcal{X}_{T,+}$. For each $u =(q, \, \varrho, \, v) \in \mathcal{X}_{T,+}$ the derivatives satisfy
\begin{align*}
 & \|g_{q}(u, \, u^*)\|_{L^p(Q_T)} + \|g_{\varrho}(u, \, u^*)\|_{L^p(Q_T)} + \|g_v(u, \, u^*) \|_{L^p(Q_T)}\\
 & + \| g_{q_x}(u, \, u^*)\|_{L^{\infty,p}(Q_T)} + \|g_{v_x}(u, \, u^*)\|_{L^{\infty}(Q_T)} + \|g_{\varrho_x}(u, \, u^*)\|_{L^{\infty,p}(Q_T)} \\
 & \qquad \leq c_1(\|u\|_{\mathcal{X}_T} + \|u^*\|_{\mathcal{X}_T} + [\inf_{Q_T} \inf\{ \varrho, \, \varrho^*\}]^{-1} + \|\tilde{b}\|_{W^{1,0}_p(Q_T)}) \, ,\\
 & \|f_{q}(u, \, u^*)\|_{L^p(Q_T)} + \|f_{\varrho}(u, \, u^*)\|_{L^p(Q_T)} + \|f_v(u, \, u^*) \|_{L^p(Q_T)}\\
 & + \| f_{q_x}(u, \, u^*)\|_{L^{\infty}(Q_T)} + \|f_{v_x}(u, \, u^*)\|_{L^{\infty}(Q_T)} + \|f_{\varrho_x}(u, \, u^*)\|_{L^{\infty}(Q_T)} \\
 & \qquad \leq c_1(\|u\|_{\mathcal{X}_T} + \|u^*\|_{\mathcal{X}_T} + [\inf_{Q_T} \inf\{ \varrho, \, \varrho^*\}]^{-1} + \|\tilde{b}\|_{L^p(Q_T)} + \|\bar{b}\|_{L^p(Q_T)}) \, ,
 \end{align*}
 with a continuous function $c_1$ increasing of its argument.
 \end{lemma}
 We can extend this statement to the maps $g^1$ and $f^1$ introduced in the system \eqref{equationdiff1},  \eqref{equationdiff3}. Recall first that
 \begin{align}\label{g1}
 g^1 = g -  R_q \, \partial_t \hat{q}^0 + \widetilde{M} \, \triangle \hat{q}^0 - \widetilde{M}_{\varrho} \, \nabla \varrho \cdot \nabla \hat{q}^0 \, ,
 \end{align}
in which $\hat{q}^0 \in W^{2,1}_p(Q_T; \, \mathbb{R}^{N-1})$ is a given vector field. The boundedness of $g^1$ on $\mathcal{X}_{T,+}$ into $L^p(Q_T)$ is then readily verified (Lemma \ref{RightHandControl}). The derivatives satisfy
 \begin{align*}
 g^1_{\varrho} = & g_{\varrho} - R_{\varrho,q} \, \partial_t \hat{q}^0+ \widetilde{M}_{\varrho} \, \triangle \hat{q}^0 - \widetilde{M}_{\varrho,\varrho} \nabla \varrho \nabla \hat{q}^0 - \widetilde{M}_{\varrho,q} \, \nabla q \,  \nabla \hat{q}^0 \, , \\
 g_q^1 = & g_q - R_{q,q} \, \partial_t \hat{q}^0+ \widetilde{M}_{q} \, \triangle \hat{q}^0 - \widetilde{M}_{\varrho,q} \nabla \varrho \nabla \hat{q}^0 - \widetilde{M}_{q,q} \, \nabla q \,  \nabla \hat{q}^0  \, , \\
  g_{\varrho_x}^1 = &  g_{\varrho_x} - \widetilde{M}_{\varrho} \, \nabla \hat{q}^0\, , \quad g_{q_x}^1 =  g_{q_x}  - \widetilde{M}_{q} \, \nabla \hat{q}^0 \, ,  
\end{align*} 
and $g_v^1 = g_v$ and $g^1_{v_x} = g_{v_x}$. These expressions can be estimated as in the case of $g$ (replace $\tilde{b}$ in these estimates by $\nabla \hat{q}^0$). Since $f^1 = f- \varrho \partial_t \hat{v}^0 + \divv \mathbb{S}(\nabla \hat{v}^0)$, only the derivative $f^1_{\varrho} = f_{\varrho} - \partial_t \hat{v}^0$ gets a new contribution that is easily estimated.
\begin{lemma}\label{gedifferentialun}
Adopt the assumptions of Theorem \ref{MAIN}. The maps $g$ and $f$ are defined on $\mathcal{X}_{T,+}$ by the expressions \eqref{A1right} and \eqref{A3right} and $g^1$ is defined via \eqref{g1} and $f^1 = f -  \varrho \partial_t \hat{v}^0 + \divv \mathbb{S}(\nabla \hat{v}^0)$, in which $(\hat{q}^0, \, \hat{v}^0)$ is a given pair in $\mathcal{Y}_T$. Then, $g^1$ and $f^1$ are continuously differentiable at every $u^* = (q^*, \, \varrho^*, \, v^*) \in \mathcal{X}_{T,+}$. For each $u =(q, \, \varrho, \, v) \in \mathcal{X}_{T,+}$ the derivatives satisfy
\begin{align*}
 & \|g^1_{q}(u, \, u^*)\|_{L^p(Q_T)} + \|g^1_{\varrho}(u, \, u^*)\|_{L^p(Q_T)} + \|g^1_v(u, \, u^*) \|_{L^p(Q_T)}\\
 & + \| g^1_{q_x}(u, \, u^*)\|_{L^{\infty,p}(Q_T)} + \|g^1_{v_x}(u, \, u^*)\|_{L^{\infty}(Q_T)} + \|g^1_{\varrho_x}(u, \, u^*)\|_{L^{\infty,p}(Q_T)} \\
 & \qquad \leq c_1(\|u\|_{\mathcal{X}_T} + \|u^*\|_{\mathcal{X}_T} + [\inf_{Q_T} \inf\{ \varrho, \, \varrho^*\}]^{-1} + \|\tilde{b}\|_{W^{1,0}_p(Q_T)}+\|\hat{q}^0\|_{W^{2,1}_p(Q_T)}) \, ,\\
 & \|f_{q}^1(u, \, u^*)\|_{L^p(Q_T)} + \|f_{\varrho}^1(u, \, u^*)\|_{L^p(Q_T)} + \|f^1_v(u, \, u^*) \|_{L^p(Q_T)}\\
 & + \| f^1_{q_x}(u, \, u^*)\|_{L^{\infty}(Q_T)} + \|f^1_{v_x}(u, \, u^*)\|_{L^{\infty}(Q_T)} + \|f^1_{\varrho_x}(u, \, u^*)\|_{L^{\infty}(Q_T)} \\
 & \qquad \leq c_1(\|u\|_{\mathcal{X}_T} + \|u^*\|_{\mathcal{X}_T} + [\inf_{Q_T} \inf\{ \varrho, \, \varrho^*\}]^{-1} + \|\tilde{b}\|_{L^p(Q_T)} + \|\bar{b}\|_{L^p(Q_T)} + \|\hat{v}^0\|_{W^{2,1}_p(Q_T)}) \, ,
 \end{align*}
 with a continuous function $c_1$ increasing of its argument.
 \end{lemma}

\end{document}